\documentclass[12pt,pdftex,a4paper]{amsart}

\selectfont

\usepackage{simplewick}
\usepackage{qtree}
\usepackage{forest}

\usepackage{caption}

\usepackage{etoolbox}
\usepackage{epic,eepic,ecltree}
\usepackage{graphicx}
\usepackage{tikz}
\usetikzlibrary{decorations.text}
\usepackage{amssymb,color, euscript, enumerate}
\usepackage{amsthm}
\usepackage{amsmath}
\usepackage{braket}
\usepackage{amscd}
\usepackage{txfonts}
\usepackage{comment}
\usepackage{bm}
\usepackage{amscd}
\numberwithin{equation}{section}

\newcommand{\introitem}[1]{\smallskip\noindent\textbf{#1.}\ }

\newcommand{\introheading}[1]{%
  \par\medskip
  \noindent{\bfseries\boldmath #1.}\quad
}

\makeatletter
\let\old@tocline\@tocline
\let\section@tocline\@tocline
\newcommand{\subsection@dotsep}{4.5}
\newcommand{\subsubsection@dotsep}{4.5}
\patchcmd{\@tocline}
  {\hfil}
  {\nobreak
     \leaders\hbox{$\m@th
        \mkern \subsection@dotsep mu\hbox{.}\mkern \subsection@dotsep mu$}\hfill
     \nobreak}{}{}
\let\subsection@tocline\@tocline
\let\@tocline\old@tocline

\patchcmd{\@tocline}
  {\hfil}
  {\nobreak
     \leaders\hbox{$\m@th
        \mkern \subsubsection@dotsep mu\hbox{.}\mkern \subsubsection@dotsep mu$}\hfill
     \nobreak}{}{}
\let\subsubsection@tocline\@tocline
\let\@tocline\old@tocline

\let\old@l@subsection\l@subsection
\let\old@l@subsubsection\l@subsubsection

\def\@tocwriteb#1#2#3{%
  \begingroup
    \@xp\def\csname #2@tocline\endcsname##1##2##3##4##5##6{%
      \ifnum##1>\c@tocdepth
      \else \sbox\z@{##5\let\indentlabel\@tochangmeasure##6}\fi}%
    \csname l@#2\endcsname{#1{\csname#2name\endcsname}{\@secnumber}{}}%
  \endgroup
  \addcontentsline{toc}{#2}%
    {\protect#1{\csname#2name\endcsname}{\@secnumber}{#3}}}%

\newlength{\@tocsectionindent}
\newlength{\@tocsubsectionindent}
\newlength{\@tocsubsubsectionindent}
\newlength{\@tocsectionnumwidth}
\newlength{\@tocsubsectionnumwidth}
\newlength{\@tocsubsubsectionnumwidth}
\newcommand{\settocsectionnumwidth}[1]{\setlength{\@tocsectionnumwidth}{#1}}
\newcommand{\settocsubsectionnumwidth}[1]{\setlength{\@tocsubsectionnumwidth}{#1}}
\newcommand{\settocsubsubsectionnumwidth}[1]{\setlength{\@tocsubsubsectionnumwidth}{#1}}
\newcommand{\settocsectionindent}[1]{\setlength{\@tocsectionindent}{#1}}
\newcommand{\settocsubsectionindent}[1]{\setlength{\@tocsubsectionindent}{#1}}
\newcommand{\settocsubsubsectionindent}[1]{\setlength{\@tocsubsubsectionindent}{#1}}

\renewcommand{\l@section}{\section@tocline{1}{\@tocsectionvskip}{\@tocsectionindent}{}{\@tocsectionformat}}%
\renewcommand{\l@subsection}{\subsection@tocline{2}{\@tocsubsectionvskip}{\@tocsubsectionindent}{}{\@tocsubsectionformat}}%
\renewcommand{\l@subsubsection}{\subsubsection@tocline{3}{\@tocsubsubsectionvskip}{\@tocsubsubsectionindent}{}{\@tocsubsubsectionformat}}%
\newcommand{\@tocsectionformat}{}
\newcommand{\@tocsubsectionformat}{}
\newcommand{\@tocsubsubsectionformat}{}
\expandafter\def\csname toc@1format\endcsname{\@tocsectionformat}
\expandafter\def\csname toc@2format\endcsname{\@tocsubsectionformat}
\expandafter\def\csname toc@3format\endcsname{\@tocsubsubsectionformat}
\newcommand{\settocsectionformat}[1]{\renewcommand{\@tocsectionformat}{#1}}
\newcommand{\settocsubsectionformat}[1]{\renewcommand{\@tocsubsectionformat}{#1}}
\newcommand{\settocsubsubsectionformat}[1]{\renewcommand{\@tocsubsubsectionformat}{#1}}
\newlength{\@tocsectionvskip}
\newcommand{\settocsectionvskip}[1]{\setlength{\@tocsectionvskip}{#1}}
\newlength{\@tocsubsectionvskip}
\newcommand{\settocsubsectionvskip}[1]{\setlength{\@tocsubsectionvskip}{#1}}
\newlength{\@tocsubsubsectionvskip}
\newcommand{\settocsubsubsectionvskip}[1]{\setlength{\@tocsubsubsectionvskip}{#1}}

\patchcmd{\tocsection}{\indentlabel}{\makebox[\@tocsectionnumwidth][l]}{}{}
\patchcmd{\tocsubsection}{\indentlabel}{\makebox[\@tocsubsectionnumwidth][l]}{}{}
\patchcmd{\tocsubsubsection}{\indentlabel}{\makebox[\@tocsubsubsectionnumwidth][l]}{}{}

\newcommand{\@sectypepnumformat}{}
\renewcommand{\contentsline}[1]{%
  \expandafter\let\expandafter\@sectypepnumformat\csname @toc#1pnumformat\endcsname%
  \csname l@#1\endcsname}
\newcommand{\@tocsectionpnumformat}{}
\newcommand{\@tocsubsectionpnumformat}{}
\newcommand{\@tocsubsubsectionpnumformat}{}
\newcommand{\setsectionpnumformat}[1]{\renewcommand{\@tocsectionpnumformat}{#1}}
\newcommand{\setsubsectionpnumformat}[1]{\renewcommand{\@tocsubsectionpnumformat}{#1}}
\newcommand{\setsubsubsectionpnumformat}[1]{\renewcommand{\@tocsubsubsectionpnumformat}{#1}}
\renewcommand{\@tocpagenum}[1]{%
  \hfill {\mdseries\@sectypepnumformat #1}}

\let\oldappendix\appendix
\renewcommand{\appendix}{%
  \leavevmode\oldappendix%
  \addtocontents{toc}{%
    \protect\settowidth{\protect\@tocsectionnumwidth}{\protect\@tocsectionformat\sectionname\space}%
    \protect\addtolength{\protect\@tocsectionnumwidth}{2em}}%
}
\makeatother

\makeatletter
\settocsectionnumwidth{2em}
\settocsubsectionnumwidth{2.5em}
\settocsubsubsectionnumwidth{3em}
\settocsectionindent{1pc}%
\settocsubsectionindent{\dimexpr\@tocsectionindent+\@tocsectionnumwidth}%
\settocsubsubsectionindent{\dimexpr\@tocsubsectionindent+\@tocsubsectionnumwidth}%
\makeatother

\settocsectionvskip{10pt}
\settocsubsectionvskip{0pt}
\settocsubsubsectionvskip{0pt}
    
\settocsectionformat{\bfseries}
\settocsubsectionformat{\mdseries}
\settocsubsubsectionformat{\mdseries}
\setsectionpnumformat{\bfseries}
\setsubsectionpnumformat{\mdseries}
\setsubsubsectionpnumformat{\mdseries}

\let\oldtableofcontents\tableofcontents
\renewcommand{\tableofcontents}{%
  \vspace*{-\linespacing}% Default gap to top of CONTENTS is \linespacing.
  \oldtableofcontents}

\newcommand{\CP}{\C P^1}

\newcommand{\F}{{\overline{H}}}

\newcommand{\ad}{\mathrm{ad}}

\newcommand{\ar}{{a_{[n]}}}
\newcommand{\zr}{{z_{[n]}}}

\newcommand{\twi}{{\mathrm{twist}}}

\newcommand{\Hom}{{\mathrm{Hom}}}

\newcommand{\an}{{\mathrm{an}}}

\newcommand{\df}{{\ket{\frac{2-d}{2}}}}
\newcommand{\dn}{{\frac{d-2}{2}}}

\newcommand{\da}{\dagger}

\newcommand{\N}{\mathbb{N}}
\newcommand{\Z}{\mathbb{Z}}
\newcommand{\R}{\mathbb{R}}
\newcommand{\C}{\mathbb{C}}

\newcommand{\pr}{{\mathrm{pr}}}

\newcommand{\Xr}{{X_r(\C)}}
\newcommand{\Xrd}{{X_r \times X_r}}

\newcommand{\Or}{\mathrm{O}_{\Xr}}

\newcommand{\Conf}{\mathrm{Conf}}

\newcommand{\Mr}{{M_{[0;r]}}}

\newcommand{\xr}{{x_{[n]}}}

\newcommand{\Dr}{{\mathrm{D}_{\Xrd}}}

\newcommand{\sod}{\mathfrak{so}(d+1,1)}

\newcommand{\std}{{\text{std}}}

\newcommand{\va}{\bm{1}}
\newcommand{\1}{\bm{1}}

\newcommand{\id}{{\mathrm{id}}}

\newcommand{\z}{{\bar{z}}}

\newcommand{\SO}{\mathrm{SO}}

\newcommand{\h}{{\bar{h}}}

\newcommand{\uz}{\underline{z}}

\newcommand{\m}{{\mathbb{Y}}}

\newcommand{\pa}{{\partial}}

\newcommand{\al}{\alpha}
\newcommand{\Harm}{\mathrm{Harm}}
\newcommand{\ep}{\epsilon}
\newcommand{\be}{\beta}
\newcommand{\ga}{\gamma}
\newcommand{\ze}{\zeta}
\newcommand{\de}{\delta}
\newcommand{\D}{\Delta}
\newcommand{\om}{\omega}

\newcommand{\si}{\sigma}

\newcommand{\ft}{\frac{1}{2}}

\newcommand{\CB}{{\mathcal{CB}}}

\newcommand{\Ld}{{\overline{L}}}

\newcommand{\Ad}{{\hat{\mathfrak{A}}_{d,\al}}}
\newcommand{\Hd}{H_{d,\al}}
\newcommand{\Sym}{\mathrm{Sym}}

\def\todo{{\textbf{TODO}}}

\newcommand{\Hf}{{\tilde{H}}}

\newcommand{\End}{\mathrm{End}}

\newcommand{\cut}{\mathrm{cut}}

\newcommand{\CE}{\mathbb{CE}}

\newcommand{\bD}{\overline{D}}

\newcommand{\Bord}{\mathrm{Bord}_d^{\mathrm{CF}}}

\newcommand{\so}{{\mathfrak{so}}}

\newtheorem{thm}{Theorem}[section]
\newtheorem{dfn}[thm]{Definition}
\newtheorem{lem}[thm]{Lemma}
\newtheorem{prop}[thm]{Proposition}
\newtheorem{cor}[thm]{Corollary}
\newtheorem{rem}[thm]{Remark}

\begin{document}

\begin{center}
{{\LARGE \bf Higher-dimensional full vertex algebras
and systems of conformally covariant correlation functions
}
} \par \bigskip

\renewcommand*{\thefootnote}{\fnsymbol{footnote}}
{\normalsize
Yuto Moriwaki \footnote{email: \texttt{moriwaki.yuto (at) gmail.com}}
}
\par \bigskip
{\footnotesize Interdisciplinary Theoretical and Mathematical Science Program (iTHEMS)\\
Wako, Saitama 351-0198, Japan}

\par \bigskip
\end{center}

\vspace{5mm}
\noindent

\begin{center}
\textbf{\large Abstract}
\end{center}
We formulate a class of systems of Euclidean correlation functions for compact $d$-dimensional conformal field theories which are not necessarily
unitary. We also introduce conformal $d$-vertex algebras, $d$-dimensional analogues of full vertex algebras, and prove that such an algebra is obtained from every system of conformally covariant correlation functions.  Finally, for each $d\geq 2$, we construct a family of examples satisfying these axioms, corresponding to generalized free scalar fields.
For $d>2$, this family contains both unitary and non-unitary examples.

\vspace{3mm}

\tableofcontents

\begin{center}
\textbf{\large Introduction}
\end{center}

One formulation of quantum field theory is in terms of correlation functions.  The Osterwalder--Schrader axioms formulate Euclidean quantum
field theory as a sequence of tempered distributions satisfying Euclidean covariance, permutation invariance, reflection positivity, the cluster property, and the linear growth condition \cite{OS1,OS2}. In dimension two, a conformal version of these axioms has recently been established for a class of unitary {full vertex operator algebras} in \cite{AMT}.

The purpose of the present paper is to give a related, but different, formulation of conformal field theory in dimension \(d\geq 2\) in terms of correlation functions.  In the Osterwalder--Schrader axioms, reflection positivity plays an essential role.  In contrast, we do not impose reflection positivity; instead, we impose a {\it strong cluster decomposition} for systems of correlation functions, corresponding to the {\it operator product expansion} (OPE). 
Thus the present formulation also includes non-unitary conformal field theories.

%The consistency and associativity of OPEs are central to the {\it conformal bootstrap} in physics \cite{Polyakov,BPZ,RRTV}.
% Hollands \cite{HollandsOPE} proposed an axiomatic formulation of quantum field theory in terms of OPE coefficients satisfying associativity conditions.
%The consistency and associativity of the OPE are central to the {\it conformal bootstrap} \cite{Polyakov,FGG,BPZ,RRTV}. An axiomatic approach to quantum field theory formulated in terms of the OPE and its associativity was also proposed by Hollands \cite{HollandsOPE}.

%The idea of regarding consistency and associativity of the OPE as fundamental structures is central to the {\it conformal bootstrap} in physics \cite{Polyakov,BPZ,RRTV} (see also the axiomatic approach to quantum field theory based on OPEs proposed by Hollands \cite{HollandsOPE}).

The systems of correlation functions introduced below are higher-dimensional analogues, with some modifications, of the \emph{full field algebras} of Huang and Kong in dimension two \cite{HK}.  As in the work of Huang and Kong, we extract formal vertex operators, or operator product expansions, from such systems.  
Products of these \(d\)-dimensional vertex operators are obtained as expansions, in different regions, of a single real analytic function,
which corresponds to the associativity and commutativity of vertex algebras \cite{FLM,LL,FB}.
This leads to the notion of a \emph{conformal \(d\)-vertex algebra}, which is a higher-dimensional generalization of \emph{chiral and full vertex algebras} \cite{M1}.
\footnote{
Higher-dimensional vertex algebras have previously been studied in \cite{Nikolov,BN}. Nikolov's formulation is based on an algebraic locality
condition, which is a higher-dimensional analogue of the locality axiom for chiral vertex algebras. Such a condition is too restrictive
for the general class of non-chiral two-dimensional and higher-dimensional CFTs. Our formulation is instead intended as a higher-dimensional extension of the notions of full field algebras and full vertex algebras for non-chiral two-dimensional CFT.
%Such a condition is too restrictive for the class of conformal field theories considered here; in particular, general non-chiral two-dimensional CFTs and higher-dimensional CFTs need not satisfy this type of algebraic locality. Our formulation should therefore be viewed not as a higher-dimensional extension of chiral vertex algebras, but rather as a higher-dimensional extension of the full field algebra and full vertex algebra, which formulates non-chiral two-dimensional CFT.
}

We first recall the corresponding picture for vertex operator algebras.
Let \(V\) be a vertex operator algebra and \(V^\vee=\bigoplus_{h\in\Z} V_h^*\) the restricted dual. 
For $u \in V^\vee$ and $a_1,\dots,a_n \in V$, the formal power series
\begin{align*}
\langle u,Y(a_1,z_1)\cdots Y(a_n,z_n)\va\rangle
\end{align*}
converges in the region $|z_1|>\cdots>|z_n|$ and analytically continues to a rational function on the configuration spaces \(\Conf_n(\C)=\{(z_1,\dots,z_n)\in\C^n \mid z_i\neq z_j\}\) \cite[Proposition 3.5.1]{FHL}.
Equivalently, one obtains a sequence of maps
%解析接続により $\Conf_n(\C)$上の有理多項式関数に解析接続される。すなわち
\begin{align}
S_n:V^\vee \otimes V^{\otimes n} \rightarrow \C[z_i, (z_i-z_j)^{-1} \mid i\neq j].\label{intro_chiral_cor}
\end{align}
Conversely, one can see that such a system of correlation functions determines the vertex algebra (see Theorem \ref{thm_from} and Proposition \ref{prop_vertex_algebra}).
%を得る。逆に頂点作用素代数はこうした相関関数系から復元することができる。
Thus, in the chiral case, the vertex algebra and the system of correlation functions can be viewed as two descriptions of essentially the same algebraic data. One of the aims of the present paper is to formulate an analogous relation in a non-holomorphic higher-dimensional setting.

In non-chiral two-dimensional conformal field theory, and also in
dimensions \(d\geq 3\), correlation functions are generally not
holomorphic; rather, they are real analytic functions on configuration
spaces \cite{LM,HK,M8}.
In this paper we formulate a class of conformally covariant systems of
real analytic correlation functions in dimension \(d\geq 2\), and show how
formal vertex operators, or operator product expansions, are obtained from
such systems.
We also construct, for every \(d\geq2\), a family of examples associated
with generalized free scalar fields. For \(d>2\), this family contains both
unitary and non-unitary examples.

Before describing the formulation and the main results, we spell out a
basic dictionary for readers familiar with vertex algebras.

\begin{table}[h]
\centering
\renewcommand{\arraystretch}{1.25}
\begin{tabular}{c|c|c|c}
\hline
 & chiral \(d=2\)
 & full \(d=2\)
 & \(d\)-dimensional theory \\ \hline
global symmetry
 & \(\mathfrak{sl}_2\subset\mathrm{Vir}\)
 & \(\mathfrak{sl}_2\oplus\mathfrak{sl}_2\)
 & \(\sod\) \\
grading
 & \(L(0)\)
 & \(D=L(0)+\bar L(0)\)
 & \(D\) \\
state space
 & \( \bigoplus_{h\in\Z} H_h\)
 & \( \bigoplus_{\substack{h,\bar h\in\R\\h-\h\in\Z}}H_{h,\bar h}\)
 & \( \bigoplus_{\Delta\in\R}H_\Delta\) \\
%coordinates
% & \(z\)
% & \((z,\bar z)\)
% & \(x=(x_1,\ldots,x_d)\) \\
formal series
 & \(H((z))\)
 & \(H((z,\bar z,|z|^\R))\)
 & \(H((x_1,\ldots,x_d,|x|^\R))\) \\
correlation functions
 & holomorphic
 & real analytic on \(\Conf_n(\R^2)\)
 & real analytic on \(\Conf_n(\R^d)\) \\
\hline
\end{tabular}
%\caption{A dictionary between vertex algebras, two-dimensional full
%conformal field theory, and the \(d\)-dimensional formalism of this paper.}
\label{tab:dictionary}
\end{table}

\introitem{Symmetry}
The Virasoro algebra, and in particular its subalgebra
\(\mathfrak{sl}_2\), plays a central role in the theory of vertex operator
algebras.  In the \(d\)-dimensional setting considered in this paper, it
is replaced by the global conformal Lie algebra
\[\sod.\]
Geometrically, this Lie algebra is realized on flat Euclidean space
\(\R^d\) by infinitesimal conformal transformations generated by
translations, rotations, dilations, and special conformal transformations \footnote{
For $d\geq 3$, Liouville's theorem implies that every orientation-preserving conformal embedding between connected open subsets of $\mathbb R^d$ is the restriction of an element of $\mathrm{SO}^+(d+1,1)$, the group of orientation-preserving conformal diffeomorphisms of the standard $d$-sphere \cite{Liouville}. In contrast, for $d=2$, such embeddings are precisely injective holomorphic maps, which need not be fractional linear transformations. However, the global orientation-preserving conformal diffeomorphisms of $S^2\cong\mathbb{CP}^1$ form $\mathrm{SO}^+(3,1)\cong\mathrm{PSL}_2(\mathbb C)$ and act by fractional linear transformations.
%For $d\geq 3$, Liouville's theorem implies that every orientation-preserving local conformal transformation is the restriction of 
%an element of the group of conformal diffeomorphisms of the standard $d$-sphere $\mathrm{SO}^+(d+1,1)$ \cite{Liouville}. In contrast, for $d=2$, orientation-preserving local conformal transformation are precisely injective holomorphic maps, which need not be fractional linear transformations. However, the global orientation-preserving conformal diffeomorphisms of $S^2\cong\mathbb{CP}^1$ form $\mathrm{SO}^+(3,1)\cong\mathrm{PSL}_2(\mathbb C)$ and act by fractional linear transformations.
%For \(d\geq3\), Liouville's theorem implies that any local (orientation-preserving) conformal
%transformation of Euclidean space is obtained from a restriction of 
%an element of the group of conformal diffeomorphisms of the standard \(d\)-sphere $\mathrm{SO}^+(d+1,1)$ \cite{Liouville}.
%Note that for \(d=2\), 向きを保つ共形写像と単射な正則写像が同値なため、\(\mathrm{SO}^+(3,1)\cong\mathrm{PSL}_2(\mathbb C)\)の\(S^2\cong\mathbb{CP}^1\) への fractional linear transformationsと一致する。
}.

\introitem{Grading}
In two-dimensional CFT, the dilation operator
$D=L(0)+\Ld(0)$
gives an eigenspace decomposition of the state space.
%In two-dimensional CFT, the state space is graded by
%the dilation operator
%\[
%D=L(0)+\bar L(0).
%\]
One usually requires this decomposition to be lower truncated with finite-dimensional eigenspaces; a vertex operator algebra
has
\[
V=\bigoplus_{n\in\mathbb Z}V_n,
\quad
V_n=0\ \text{for }n\ll0,
\quad
\dim V_n<\infty.
\]
%In algebraic formulations, it is natural to impose lower-truncation and
%finite-dimensionality conditions on this grading; for instance, a vertex
%operator algebra is usually assumed to have a lower-truncated grading
%\(V=\bigoplus_{n\in\Z}V_n\) with finite-dimensional homogeneous subspaces.
As a higher-dimensional analogue of such grading restrictions, we
introduce compact \(\sod\)-modules in Definition \ref{def_compact}.  Thus
the dilation operator \(D\) acts semisimply, and the state space has a
decomposition
\[
H=\bigoplus_{\Delta\in\R}H_\Delta
\]
satisfying the finiteness conditions specified there.

\introitem{Formal series}
For a vertex algebra, vertex operators take values in Laurent series
\(V((z))\).  In the two-dimensional full case, the corresponding formal
series involve both \(z\) and \(\bar z\), and have the form
\(H((z,\bar z,|z|^\R))\) \cite{M1}. In the \(d\)-dimensional setting, we introduce
the space
\[
H((x_1,\ldots,x_d,|x|^\R)),
\]
where real powers of
$|x|=(x_1^2+\cdots+x_d^2)^{1/2}$ are allowed, and the series are lower-truncated with respect to the total
degree.  This space is defined in Section \ref{sec_Laurent}.

%こうした定式化と主結果を説明する前に、頂点代数になじみのある読者のために次のテーブルにある基本的な辞書を説明したい。

%
%\begin{description}
%\item[symmetry.]
%頂点作用素代数で重要な役割を果たす Virasoro algebra (とくにその部分代数である$\mathfrak{sl}_2$)は、\(d\geq 2\)
%においては global conformal symmetry である、Lie algebra \(\sod\)に置き換わる\footnote{$d \geq 3$においてユークリッド空間$\R^d$上の任意の局所共形変換は、Liouville の定理から $\mathrm{SO}^+(d+1,1)$の$S^d$への作用の制限によって与えられる。$d=2$の場合、$\mathrm{SO}^+(3,1) \cong \mathrm{PSL}_2\C$の$S^2 \cong \CP$への作用であり、これは一次分数変換の高次元への一般化である。}。
%\item[grading)]
%また二次元の共形場理論のunderlying vector space は スケール変換 (dilation)に対応する $D=L(0)+\Ld(0)$ によって$\R$次数付けを持つが、
%その代数的定式化においては、その spectrum が下に有界かつ離散的であることを課すのは自然である (たとえば vertex operator algebra の場合は、$V=\bigoplus_{n\in\Z}V_n$が下に有界であることを仮定する)。
%こうした条件の高次元類似として、我々は$\sod$-module $H$に対して、compact という条件を導入する \footnote{この条件は$\{e^{-tD}\}_{t>0}$によって定めまる半群の表現が trace-class operator になることに関係する}. 
%\item[formal series)]
%$d$次元の頂点作用素(operator product expansion) を定義するため、
%$|x|=\sqrt{x_1^2+\dots+x_d^2}$の実数べきを許した下に有界かつ離散的な級数の空間
%\[
%H((x_1,\ldots,x_d,|x|^\R)),
%\]
%is introduced in Definition \ref{def_spherical_laurent}.  These spaces play,
%in the \(d\)-dimensional setting, the role played by \(H((z))\) in
%the theory of vertex algebras and by \(H((z,\bar z,|z|^\R))\) in the
%two-dimensional full case.
%\end{description}

\vspace{3mm}

Let \(H\) be a compact \(\sod\)-module, and let
\[
H^\vee=\bigoplus_{\Delta\in\R}H_\Delta^*,\qquad \overline{H}= \prod_{\D \in \R} H_\D
\]
be its restricted dual and its algebraic completion.  A \emph{\(d\)-dimensional conformally covariant system
of correlation functions}, or simply a \emph{\(d\)-CC system}, consists of a
family of maps
\[
\m:\Conf_n(\R^d)\longrightarrow \Hom_\C(H^{\otimes n},\overline{H}),
\qquad
(x_1,\dots,x_n)\longmapsto \m_{x_1,\dots,x_n},
\]
satisfying the analyticity, the vacuum property, permutation invariance,
translation covariance, conformal covariance, and the strong cluster
decomposition property.  The analyticity condition requires that, for
every \(u\in H^\vee\) and \(a_1,\ldots,a_n\in H\), the function
\[
\Conf_n(\R^d) \rightarrow \C, \qquad
(x_1,\dots,x_n)\mapsto
\langle u,\m_{x_1,\dots,x_n}(a_1,\ldots,a_n)\rangle
\]
is real analytic on \(\Conf_n(\R^d)\).

We now explain the strong cluster decomposition.  Let
\((x_1,\ldots,x_n,w)\in \Conf_{n+1}(\R^d)\) and
\((y_1,\ldots,y_m)\in \Conf_m(\R^d)\), and assume that
\begin{align}
\min_{1\leq i\leq n}|x_i-w|
>
\max_{1\leq j\leq m}|y_j|.
\label{intro_domain}
\end{align}
Then
$(x_1,\dots,x_n,w+y_1,\dots,w+y_m)\in \Conf_{n+m}(\R^d)$.

The cluster decomposition concerns the relation between
$\m_{x_1,\dots,x_n,w+y_1,\dots,w+y_m}(\bullet,\dots,\bullet)$
and $\m_{x_1,\dots,x_n,w}
\bigl(
\bullet,\dots,\bullet,
\m_{y_1,\dots,y_m}(\bullet,\dots,\bullet)
\bigr)$.
There is, however, an important point. The image of \(\m_{y_1,\ldots,y_m}\) lies in \(\overline H\), whereas the last argument of 
\(\m_{x_1,\ldots,x_n,w}\) is required to lie in \(H\). Thus the two maps cannot be composed directly.
%There is, however,
%an important point.  The image of $\m_{y_1,\dots,y_m}$ is $\overline{H}$, whereas the domains of \(\m_{x_1,\dots,x_n,w}\) is \(H\).  Thus the two maps cannot be
%composed directly.
\footnote{This is a common difficulty in the theory of vertex operator
algebras.  In Huang's pioneering work \cite{Huang}, vertex operator
algebras were described as algebras over partial operads, whose
compositions are only partially defined.  The composability condition, such
as \eqref{intro_domain}, is imposed by the underlying cobordism-type
geometry.  In \cite{M9}, for \(d\geq 2\), this point of view was replaced
by a (non-partial) operad of conformal disk embeddings. In that formulation, the same kind of constraint is
interpreted both geometrically and as the domains of unbounded operators on Hilbert spaces.}
One therefore decomposes the intermediate state with respect to the
\(D\)-eigenspace decomposition.  Let
\[
\pr_\Delta:\overline H=\prod_{\Delta'\in\R}H_{\Delta'}
\longrightarrow H_\Delta
\]
be the projection onto the \(\Delta\)-eigenspace.  The strong cluster
decomposition requires that, for every \(u\in H^\vee\) and
\(a_1,\ldots,a_{n+m}\in H\), the series
\[
\sum_{\Delta\in\R}
\left\langle
u,
\m_{x_1,\dots,x_n,w}
\left(
a_1,\ldots,a_n,
\pr_\Delta
\m_{y_1,\dots,y_m}(a_{n+1},\ldots,a_{n+m})
\right)
\right\rangle
\]
converges absolutely and uniformly on compact subsets of the
region \eqref{intro_domain} to
\[
\left\langle
u,
\m_{x_1,\dots,x_n,w+y_1,\dots,w+y_m}
(a_1,\ldots,a_{n+m})
\right\rangle .
\]
This is a $d$-dimensional variant of the convergence property in the
definition of a full field algebra of Huang--Kong.
The strong cluster decomposition is also closely related to a geometric formulation of the local structure of $d$-dimensional conformal field theory in terms of {\it conformally flat $d$-disk operads} \cite{M9}. In particular, under suitable conditions, their operadic composition laws imply the expansion formula appearing in the strong cluster decomposition \cite[Theorem~3.24]{M9}.

Starting from a \(d\)-CC system, we define an analytic vertex operator by
\[
\tilde{Y}(a,x)b=\m_{(x,0)}(a,b),
\]
which is real analytic on $x\in\R^d\setminus \{0\}$ (see \cite[Section 1]{HK} for $d=2$). Then, we show that the conformal covariance implies that
this operator admits a formal expansion in $H((x_1,\ldots,x_d,|x|^\R))$.
Hence, we obtain a $d$-dimensional vertex operator
\begin{align*}
Y(-,x):H\otimes H \rightarrow H((x_1,\ldots,x_d,|x|^\R)).
\end{align*}
Moreover, the strong cluster decomposition implies that products of these vertex operators are obtained as expansions, in different regions, of the same real analytic function. This is the \(d\)-dimensional counterpart of the associativity and
commutativity properties of vertex algebras in the chiral two-dimensional
case \cite{LL,FB}.

Such consistency of OPEs under different iterated expansions is central to the {\it conformal bootstrap} in physics \cite{Polyakov,BPZ,RRTV}.
In an axiomatic direction, Hollands \cite{HollandsOPE} formulated general quantum field theory in terms of OPE coefficients satisfying associativity conditions. In the two-dimensional full case, the associativity and commutativity of full vertex operators are essentially equivalent to the conformal bootstrap equations \cite{M2}. For locally \(C_1\)-cofinite full vertex operator algebras, the convergence and compatibility of arbitrary iterated OPEs were established in \cite[Theorem~3.9]{M8}.
Thus, the identities obtained here express, in a higher-dimensional setting, the OPE consistency underlying the conformal bootstrap.
This leads naturally to the notion of a {\it conformal \(d\)-vertex algebra}.

%Moreover, the strong cluster decomposition implies that products of these vertex operators are obtained as expansions, in different regions, of the same real analytic function. 
%This is the \(d\)-dimensional counterpart of the associativity and commutativity properties of vertex algebras in the chiral two-dimensional case. In the two-dimensional full case, the analogous associativity and commutativity
%of full vertex operators are essentially equivalent to the conformal bootstrap equations \cite{M2}.
%For locally $C_1$-cofinite full vertex operator algebras, the convergence and
%compatibility of arbitrary iterated OPEs were established in \cite[Theorem~3.9]{M8}. 
%Thus, the identities obtained here express, in a higher-dimensional setting, the OPE consistency underlying the conformal bootstrap in physics \cite{Polyakov,BPZ,RRTV}.
%Thus the identities obtained here
%are regarded as the OPE consistency expressed by the conformal bootstrap equations in physics. 

When \(d=2\), the exceptional isomorphism
\[
\so(3,1)_\C\simeq \mathfrak{sl}_2\C\oplus\mathfrak{sl}_2\C
\]
allows one to compare conformal \(2\)-vertex algebras with ordinary vertex algebras.
We show that the chiral and anti-chiral subspaces of a conformal \(2\)-vertex algebra carry vertex algebra structures, and that the underlying state space is naturally a module over each of these vertex algebras
 \footnote{Essentially the same result was proved in our previous paper \cite[Propositions 3.18 and 3.12]{M1}. Since the axioms are formulated slightly differently here, we include the argument in the present notation.}.

We also construct a family of examples corresponding to the generalized free scalar field theory.
%$d$-CC system の underlying vector space は compact $\sod$-module であり、この論文では parabolic induction を用いて、定義される $\sod$ の parameter $\al \in\R$を持った既約な highest および lowest weight modules $L(\al), L(\al)^\da$を考える。
%
%These examples depend on a real parameter \(\alpha\).  This parameter is the scaling dimension of the generating scalar field: the field constructed
%below satisfies
%\[
%[D,\phi(x)]=\left(\sum_{\mu=1}^d x_\mu \frac{d}{dx_\mu}+\alpha\right)\phi(x).
%\]
For $\alpha \in\R$, let \(L(\alpha)\) and \(L(\alpha)^\dagger\) be, respectively, irreducible highest and lowest weight representations of 
\(\sod\) obtained from parabolically induced modules (see Section \ref{sec_Verma}).
The construction is carried out for generic values of \(\alpha\)
\[
\alpha\in
\R_{>0}\setminus
\left\{
\frac{d-2}{2},\frac{d-4}{2},\frac{d-6}{2},\dots
\right\},
\]
where the parabolic Verma modules are irreducible.  We also treat
the distinguished value
$\alpha=\frac{d-2}{2}$
for \(d>2\).
%,which corresponds the massless scalar field theory.
%, where the corresponding quotient module is described in terms
%of harmonic polynomials.
%and their two-point correlation function is $\frac{1}{||x-y||^{2\alpha}}$. The construction works for generic \(\alpha \in \R \setminus \{\frac{d-2}{2},\frac{d-4}{2},\dots,0,-1,-2,\dots\}\), and
%also at the distinguished value
%$\alpha=\frac{d-2}{2}$ when \(d>2\),
%which corresponds to the usual massless free scalar field. 
  Using the
natural pairing
$\langle-,-\rangle:L(\alpha)^\dagger\otimes L(\alpha)\longrightarrow \R$,
we define a Heisenberg Lie algebra
\[
\widehat A_{d,\alpha}
=
L(\alpha)^\dagger\oplus L(\alpha)\oplus \R c
\]
by
\[
[a,a']=0,\quad
[a^\dagger,b^\dagger]=0,\quad
[a,a^\dagger]=\langle a^\dagger,a\rangle c,
\]
for \(a,a'\in L(\alpha)\) and \(a^\dagger,b^\dagger\in
L(\alpha)^\dagger\), where $c$ is in the center.  
%Then, the underlying vector space
%Let
%$\widehat A^+_{d,\alpha}=L(\alpha)\oplus \R c$ and let \(\R 1\) be the one-dimensional \(\widehat A^+_{d,\alpha}\)-module
%on which \(c\) acts as \(1\) and \(L(\alpha)\) acts trivially. Set
%\[
%H_{d,\alpha}
%=
%\mathrm{Ind}_{\widehat A^+_{d,\alpha}}^{\widehat A_{d,\alpha}}\R 1.
%\]
%By the PBW theorem, \(H_{d,\alpha}\) is naturally identified with the
%symmetric algebra of \(L(\alpha)^\dagger\).  Thus \(H_{d,\alpha}\) plays
%the role of the Fock representation of the Heisenberg algebra.
Set $\widehat{A}^+_{d,\alpha}=L(\alpha)\oplus \R c$ and
\[
H_{d,\alpha}=
\mathrm{Ind}_{\widehat A^+_{d,\alpha}}^{\widehat A_{d,\alpha}}\R 1,
\]
where \(\R1\) is the one-dimensional module on which \(c\) acts as \(1\)
and \(L(\alpha)\) trivially.
The compact \(\so(d+1,1)\)-module \(H_{d,\alpha}\) carries a conformally covariant scalar field
\[
\phi(x)=\phi^+(x)+\phi^-(x) \in \End\,H_{d,\alpha}[[x_1,\dots,x_d,|x|^\R]],
\]
satisfying
\[
[\phi^+(x),\phi^-(y)]= \frac{1}{|x-y|^{2\alpha}}\Bigl|_{|x|>|y|}.
\]
This is an analogue of $[h^+(z),h^-(w)]=\frac{1}{(z-w)^2} \Bigl|_{|z|>|w|}$ for the affine Heisenberg vertex algebra.
Using normal ordered products, we define multilinear maps
\[
\m_{x_1,\dots,x_n}:H_{d,\alpha}^{\otimes n}\longrightarrow \overline H_{d,\alpha}
\]
on configuration spaces, and show that these maps define a \(d\)-CC system, and hence a conformal \(d\)-vertex algebra.

The natural invariant bilinear form on \(H_{d,\alpha}\) is positive-definite precisely when
$\alpha\geq \frac{d-2}{2}$.
Thus the examples with \(0<\alpha<(d-2)/2\), whenever \(\alpha\) is generic, are non-unitary.

%We note that the natural invariant bilinear form on $H_{d,\alpha}$ is positive-definite precisely for
%\(\alpha\geq (d-2)/2\).
%For a general value of the parameter \(\alpha\), the theory associated with
%\(H_{d,\alpha}\) is non-unitary.

In Appendix~C, we consider the distinguished value \(\alpha=\dn\) and show that the complexification of the lowest weight module
$L\left(\dn\right)^\dagger\otimes_{\mathbb R}\mathbb C$ is, under the identification
$\so(d+1,1)_{\mathbb C}\cong\so(d,2)_{\mathbb C}$, the Harish--Chandra module of a projective unitary representation \(H^+\) of \(\mathrm{SO}_e(d,2)\). The representation \(H^+\) can be realized on a space of solutions to the Klein--Gordon equation and has been studied in \cite{BZ,KO1,KO2,HSS}. Moreover, its restriction to the Poincaré subgroup is the positive-energy massless scalar representation (see Appendix C). Consequently, the Hilbert-space completion of $H_{d,\dn}\cong \Sym \,L\left(\dn\right)^\dagger$
 is naturally isomorphic to the symmetric Fock space \(\Gamma_s(H^+)\), namely the state space of the free massless scalar field on Minkowski spacetime \cite{Arai,HL} (see Corollary \ref{cor_app_Fock}).
This may be viewed as a higher-dimensional, representation-theoretic counterpart of the relation between full two-dimensional Euclidean CFT and quantum field theory on Minkowski spacetime satisfying the Wightman axioms \cite{AGT,AMT}.

%This may be viewed as a higher-dimensional, representation-theoretic counterpart of the relation between full vertex operator algebras and Wightman conformal field theory in dimension two: for a unitary full vertex operator algebra, the full vertex operators give rise to Wightman fields on the Hilbert-space completion of the state space; see \cite{AGT,AMT}\footnote{A revised version of \cite{AMT}, containing this Wightman-field construction, is currently in preparation.}.

%
%In appendix, 境界の値$\alpha=\dn$において、我々は lowest weight module の複素化 $L\left(\dn\right)^\da \otimes \C$ が、$\sod_\C \cong \so(d,2)$-module として、ある$\mathrm{SO}_e(d,2)$の射影的ユニタリ表現$H^+$の Harish-Chandra module であることを見る。
%$H^+$は Klein-Gordon 方程式の解空間を用いて構成できる$\mathrm{SO}_e(d,2)$の被覆群のユニタリ表現であり、\cite{HSS,}などで調べられた。
%とくに$\Sym H_{d,\alpha}$の自然な内積によるヒルベルト空間としての完備化は $H^+$上のFock spaceであり、これは Minkowski時空上の massless scalar 場の state space に他ならない。
%この結果は、unitary full vertex operator algebra の Hilbert space としての完備化の上に full vertex operator algebra の頂点作用素を用いて Wightman field が定義できるという\cite{}の結果の高次元における表現論のレベルにおける対応物である。

The paper is organized as follows. In Section 1, we recall the infinitesimal conformal action of $\so(d+1,1)$ on $\R[x_1,\dots,x_d]$, introduce the space of formal series $H((x_1,\ldots,x_d,|x|^\R))$, and define conformally covariant vertex operators. In Section 2, we define $d$-CC systems and construct conformal $d$-vertex algebras from them. We also discuss the case $d=2$ and its relation with ordinary vertex algebras. In Section 3, we construct the examples described above. 
Appendix A contains the representation-theoretic results on the parabolic Verma modules used in Section 3. 
Appendix B proves the commutator formula at the distinguished value $\alpha=(d-2)/2$. Appendix C also studies $\alpha=(d-2)/2$ from the viewpoint of unitary representation theory.

\section{Preliminary}
Throughout this paper, we assume that $d \geq 2$.
In this section, we describe the action of the global conformal symmetry
$\mathfrak{so}(d+1,1)$ on polynomial rings and spaces of formal series, and,
based on these actions, introduce vertex operators in $d$-dimensional
conformal field theory.

\subsection{Infinitesimal conformally flat geometry}\label{sec_conf_inf}

In this section, we recall the action of $\mathfrak{so}(d+1,1)$ on the polynomial ring $\R[x_1,\dots,x_d]$ by differential operators. For the
geometric meaning of this action, see the following remark.

\begin{rem}
Let $S^d$ be the $d$-dimensional sphere equipped with the standard
Riemannian metric $g_\std$. The group of orientation-preserving conformal
diffeomorphisms $\mathrm{Conf}^+(S^d,g_\std)$ is isomorphic to
$\mathrm{SO}^+(d+1,1)$. Via the stereographic projection
$\R^d \hookrightarrow S^d$, the Lie algebra $\mathfrak{so}(d+1,1)$ acts on
the polynomial ring on $\R^d$.

When $d=2$, we have
$\mathrm{Conf}^+(S^2,g_\std)\cong \mathrm{SO}^+(3,1)\cong \mathrm{PSL}_2\C$,
and its action on $S^2$ is given by fractional linear transformations on
the complex projective line $\mathbb{C}P^1$.
%This action gives the
%differential-operator representation of
%$\mathfrak{so}(3,1) \cong \mathrm{sl}_2\C$,
%\[
%L(n) \mapsto -z^{n+1} \frac{d}{dz}
%\]
%for $n=-1,0,1$. 
For $d \geq 3$, the action of
$\mathrm{SO}^+(d+1,1)$ on $S^d$ is a higher-dimensional analogue of
fractional linear transformations. For more detailed explanations of
these actions, see, for example, \cite[Section 1.2]{M9}.
\end{rem}

%この論文を通じて$d \geq 2$とする。
%この章では大域的な共形対称性$\mathfrak{so}(d+1,1)$の多項式環や形式的級数環への作用を与え、それらに基づき$d$次元の共形場理論の頂点作用素を導入する。
%
%\subsection{Infinitesimal conformally flat geometry}\label{sec_conf_inf}

%この章では、$\mathfrak{so}(d+1,1)$の多項式環$\R[x_1,\dots,x_d]$への微分作用素による作用を振り返る (その幾何学的な意味については次のRemark をみよ)。
%\begin{rem}
%$d$次元の標準的なリーマン計量を入れた球面$S^d$の向きを保つ共形微分同相群$\mathrm{Conf}^+(S^d,g_\std)$は、$\mathrm{SO}^+(d+1,1)$と同型であり、stereographic projection $\R^d \hookrightarrow S^d$ を通じて、$\mathfrak{so}(d+1,1)$は$\R^d$上の多項式環に作用する。
%$d=2$の場合、$\mathrm{Conf}^+(S^2,g_\std)$は$\mathrm{SO}^+(3,1)\cong \mathrm{PSL}_2\C$であり、その$S^2$への作用は複素射影直線$\mathbb{C}P^1$への一次分数変換によって与えられる。こうした作用は$\mathfrak{so}(3,1) \cong \mathrm{sl}_2\C$の微分作用素表現$L(n) \mapsto -z^{n+1} \frac{d}{dz}$ ($n=-1,0,1$)を与える。$d \geq 3$における$\mathrm{SO}^+(d+1,1)$の$S^d$への作用は一次分数変換の高次元への拡張である (これらのより詳細な説明はたとえば\cite{}を参照)。
%\end{rem}
%
%
%$d$次元の
%$d$次元共形場理論の大域的な時空の対称性はメビウス変換群$M(S^d)$で与えられる。
%そのリー代数は符号数$(d+1,1)$の直交リー代数
%$\mathfrak{so}(d+1,1)$である。この章では$\mathfrak{so}(d+1,1)$の 関数環 $C^\infty(S^d,\R)$ および多項式環$\R[x^1,\dots,x^d]$への作用を調べる。

Let $\{g_{\mu,\nu}\}_{\mu,\nu \in \{0,1, \dots,d+1\}}$ be
the symmetric $(d+2)\times (d+2)$-matrix defined by
$$g_{\mu,\nu}= \begin{cases}
   0  & (\nu \neq \mu) \\
  -1 & (\nu=\mu=0) \\
   1 & (\nu=\mu \in \{1,\dots,d+1\}),
\end{cases}
$$
which gives a Lorentzian metric on $\R^{d+1,1}$.
%Then, $\mathfrak{so}(d+1,1)$ be the Lie algebra of the orthogonal group on $(\R^{d+1,1},g_{\mu,\nu})$.
Let $J_{\mu,\nu}$ be a $(d+2)\times (d+2)$-matrix whose $ij$-component is
$$
(J_{\mu,\nu})_{ij}= g_{i,\mu}\de_{\nu,j}-g_{i,\nu}\de_{j,\mu}.
$$

% search for combinient def.
The Lie algebra $\mathfrak{so}(d+1,1)$
is a $\frac{1}{2}(d+1)(d+2)$-dimensional Lie algebra with a basis $\{J_{\mu,\nu}\}_{\mu,\nu \in \{0,\dots,d+1\}}$ which satisfies the relations
\begin{align*}
[J_{\mu,\nu},J_{\rho,\sigma}]&=
g_{\nu,\rho}J_{\mu,\si}
- g_{\mu,\rho}J_{\nu,\si}
-g_{\nu,\si}J_{\mu,\rho}
+g_{\mu,\si}J_{\nu,\rho}
%[J_{0,\rho},J_{\mu,\nu}]&= \\
%[J_{0,\rho},J_{0,\sigma}]&=J_{\rho,\sigma} \\
\end{align*}
with $J_{\mu,\nu}=-J_{\nu,\mu}$ for $\mu,\nu,\rho,\si \in \{0,1,2,\dots,d+1 \}$.
The action of $\mathrm{SO}^+(d+1,1)$ on $S^d$ gives geometric meanings
to the elements of $\sod$, such as translations, rotations, and dilations.
It is therefore convenient to use the following basis, which is adapted to
this geometric interpretation:
%$\mathrm{SO}^+(d+1,1)$の$S^d$への作用から $\sod$の元は平行移動や回転、dilationなどの幾何学的な意味を持っている。
%幾何との関係から次の基底を用いることが便利である:
\begin{align*}
D&= J_{0,d+1}\\
P_\mu &= J_{0,\mu} + J_{d+1,\mu}\\
K_\mu &= -J_{0,\mu} + J_{d+1,\mu}
\end{align*}
for $\mu =1,\dots,d$.
These elements correspond respectively to dilations, translations, and
special conformal transformations (see \cite{M9}).
Then, $\{D, J_{\mu,\nu},K_\mu,P_\mu\}_{\mu,\nu \in \{1,\dots,d\}}$ forms a basis of $\sod$, which satisfies the following  relations
%The $d$-dimensional conformal group is a Lie algebra satisfying the following relations:
\begin{align}
\begin{split}
[D,K_\mu]&=  -K_\mu \\
[D,P_\mu]&= P_\mu \\
[D, J_{\mu,\nu}]&= 0 \\
[P_\mu, K_\nu] &= -2(\de_{\mu,\nu}D + J_{\mu,\nu}) \\
[J_{\mu,\nu},P_\rho]&=-\de_{\mu,\rho}P_\nu + \de_{\nu,\rho} P_\mu \\
[J_{\mu,\nu},K_\rho]&= -\de_{\mu,\rho}K_\nu + \de_{\nu,\rho} K_\mu  \\
[J_{\mu \nu}, J_{\rho,\si}]&=
\delta_{\nu,\rho}J_{\mu,\si}
- \de_{\mu,\rho}J_{\nu,\si}
-\de_{\nu,\si}J_{\mu,\rho}
+\de_{\mu,\si}J_{\nu,\rho}.
\end{split}
\label{eq_com_Lie}
\end{align}

\begin{rem}\label{rem_sod_sl2}
%頂点作用素代数になじみがある読者は、$\mathfrak{so}(3,1)$の複素化が、$\mathrm{sl}_2\C \oplus \mathrm{sl}_2\C$と同型であることに注意されたい。この同型は次のように与えられる:
Readers familiar with vertex operator algebras may note that the
complexification of $\mathfrak{so}(3,1)$ is isomorphic to
$\mathrm{sl}_2\C \oplus \mathrm{sl}_2\C$. This isomorphism is given as
follows:
\begin{align}
\begin{split}
L(1)=\frac{1}{2} (K_1+iK_2), \;\;&\Ld(1)=  \frac{1}{2}(K_1- iK_2) \\
L(0)= \frac{1}{2}(D + iJ_{12}),\;\; &\Ld(0)= \frac{1}{2}(D - iJ_{12})\\
L(-1)= \ft(P_1-iP_2), \;\;&\Ld(-1)= \ft(P_1+ iP_2),
\end{split}
\label{eq_L_sod}
\end{align}
where $\{L(n)\}_{n=-1,0,1}$ forms a subalgebra of the Virasoro algebra.
More generally, for every \(d\geq 2\), the elements appearing in
\eqref{eq_L_sod} span a subalgebra of \(\sod_\C\) isomorphic to
\(\mathrm{sl}_2\C \oplus \mathrm{sl}_2\C\).
%より一般に、\eqref{eq_L_sod}は$d$次元においても、$\sod_\C$の$\mathrm{sl}_2\C \oplus \mathrm{sl}_2\C$に同型な部分代数を定める。
\end{rem}
%and
%\begin{align*}
%z=x^1+ix^2,\;\; &\z=x^1-ix^2 \\
%\pa_z =\frac{1}{2}( \pa_1-i\pa_2),\;\;&\pa_\z = \frac{1}{2}( \pa_1+i\pa_2).
%\end{align*}

%\subsection{Action of conformal symmetry on function space}\label{sec_conf_action}
Let $\R[x_1,\partial_1,\dots,
x_d,\partial_d]$ be the ring of differential operators on $\R^d$.
Set $$|x|^2=
||x||^2=x_1^2+\dots+x_d^2 \in \R[x_1,\partial_1,\dots,
x_d,\partial_d]$$
and $$E(x)=\sum_{i=1}^d x_i \pa_i \in \R[x_1,\partial_1,\dots,
x_d,\partial_d],$$ the Euler operator.
Let $d:\mathfrak{so}(d+1,1) \rightarrow \R[x_1,\partial_1,\dots,
x_d,\partial_d]$ be a linear map defined by
\begin{align*}
d(D)&= - E(x) \\
d(K_\mu) &= ||x||^2\pa_\mu -2x_\mu E_x \\
d(P_\mu) &= -\pa_\mu \\
d(J_{\mu,\nu}) &= x_\mu \pa_\nu-x_\nu \pa_\mu
\end{align*}
for $\mu,\nu \in \{1,2,\dots,d \}$.
It is straightforward to check the following lemma:
\begin{lem}\label{lem_sod_dmod}
The linear map $d:\mathfrak{so}(d+1,1) \rightarrow \R[x_1,\partial_1,\dots,x_d,\partial_d]$ is a Lie algebra homomorphism,
where $\R[x_1,\partial_1,\dots,x_d,\partial_d]$ is regarded
as a Lie algebra by the commutators.
\end{lem}

\begin{rem}\label{rem_chiral_d}
In the case of $d=2$, we have:
\begin{align*}
d(L(-1)) &= \ft d(P_1-iP_2) = - \ft (\pa_1-i\pa_2)=-\pa_z,\\
d(L(0))&= \ft d(D+iJ_{12}) = -\ft (x_1\pa_1+x_2\pa_2-i x_1\pa_2+ix_2\pa_1) =-z\pa_z,\\
d(L(1)) &= \ft d(K_1+iK_2) =\ft(z\z (\pa_1+i\pa_2)-2z (z\pa_z+\z\pa_\z))=-z^2\pa_z,
\end{align*}
which is the representation of $\mathrm{sl}_2\C$ induced from the action of $\mathrm{PSL}_2\C$ on $\CP$.
\end{rem}

\subsection{Spaces of spherical Laurent series}\label{sec_Laurent}

Let \(\mathbb R[t^{\mathbb R}]\) be the group algebra of the additive
group \(\mathbb R\). Thus \(\mathbb R[t^{\mathbb R}]\) has a basis
\(\{t^r\}_{r\in\mathbb R}\) and multiplication given by $t^r t^s=t^{r+s}$.
Set
\[
\mathbb R[x_1,\ldots,x_d,|x|^{\mathbb R}]
=
\frac{\mathbb R[x_1,\ldots,x_d]\otimes \mathbb R[t^{\mathbb R}]}
{(x_1^2+\dots+x_d^2-t^2)}.
\]
In what follows we write \(t=|x|\) formally.

We give \(\mathbb R[x_1,\ldots,x_d,|x|^{\mathbb R}]\) an
\(\mathbb R\)-grading by
\[
\deg x_\mu=1,\qquad \deg |x|^r=r.
\]
We denote the homogeneous component of degree \(\lambda\) by
\[
\mathbb R[x_1,\ldots,x_d,|x|^{\mathbb R}]_\lambda.
\]

Let \(H\) be a vector space. We set
\[
H[x_1,\ldots,x_d,|x|^{\mathbb R}]_\lambda
=
H\otimes
\mathbb R[x_1,\ldots,x_d,|x|^{\mathbb R}]_\lambda
\]
and
\[
H[x_1,\ldots,x_d,|x|^{\mathbb R}]
=
H\otimes
\mathbb R[x_1,\ldots,x_d,|x|^{\mathbb R}].
\]
We also define its formal completion by
\[
H[[x_1,\ldots,x_d,|x|^{\mathbb R}]]
=
\prod_{\lambda\in\mathbb R}
H[x_1,\ldots,x_d,|x|^{\mathbb R}]_\lambda.
\]

The space of \(H\)-valued {\bf spherical Laurent series} is the subspace
\[
H((x_1,\ldots,x_d,|x|^{\mathbb R}))
\subset
H[[x_1,\ldots,x_d,|x|^{\mathbb R}]]
\]
consisting of elements \(f=(f_\lambda)_{\lambda\in\mathbb R}\) such that
for every \(M\in\mathbb R\),
\[
\{\lambda\in\mathbb R\mid f_\lambda\neq 0,\ \lambda<M\}
\]
is finite.

Define \(\partial_\mu\) on
these spaces by
\[
\partial_\mu\bigl(P(x)|x|^r\bigr)
=
(\partial_\mu P)(x)|x|^r
+
r x_\mu P(x)|x|^{r-2}.
\]
Then,
\[
H[x_1,\ldots,x_d,|x|^{\mathbb R}],\quad
H[[x_1,\ldots,x_d,|x|^{\mathbb R}]],\quad
H((x_1,\ldots,x_d,|x|^{\mathbb R}))
\]
are modules over $\mathbb R[x_1,\ldots,x_d,\partial_1,\ldots,\partial_d].$
Moreover, \(x_\mu\) raises the degree by \(1\), and \(\partial_\mu\) lowers
the degree by \(1\).

An element of
\(\mathbb C[x_1,\ldots,x_d,|x|^{\mathbb R}]\) can be regarded as a
real analytic function on \(\mathbb R^d\setminus\{0\}\) by putting
\(t=|x|=(x_1^2+\cdots+x_d^2)^{1/2}\). Indeed, by the polar decomposition
\begin{align}
\mathbb R^d\setminus\{0\}
\cong
\mathbb R_{>0}\times S^{d-1},
\qquad
x\mapsto
\left(|x|,\frac{x}{|x|}\right),
\label{eq_polar_regard}
\end{align}
each monomial \(P(x)|x|^r\) is real analytic on
\(\mathbb R^d\setminus\{0\}\).

\begin{dfn}
We denote by \(\mathcal O(S^{d-1})\) the subspace of
\(C^\infty(S^{d-1},\C)\) consisting of \(\SO(d)\)-finite functions, that is,
functions \(f\) such that the vector space spanned by the \(\SO(d)\)-orbit
of \(f\) is finite-dimensional.
\end{dfn}
It is clear that \(\mathcal O(S^{d-1})\) is a subalgebra of
\(C^\infty(S^{d-1})\).
The following proposition is well-known (see for example \cite[Theorem 5.7 and Theorem 5.12]{ABR}):
\begin{prop}
\label{prop_rep_func_space}
The image of the natural restriction map
\[
\mathbb C[y_1,\ldots,y_d]/(y_1^2+\cdots+y_d^2-1)
\longrightarrow
C^\infty(S^{d-1})
\]
is \(\mathcal O(S^{d-1})\).
\end{prop}

%For \(\lambda\in\R\), define a linear map
%\[
%\eta_\lambda:
%\C[y_1,\ldots,y_d]/(y_1^2+\cdots+y_d^2-1)
%\rightarrow
%\C[x_1,\ldots,x_d,|x|^\R]_\lambda
%\]
%as follows. If \(P(y)=\sum_{m\ge0}P_m(y)\) is the homogeneous decomposition,
%with \(P_m\in\C[y_1,\ldots,y_d]_m\), set
%\[
%\eta_\lambda([P])
%=
%\sum_{m\ge0}P_m(x)|x|^{\lambda-m}.
%\]
%Then, it is easy to show that \(\eta_\lambda\) is
%well-defined and is a linear isomorphism.
%
%First we check that \(\eta_\lambda\) is well-defined. Let
%\[
%Q_y=y_1^2+\cdots+y_d^2,\qquad Q_x=x_1^2+\cdots+x_d^2.
%\]
%For \(R_m\in \C[y_1,\ldots,y_d]_m\), we have
%\[
%\eta_\lambda(Q_yR_m-R_m)
%=
%Q_xR_m(x)|x|^{\lambda-m-2}
%-
%R_m(x)|x|^{\lambda-m}.
%\]
%This is zero in \(A_\lambda\), because \(Q_x=|x|^2\) in \(A\). Hence
%\(\eta_\lambda\) descends to the quotient by \((Q_y-1)\).
%
%Conversely, define
%\[
%\pi_\lambda:A_\lambda
%\longrightarrow
%\C[y_1,\ldots,y_d]/(Q_y-1)
%\]
%by
%\[
%\pi_\lambda\left(P_m(x)|x|^{\lambda-m}\right)
%=
%[P_m(y)]
%\]
%for \(P_m\in\C[x_1,\ldots,x_d]_m\). This is well-defined, since
%\[
%\pi_\lambda\left(Q_xP_m(x)|x|^{\lambda-m-2}\right)
%=
%[Q_yP_m(y)]
%=
%[P_m(y)]
%=
%\pi_\lambda\left(P_m(x)|x|^{\lambda-m}\right).
%\]
%It is immediate that
%\[
%\pi_\lambda\circ\eta_\lambda=\mathrm{id},
%\qquad
%\eta_\lambda\circ\pi_\lambda=\mathrm{id}.
%\]
%Thus \(\eta_\lambda\) is an isomorphism.

The following proposition follows from Proposition \ref{prop_rep_func_space}:
\begin{prop}
\label{prop_analytic_Laurent}
The natural map
\[
\C[x_1,\ldots,x_d,|x|^\R]
\longrightarrow
C^\omega(\R^d\setminus\{0\})
\]
defined by
$|x|^r=(x_1^2+\cdots+x_d^2)^{r/2}$
is injective. Moreover, for any \(\lambda\in\R\), the image of
$\C[x_1,\ldots,x_d,|x|^\R]_\lambda$
is exactly the space of real analytic functions \(f\) on
\(\R^d\setminus\{0\}\) satisfying
\[
E(x) f=\lambda f
\]
and such that the vector space spanned by $\SO(d)$-orbit of $f$ is finite dimensional.
\end{prop}

In the case \(d=2\), the space \(\C((x_1,x_2,|x|^\R))\) can be described
in terms of the complex coordinate \(z\) and its complex conjugate \(\z\) as follows.
%$d=2$の場合は、複素座標$z$とその共役$\z$を用いて、$\C((x_1,x_2,|x|^\R))$を与えることができる。以下でこのことを説明する。
%\cite{M1}で、我々は二次元共形場理論を定義するために形式的級数の空間を導入した。以下で、$d=2$の場合に上記のは\cite{M1}の定義を少しだけ変更したものになっていることを見る。
Let $z,\z$ be independent formal variables and $\C[[z,\z,|z|^\R]]$ the space of formal power series spanned by
\begin{align}
\sum_{r,s\in \R}a_{r,s} z^r\z^s\quad\quad (a_{r,s}\in \C)
\label{eq_series_bound}
\end{align}
such that:
\begin{itemize}
\item
$a_{r,s}=0$ unless $r-s \in\Z$.
\end{itemize}
Denote by $\C((z,\z,|z|^\R))$ the subspace of $\C[[z,\z,|z|^\R]]$ spanned by vectors of the form \eqref{eq_series_bound} such that:
\begin{enumerate}
\item[M1)]
For any $M\in \R$, 
$\{(r,s)\in \R^2 \mid a_{r,s}\neq 0 \text{ and }r+s < M \}$
is a finite set.
%\item
%There exists $N \in \R$ such that $a_{r,s}=0$ if $r<N$ or $s<N$.
\end{enumerate}

For any $n,m\in \Z_{\geq 0}$ and $r\in \R$, by
\begin{align*}
x_1^n x_2^m|x|^r \mapsto \left(\frac{z+\z}{2}\right)^n\left(\frac{z-\z}{2i}\right)^m (z\z)^{\frac{r}{2}} \in \C((z,\z,|z|^\R)),
\end{align*}
we have a linear isomorphism:
\begin{align*}
\C((x_1,x_2,|x|^\R)) \rightarrow \C((z,\z,|z|^\R)).
\end{align*}

%
%
%\begin{rem}
%\cite{M1}において我々は下に有界な形式的級数の空間を$\C[[z,\z,|z|^\R]]$の部分空間であって、(M1)および次の条件(M2)を満たすものとして定義した。
%\begin{itemize}
%\item[M2)]
%There exists $N \in \R$ such that $a_{r,s}=0$ if $r<N$ or $s<N$
%\end{itemize}
%またその定義を元に full vertex algebra を導入し様々な性質を証明した。しかしそうした証明は多くの場合、(M2) を仮定することなく証明できることが分かる。高次元における定義の自然さからも full vertex algebra の定義から (M2) は外すべきかもしれない。
%(ただし有理的共形場理論やユニタリ共形場理論では(M2)が自然に満たされており、我々はM2が成り立たない自然な例を知らない).
%\end{rem}
%

\subsection{Space of vertex operators}\label{sec_space_vertex}
In this section, we define a space of $d$-dimensional vertex operators.

\begin{dfn}\label{def_compact}
We call a $\mathfrak{so}(d+1,1)$-module $H$ 
\textbf{compact} if it satisfies the following conditions:
\begin{enumerate}
\item[C1)]
$D$ acts semisimply on $H$ with real eigenvalues, i.e.,
$H=\bigoplus_{\D\in \R}H_\D$, where $H_\D=\{a\in H\mid Da=\D a\}$ for $\D \in \R$;
\item[C2)]
For any $M \in \R$, $\sum_{\D \leq M}\dim H_\D$ is finite;
%$\{\D \in \R\mid H_\D \neq 0\text{ and }\D \leq M\}$ is a finite set;
\end{enumerate}
We note that for any $\D \in \R$ $J_{\nu,\mu}H_\D \subset H_\D$, $P_\mu H_\D \subset H_{\D+1}$ and $K_\mu H_\D \subset H_{\D-1}$. In particular, by (C2), $K_\mu$ is locally nilpotent and $H_\D$ is a module of the subalgebra $\mathfrak{so}(d)=\langle J_{\nu,\mu} \rangle \subset \mathfrak{so}(d+1,1)$.
\begin{enumerate}
\item[C3)]
For any $\D \in \R$, the action of $\mathfrak{so}(d)$ on $H_\D$ can be exponentiated to the Lie group $\mathrm{SO}(d)$,
i.e., it is properly a representation of $\mathrm{SO}(d)$ and not of $\mathrm{Spin}(d)$.
% and a direct sum of finite dimensional representations of $\mathrm{SO}(d)$.
\end{enumerate}
\end{dfn}

Let $\mathfrak{so}(d+1,1) \otimes_\R \R[x_1,x_2,\dots,x_d]$ be the Lie algebra defined by
\begin{align*}
[X\otimes f,Y\otimes g] = [X,Y]\otimes fg
\end{align*}
for any $X,Y \in \sod$ and $f,g \in \R[x_1,\dots,x_d]$.
Since $P_\mu$ is ad-nilpotent for all $\mu =1,\dots,d$,
the map $\exp(-\sum_{\rho} x_\rho \mathrm{ad} P_\rho)$ defines a Lie algebra automorphism of $\mathfrak{so}(d+1,1) \otimes_\R \R[x_1,x_2,\dots,x_d]$.

\begin{rem}
\label{rem_shift_automorphism}
This automorphism appears for the following reason.
In conformal field theory, the underlying vector space \(H\) is the space of states of particles localized at the
origin \(0 \in \R^d\). The state obtained by inserting a state \(a \in H\)
at a point \(x\) is given by
\[
\exp\left(\sum_{\rho} x_\rho P_\rho\right)a.
\]
For readers familiar with vertex operator algebras, this is a special case
of skew-symmetry:
\begin{align*}
Y(a,z)\va = \exp(zL(-1))a.
\end{align*}
Then the action of \(X \in \mathfrak{so}(d+1,1)\) on
\(\exp(\sum_{\rho} x_\rho P_\rho) a\) can be formally written as
\begin{align*}
X\exp(\sum_{\rho} x_\rho P_\rho) a 
&= \exp(\sum_{\rho} x_\rho P_\rho)\exp(-\sum_{\rho} x_\rho P_\rho)X \exp(\sum_{\rho} x_\rho P_\rho) a\\
&=\exp(\sum_{\rho} x_\rho P_\rho) 
\left(\exp(-\sum_{\rho} x_\rho \mathrm{ad}P_\rho)X\right)a.
\end{align*}
For example, in the case of a vertex operator algebra, we have
\begin{align*}
\exp(-z \mathrm{ad} L(-1))L(1) = L(1)+2zL(0)+z^2L(-1),
\end{align*}
which appears precisely in the well-known formula (see for example \cite{FHL})
\begin{align*}
[L(1),Y(a,z)] = \sum_{k=0}^2 \binom{2}{k} Y( L(1-k)a ,z)z^{k}.
\end{align*}
Thus, in order to formulate the conformal covariance of vertex operators,
one needs to twist by this automorphism.
%この自己同型群は次のような理由で現れる。
%The vector space $H$は原点$0 \in \R^d$に局所的に存在する粒子たちの状態のなすベクトル空間である。
%点$x$に状態$a \in H$を挿入した状態は $\exp(\sum_{\rho} x_\rho P_\rho) a$とかける。
%頂点作用素代数になじみがある場合、これは skew-symmetry の特別な場合である:
%\begin{align*}
%Y(a,\uz)\va = \exp(zL(-1))a.
%\end{align*}
%このとき、$\exp(\sum_{\rho} x_\rho P_\rho) a$への$X \in \mathfrak{so}(d+1,1)$の作用は、
%\begin{align*}
%X\exp(\sum_{\rho} x_\rho P_\rho) a 
%&= \exp(\sum_{\rho} x_\rho P_\rho)\exp(-\sum_{\rho} x_\rho P_\rho)X \exp(\sum_{\rho} x_\rho P_\rho) a\\
%&=\exp(\sum_{\rho} x_\rho P_\rho) 
%\left(\exp(-\sum_{\rho} x_\rho \mathrm{ad}P_\rho)X\right)a
%\end{align*}
%とかける。たとえば頂点作用代数の場合、
%\begin{align*}
%\exp(-z \mathrm{ad} L(-1))L(1) = L(1)+2zL(0)+z^2L(-1)
%\end{align*}
%であるが、これは良く知られた公式
%\begin{align*}
%[L(1),Y(a,z)] = \sum_{k=0}^2 \binom{2}{k} Y( L(1-k)a ,z)z^{k}
%\end{align*}
%に他ならない。
\end{rem}

\begin{lem}\label{lem_covariance_Lie}
For any $\mu,\nu \in \{1,2,\dots,d\}$,
\begin{align*}
\exp(-\sum_{\rho=1}^d x_\rho \mathrm{ad} P_\rho)P_\mu
&=P_\mu \\
\exp(-\sum_{\rho=1}^d x_\rho \mathrm{ad} P_\rho)J_{\mu,\nu}
&=J_{\mu,\nu}+ x_\nu P_\mu-x_\mu P_\nu \\
\exp(-\sum_{\rho=1}^d x_\rho \mathrm{ad} P_\rho)D
&=D + \sum_{\rho=1}^d x_\rho P_\rho \\
\exp(-\sum_{\rho} x_\rho \mathrm{ad} P_\rho)K_\mu
&=K_\mu+ 2x_\mu D + 2 \sum_{\rho =1}^d x_ \rho J_{\rho,\mu}
+ 2x_\mu \sum_{\rho=1}^d x_\rho P_\rho
- ||x||^2P_\mu.
\end{align*}
\end{lem}
\begin{proof}
We will check only the last equation.
Since $[-\sum_{\rho} x_\rho P_\rho, K_\mu] = 2x_\mu D+2\sum_\rho x_\rho J_{\rho,\mu}$
, $ [-\sum_{\rho} x_\rho P_\rho, x_\mu D] =x_\mu \sum_\rho x_\rho P_\rho$ and $[-\sum_{\rho} x_\rho P_\rho, \sum_\rho x_\rho J_{\rho,\mu}] = x_\mu \sum_{\rho}x_\rho P_\rho- ||x||^2P_\mu$
we have:
\begin{align*}
\exp(-\sum_{\rho} x_\rho \mathrm{ad} P_\rho)K_\mu
&=K_\mu+ 2x_\mu D + 2 \sum_{\rho =1}^d x_ \rho J_{\rho,\mu}
+ 2x_\mu \sum_{\rho=1}^d x_\rho P_\rho
- ||x||^2P_\mu.
\end{align*}
\end{proof}

\begin{rem}\label{rem_chiral_cov}
By applying Lemma \ref{lem_covariance_Lie}
in the case of $d=2$, we obtain
\begin{align*}
\exp(-x_1 \mathrm{ad}P_1-x_2 \mathrm{ad} P_2) K_1 = K_1+2x_1 D - 2x_2 J_{12} + 2x_1(x_1P_1+x_2P_2)- ||x||^2P_1,\\
\exp(-x_1 \mathrm{ad}P_1-x_2 \mathrm{ad} P_2) K_2 = K_2+2x_2 D + 2x_1 J_{12} + 2x_2(x_1P_1+x_2P_2)- ||x||^2P_2
\end{align*}
By Remark \ref{rem_sod_sl2}, we have $L(1)=\ft( K_1+iK_2)$. Setting $z=x_1+ix_2$, we get
\begin{align*}
\exp(-x_1 \mathrm{ad}P_1-x_2 \mathrm{ad} P_2)L(1) &= L(1)+z D +izJ_{12}+z(x_1P_1+x_2P_2) -\ft z\z(P_1+iP_2)\\
&= L(1)+2zL(0)+\ft z^2 P_1- \frac{i}{2}z^2P_2 \\
&= L(1)+2zL(0)+z^2 L(-1),
\end{align*}
which is consistent with Remark \ref{rem_shift_automorphism}.
\end{rem}

%In section \ref{sec_conf_inf}, we introduce 
%a $\R[x_1,\pa_1,\dots,x_d,\pa_d]$-modules $H[x_1,\dots,x_d,|x|^\R]$, $H[[x_1,\dots,x_d,|x|^\R]]$ and $H((x_1,\dots,x_d,|x|^\R))$.
%By Lemma \ref{lem_sod_dmod}, they are $\sod$-modules.
%Assume $H$ is a compact $\sod$-module.
%Then, we can consider the tensor product representation of $\sod$ on $H[x_1,\dots,x_d,|x|^\R] = H \otimes_\C \R[x_1,\dots,x_d,|x|^\R]$, and similarly on $H[[x_1,\dots,x_d,|x|^\R]]$ and $H((x_1,\dots,x_d,|x|^\R))$.
%Hereafter, we will think of these as the tensor product module.
Let $H$ be a compact $\sod$-module.
For a linear map
\begin{align*}
F \in \mathrm{Hom}_\C\left(H\otimes H, H((x_1,\dots,x_d,|x|^\R))\right)
\end{align*}
define an action of $\sod$ by
\begin{align*}
\left(A\cdot F\right) (a,b) = A F(a,b) -F(a,A b) -F((\exp(-\sum_\rho x_\rho \ad P_\rho ) A) a,b)
\end{align*}
for $A\in \sod$ and $a,b\in H$.
%which is a usual action of Lie algebra on a tensor product and a $\mathrm{Hom}$-space.

\begin{dfn}\label{def_cov}
A linear map $F \in \mathrm{Hom}_\C\left(H\otimes H, H((x_1,\dots,x_d,|x|^\R))\right)$ is called a \textbf{ conformally covariant vertex operator} if
$F(P_\mu a,b)=\frac{d}{dx_\mu} F(a,b)$ and $A\cdot F=0$ for any $A\in \sod$.
Following the conventions of vertex algebra, we denote it by
\begin{align*}
Y(-,x): H\otimes H\rightarrow H((x_1,\dots,x_d,|x|^\R)),\quad a\otimes b \mapsto Y(a,x)b.
\end{align*}
\end{dfn}

More explicitly, by Lemma \ref{lem_covariance_Lie}, a linear map $Y(-,x)$ is conformally covariant if 
\begin{enumerate}
\item
$Y(P_\mu a,x)=\frac{d}{dx_\mu} Y(a,x)$ for any $\mu =1,\dots,d$ and $a\in H$;
\item
For any $a,b\in H$,
\begin{align}
\begin{split}
P_\mu {Y}(a,x)b &= {Y}(P_\mu a,x)b +{Y}(a,x)P_\mu b\\
J_{\mu,\nu} {Y}(a,x)b
&={Y}(J_{\mu,\nu}a,x)b+ {Y}(a,x)J_{\mu,\nu} b+
\left(x_\nu \frac{d}{dx_\mu}-x_\mu \frac{d}{dx_\nu}\right){Y}(a,x)b\\
D {Y}(a,x)b &=
{Y}(D a,x)b+{Y}(a,x) Db+E(x){Y}(a,x)b\\
K_\mu {Y}(a,x)b
 &= {Y}(K_\mu + 2x_\mu D+2 \sum_\rho x^\rho J_{\rho,\mu}a,x)b + {Y}(a,x)K_\mu b+\left(2x_\mu E(x)-||x||^2 \frac{d}{dx_\mu}\right){Y}(a,x)b.
\end{split}
\label{eq_cov_def}
\end{align}
\end{enumerate}

\section{System of correlation functions in conformal field theory}
In this section, we formulate and study systems of correlation functions
for \(d\)-dimensional conformal field theories whose underlying vector spaces (the state spaces) are
compact \(\sod\)-modules. We call such theories compact conformal field
theories.

In Section \ref{sec_cluster}, we give the definition of a system of correlation
functions. In Section \ref{sec_from_d}, we prove that conformally covariant vertex
operators can be extracted from such systems of correlation functions. 
In Section \ref{sec_def_vertex}, we show that compositions of these vertex operators are realized as different expansions of a single real analytic function, and that they satisfy identities corresponding to
associativity and commutativity. These identities correspond to the conformal bootstrap equations in physics.
In Section
\ref{sec_rel_vertex}, we study the case \(d=2\) in detail and, using the exceptional
isomorphism
\(\mathfrak{so}(3,1)_\C\cong \mathrm{sl}_2\C\oplus \mathrm{sl}_2\C\),
derive the relation with vertex algebras.
%この章では state 空間が compact $\sod$-module となるような$d$次元の共形場理論(コンパクト共形場理論)の相関関数系を定式化しその性質を調べる。
%Section \ref{} では相関関数系の定式化を与える。
%Section \ref{} では、こうした相関関数系からconformally covariant な頂点作用素が取り出せることを証明する。Section \ref{} では、この頂点作用素の合成が解析接続をすると、結合法則や交換法則にあたる等式を満たすことを示す。こうした等式は物理におけるconformal bootstrap equation に対応する。
%Section \ref{} では、$d=2$ の場合を詳しく調べ、例外同型$\mathfrak{so}(3,1)_\C\cong \mathrm{sl}_2\C\oplus \mathrm{sl}_2\C$に基づき、頂点代数との関係を導出する。
%
%
% conformally flat d-algebra $(F,\m,\va)$から頂点作用素
%\begin{align}
%Y(\bullet,x):F \rightarrow \End (F)[[x_1,\dots,x_d,|x|^\R]]
%\label{eq_vertex_intro2}
%\end{align}
%を取り出す。
%Section \ref{sec_space_vertex}では\eqref{eq_vertex_intro2}の形の線形作用素の空間への$\mathfrak{so}(3,1)$の作用を定義する。
%二次元の場合は \cite{HK} の？章が対応する (see also \cite[Section ]{})。
%Section \ref{sec_from_d}では
%conformally flat $d$-algebra から頂点作用素を取り出し、その性質を調べる。
%Section \ref{sec_def_vertex}では$\mathfrak{so}(3,1)$不変な頂点作用素が満たすべき性質を用いて、full d-vertex algebra を定義する。
%Section \ref{sec_2d}では$d=2$の場合に、頂点代数やfull頂点代数との関係を調べる。とくに full $d$-vertex algebra は (full) 頂点代数の$d$次元への一般化と思える。

\subsection{Conformally covariant systems of correlation functions}
\label{sec_cluster}

Let $H$ be a compact $\mathfrak{so}(d+1,1)$-module. Set 
\begin{align}
\overline{H} = \Pi_{\D \in \R} H_\D\quad\text{and}\quad
H^\vee =\bigoplus_{\D\in \R} H_{\D}^*,\label{eq_restricted_dual}
\end{align}
where $H_\D^*$ is the dual vector space.
Denote the canonical pairing on $H^\vee \otimes \F \rightarrow \C$ by $\langle -,-\rangle$.
Let
\begin{align*}
\pr_\D:\F \rightarrow H_\D,
\end{align*}
be the projection for $\D\in\R$.
For $n \geq 1$ and $u \in H^\vee$ and $a_i \in H$ ($i=1,\dots,n$)
and $f\in \End\, H$, $p \in \{1,\dots,n\}$,
we will use the following notation:
\begin{align*}
[n]&= \{1,\dots,n\}\\
\ar&=a_1\otimes \cdots\otimes a_n \in  H^{\otimes n},\\
f_p  \ar &=  a_1\otimes \cdots \otimes f(a_p)\otimes \cdots\otimes a_n \in H^{\otimes n}.
%f^*  u &= u \circ f \in H^\vee.
\end{align*}
%where $(f^* m_0^*)(v)=m_0^*(f(v))$ for $v\in F$.
If \(f\in\End\, H\) is homogeneous with respect to the \(D\)-grading, we set
\[
f^*u=u\circ f\in H^\vee.
\]

Set
\begin{align*}
\Conf_n(\R^d)=\{(x^{(1)},\dots,x^{(n)}) \in (\R^d)^n \mid x^{(i)} \neq x^{(j)}\},
\end{align*}
the $n$-point configuration space of $\R^d$. Denote $(x^{(1)},\dots,x^{(n)})\in \Conf_n(\R^d)$ by $\xr$ as above.

For $n,m \geq 1$, set
\begin{align}
\begin{split}
U_{n,m}&=\{
((x^{(1)},\dots,x^{(n)},w),(x^{(n+1)},\dots,x^{(n+m)}) )\in \Conf_{n+1}(\R^d)\times \Conf_m(\R^d)\\
&\mid
\min_{k \in [n]}|x^{(k)}-w|>\max_{l \in [m]} |x^{(n+l)}| \Bigr\},
\end{split}
\label{eq_upq_def}
\end{align}
an open subset of $\Conf_{n+1}(\R^d)\times \Conf_m(\R^d)$. Set
\begin{align*}
U_{0,m}=\Conf_1(\R^d) \times \Conf_m(\R^d).
\end{align*}

The following definition is a higher-dimensional generalization, with some
modifications, of the full field algebra introduced by Huang and Kong in \cite{HK} to formulate two-dimensional non-chiral conformal field theory. For a geometric approach to correlation functions in dimension $d\geq 2$ based on conformally flat $d$-disk operads, see \cite{M9}.
%以下の定義は二次元のchiral でない共形場理論を定式化するために、 \cite{HK}が導入した full field algebra の高次元への (一部を修正した) 一般化である。

%
%この章では、完備化$\F$を用いて、我々は permutation 不変な写像の列
%\begin{align*}
%\m:\Conf_r(\R^d) \rightarrow \Hom_\C(H^{\otimes r},\F)
%\end{align*}
%を考える。ここでいくつかの注意をする。まず$r=0$の場合は、
%\begin{align*}
%\m:* \rightarrow \Hom_\C(H^{\otimes 0},\F)=\Hom_\C(\C,\F)=\F
%\end{align*}
%であり、$*$の像はある特別なベクトル$\va \in \F$を定める。
%これは場の量子論の真空にあたり$\C \va$は$\sod$の自明表現を定めることを仮定するのが自然である (Wightman 公理の vacuum もみよ)。このベクトルは代数の単位元に対応する。$r=1$の場合は
%\begin{align*}
%\m:\Conf_1(\R^d)=\R^d \rightarrow \Hom_\C(H,\F)
%\end{align*}
%であり、$0 \in \Conf_1(\R^d)$がoperadのunitであることから$\m(0)=\id_H$を要請するのが自然である。
%最後に\eqref{eq_operad_comp_ex}の解釈を与える。
%permutation invariance から 合成演算$\circ_i$は 
%\begin{align*}
%\circ_{p+1}:\Conf_{p+1}(\R^d) \times \Conf_q(\R^d) \rightarrow \Conf_{p+q}(\R^d)
%\end{align*}
%のように$p+1$番目の変数に対してのみ考えて一般性を失わない。合成を行う座標を$w$を用いて書くことにし、$(x_1,\dots,x_p,w) \in \Conf_{p+1}(\R^d)$という記法を用いる。また合成をすると、二番目の球面は$w$だけ平行移動されることから、二番目の変数は初めから平行移動された変数 $(x_{p+1}-w,\dots,x_{p+q}-w)$を用いることにする。
%すると合成可能なとき
%\begin{align}
%(x_1,\dots,x_p,w)\circ_{p+1}
%(x_{p+1}-w,\dots,x_{p+q}-w)
%= (x_1,\dots,x_{p+q}) \in \Conf_{p+q}(\R^d)
%\label{eq_comp_well_pq}
%\end{align}
%が成り立つ。\eqref{eq_comp_well_pq}が定義可能になる領域は、Proposition \ref{prop_operad_D}より
%によって与えられる。
%すなわち合成 $\cM_{(x_1,\dots,x_p,w)}(a_1,\dots,a_p,\cM_{(x_{p+1}-w,\dots,x_{p+q}-w)}(a_{p+1},\dots,a_{p+q}))$ (for $a_i \in H$)が、実解析的関数として well-defined になってほしい。
%More precisely, we define a conformally flat d-algebra as follows:
\begin{dfn}\label{def_conformal_algebra}
A \textbf{$d$-dimensional conformally covariant system of correlation functions},
or simply a \textbf{$d$-CC system}, is a compact $\mathfrak{so}(d+1,1)$-module $H$
equipped with a sequence of maps 
\begin{align*}
\m:\Conf_n(\R^d) \rightarrow  \Hom_\C(H^{\otimes n},\F),\quad\quad 
(\xr, a_1,\dots,a_n) \mapsto \m_{\xr}(a_1,\dots,a_n)
\end{align*}
($n=1,2,3,\dots$),
which are linear in $H^{\otimes n}$, and a distinguished vector $\va \in H_0$ such that:
\begin{description}
\item[unit)]
For any $a\in H$, $\m_{(0)}(a) = a$;
\item[analyticity)]
For any $\ar \in H^{\otimes n}$ and $u \in H^\vee$,
\begin{align*}
\Conf_n(\R^d) \rightarrow \C,\quad
\xr \mapsto \langle u, \m_{\xr}(\ar)\rangle
\end{align*}
is a real analytic function on $\Conf_n(\R^d)$;
\item[permutation invariance)]
For any $\si$ in the permutation group $S_n$,
\begin{align*}
\m_{(x^{(1)},\dots,x^{(n)})}(a_1,\dots,a_n)=\m_{(x^{(\si 1)},\dots,x^{(\si n)})}(a_{\si1},\dots,a_{\si n});
\end{align*}
\item[vacuum property)]
For any $n \geq 1$,
\begin{align*}
\m_{(x^{(1)},\dots,x^{(n)},x^{(n+1)})}(a_1,\dots,a_n,\1)=\m_{(x^{(1)},\dots,x^{(n)})}(a_1,\dots,a_n);
\end{align*}
\item[invariance of vacuum)]
$\C \va$ is a trivial representation of $\sod$, that is, $X\va=0$ for any $X\in \sod$;
\item[translation invariance)]
For any $p=1,\dots,n$ and $\mu \in\{1,\dots,d\}$,
\begin{align*}
\m_{\xr}((P_\mu)_p \ar) &= \frac{d}{dx_\mu^{(p)}} \m_{\xr}(\ar);
\end{align*}
\item[conformal covariance)]
For any $X\in \mathfrak{so}(d+1,1)$,
\begin{align*}
X \m_{\xr}(\ar) = \sum_{q=1}^n \m_{\xr}\left(\left(\exp(-\sum_{\rho=1}^d x_\rho^{(q)} \mathrm{ad} P_\rho)X\right)_q \ar\right).
\end{align*}
%$p=1,\dots,r$ and $\mu,\nu \in\{1,\dots,d\}$,
\item[strong cluster decomposition)]
Let $n \geq 0$ and $m \geq 1$.
Then, for any $a_i \in H$ ($i=1,\dots,n+m$) and $u\in H^\vee$ and $(x^{(1)},\dots,x^{(n)},w,x^{(n+1)},\dots,x^{(n+m)}) \in U_{n,m}$,
\begin{align}
\sum_{\D \in \R}
\langle u, \m_{(x^{(1)},\dots,x^{(n)},w)}(a_1,\dots,a_n, \pr_\D \m_{(x^{(n+1)},\dots,x^{(n+m)})}(a_{n+1},\dots,a_{n+m})\rangle
\label{eq_cluster}
\end{align}
is absolutely convergent and moreover \eqref{eq_cluster} is locally uniformly convergent to 
\begin{align*}
\langle u, \m_{(x^{(1)},\dots,x^{(n)},x^{(n+1)}+w,\dots, x^{(n+m)}+w)}(a_1,\dots,a_{n+m})\rangle,
\end{align*}
that is, for any compact subset $K \subset U_{n,m}$ and $\ep >0$, there exists $M >0$ such that
\begin{align*}
\Biggl|
&\langle u,
\m_{(x^{(1)},\dots,x^{(n)},x^{(n+1)}+w,\dots,x^{(n+m)}+w)}(a_1,\dots,a_{n+m})\rangle\\
&\quad - 
\sum_{\D < M'}
\langle u, \m_{(x^{(1)},\dots,x^{(n)},w)}(a_1,\dots,a_n, \pr_\D \m_{(x^{(n+1)},\dots,x^{(n+m)})}(a_{n+1},\dots,a_{n+m})\rangle
 \Bigr|<\ep
\end{align*}
holds for any $(x^{(1)},\dots,x^{(n)},w,x^{(n+1)},\dots,x^{(n+m)}) \in K$ and $M' \geq M$.
%Let $(x^{(1)},\dots,x^{n},w)\in \Conf_{n+1}(\R^d)$ and $(x^{(n+1)},\dots,x^{(n+m)}) \in  \Conf_m(\R^d)$ satisfy
%\begin{align*}
%\min_{k \in [n]}|x^{(k)}-w|>\max_{l \in [m]} |x^{(n+l)}| \Bigr\}.
%\end{align*}
%Then, for any $n \geq 0$ and $m \geq 1$ and $a_i \in H$ ($i=1,\dots,n+m$),
%% and $u\in H^\vee$ and $(x_1,\dots,x_{n+m},w) \in U_{n,m}$,
%\begin{align}
%\sum_{\D \in \R}
%\langle u, \m_{(x^{(1)},\dots,x^{(n)},w)}(a_1,\dots,a_n, \pr_\D \m_{(x^{(n+1)},\dots,x^{(n+m)})}(a_{n+1},\dots,a_{n+m})\rangle
%\label{eq_cluster}
%\end{align}
%is absolutely convergent and \eqref{eq_cluster} is convergent to 
%\begin{align*}
%\m_{(x^{(1)},\dots,x^{(n)},x^{(n+1)}+w,\dots, x^{(n+m)}+w)}(a_1,\dots,a_{n+m}).
%\end{align*}
\end{description}
\end{dfn}

%It is also natural to impose locally uniform convergence, in addition to
%absolute convergence, in the definition, as follows. Indeed, this stronger cluster decomposition is satisfied, for example, in
%two-dimensional rational conformal field theories.
%定義において、以下のように絶対収束だけでなく広義一様収束を課すことも自然である。

%\begin{description}
%\item[uniform cluster decomposition)]
%For any $n \geq 0$ and $m \geq 1$ and $a_i \in H$ ($i=1,\dots,n+m$) and $u\in H^\vee$ and $(x_1,\dots,x_{n+m},w) \in U_{n,m}$, \eqref{eq_cluster} is absolutely convergent and moreover \eqref{eq_cluster} is locally uniformly convergent to 
%\begin{align*}
%\m_{(x^{(1)},\dots,x^{(n)},x^{(n+1)}+w,\dots, x^{(n+m)}+w)}(a_1,\dots,a_{n+m}),
%\end{align*}
%that is, for any compact subset $K \subset U_{n,m}$ and $\ep >0$, there exists $M >0$ such that
%\begin{align*}
%\left|
%\langle u,
%\m_{(x_{[n]},x_{[m]}+w)}\rangle -
%\sum_{\D < M'}
%\langle u, \m_{(x_{[n]},w)}(a_1,\dots,a_n, \pr_\D \m_{(x_{[n+1]})}(a_{n+1},\dots,a_{n+m})\rangle
% \right|<\ep
%\end{align*}
%holds for any $(x_{[n+m]},w) \in K$ and $M' \geq M$.
%\end{description}
%\begin{rem}\label{rem_cluster_HK}
%平坦な unit disk の共形埋め込みのなす operad の代数を用いて、strong cluster decompositon は幾何学的な解釈を持つ。このことについては\cite{M}を参照。
%\end{rem}

Let $(H,\m,\va)$ be a $d$-CC system.
Let $a_1,\dots,a_n \in H$.
Then, $\exp(-P_\mu x_\mu^{(n)})\m(a_1,\dots,a_n)$ is real analytic on $\Conf_n(\R^d)$.
%We first note that $\exp(-\sum_\rho x_2^\rho P_\rho) \m_{(x_1,x_2)}(a,b)$ is a well-defined real analytic function on $\Conf_2(\R^d)$ since $P_\rho$ is an operator of degree $1$ with respect to $D$, from (C2) for any $u \in H^\vee$ the following sum for $k \geq 0$ is finite:
%\begin{align*}
In fact, since $P_\mu^*$ is an operator of degree $-1$ with respect to $D$ on $H^\vee$, from (C2) for any $u \in H^\vee$,
\begin{align*}
u(\exp(-P_\mu x_\mu^{(n)})\bullet) &= \sum_{k_\mu \geq 0}
\prod_{\mu=1}^d\left(\frac{1}{k_\mu!}(-x_\mu^{(n)})^{k_\mu} (P_\mu^{k_\mu})^*u\right)(\bullet)
\end{align*}
is a finite sum. The following lemma is useful:
\begin{lem}\label{lem_translation}
Let $(H,\m,\va)$ be a $d$-CC system, $u\in H^\vee$ and $a_1,\dots,a_n \in H$. Then, the real analytic function $F(x^{(1)},\dots,x^{(n)}) = \langle \exp(-P_\mu^* x_\mu^{(n)}) u, \m(a_1,\dots,a_n)\rangle$ on $\Conf_n(\R^d)$  satisfies
\begin{align*}
F(x^{(1)},\dots,x^{(n)})=F(x^{(1)}+v,\dots,x^{(n)}+v)
\end{align*}
for any $v \in \R^d$, that is, $F$ is translation invariant.
\end{lem}
\begin{proof}
By translation invariance, for any $\mu=1,\dots,d$,
\begin{align*}
&\left(\frac{d}{dx_\mu^{(1)}}+\dots + \frac{d}{dx_\mu^{(n)}}\right)F(x^{(1)},\dots,x^{(n)})\\
&=\left(\frac{d}{dx_\mu^{(1)}}+\dots + \frac{d}{dx_\mu^{(n)}}\right)\langle \exp(-P_\mu^* x_\mu^{(n)}) u, \m(a_1,\dots,a_n)\rangle\\
&=
-\langle \exp(-P_\mu^* x_\mu^{(n)})u, P_\mu \m(a_1,\dots,a_n)\rangle
+\sum_{p=1}^n \langle \exp(-P_\mu^* x_\mu^{(n)})u, \m(a_1,\dots,P_\mu a_p,\dots, a_n)\rangle,
\end{align*}
which is zero by conformal invariance. Hence, the assertion holds.
\end{proof}

It is enough to verify the strong cluster decomposition in the case where
the gluing point \(w\) is equal to \(0\). In practice, this gives a more
convenient form for constructing \(d\)-CC systems.
%まずはじめに strong cluster decomposition は、gluing をする点$w$が、$w=0$の場合に示せば十分であることを見る。これは $d$-CC system を構成する上でより使い易い形である。
\begin{prop}\label{prop_scd0}
Let $(H,\m)$ in Definition \ref{def_conformal_algebra} satisfy the analyticity, conformal covariance and translation invariance. Set
\begin{align*}
U_{n,m}^0 &= \{(x^{(1)},\dots,x^{(n+m)})\in \mathrm{Conf}_{n+m}(\R^d) \mid \min_{i = 1,\dots,n} |x^{(i)}| > \max_{j=1,\dots, m} |x^{(n+j)}|.
\end{align*}
Assume that for any $n,m \geq 1$ and $a_i \in H$ ($i=1,\dots,n+m$) and $u\in H^\vee$ and $x_{[n+m]} \in U_{n,m}^0$,
\begin{align}
\sum_{\D \in \R}
\langle u, \m_{(x^{(1)},\dots,x^{(n)},0)}(a_1,\dots,a_n, \pr_\D \m_{(x^{(n+1)},\dots,x^{(n+m)})}(a_{n+1},\dots,a_{n+m})\rangle
\label{eq_cluster0}
\end{align}
is absolutely convergent and \eqref{eq_cluster0} is  convergent to $\langle u,\m_{(x^{(1)},\dots,x^{(n+m)})}(a_1,\dots,a_{n+m})\rangle$
with respect to the sup-norm on any compact subset $K \subset U_{n,m}^0$.
Then, the strong cluster decomposition holds.
\end{prop}
\begin{proof}
Let $(p_1,\dots,p_{n+m},q) \in U_{n,m}$ and set
\begin{align*}
v_\Delta = \pr_\D \m_{(p_{n+1},\dots,p_{n+m})}(a_{n+1},\dots,a_{n+m})\in H.
\end{align*}

By Lemma \ref{lem_translation}, for each $\D\in\R$, we have
\begin{align}
\begin{split}
e^{-\sum_\mu P_\mu q_\mu}\m_{(p_1,\dots,p_n,q)}(a_1,\dots,a_n, v_\D)\rangle
&=\m_{(p_1-q,\dots,p_n-q,0)}(a_1,\dots,a_n, v_\D)\rangle.
\end{split}
\label{eq_trans_cluster}
\end{align}
Set $p'_i=p_i-q$ for $i\in [n]$ and $p'_{n+j}=p_{n+j}$ for $j \in [m]$, which satisfies 
\begin{align*}
(p'_1,\dots,p'_{n+m}) \in U_{n,m}^0.
\end{align*}
Hence, by the assumption, for any $u\in H^\vee$,
\begin{align*}
\sum_{\D} \langle u, (\Pi_{\mu=1}^d P_\mu^{k_\mu})\m_{(p_1-q,\dots,p_n-q,0)}(a_1,\dots,a_n, v_\D)\rangle
\end{align*}
is absolutely convergent for any $k_1,\dots,k_d \geq 0$. Since $P_\mu^*$ is locally nilpotent, \eqref{eq_cluster} is absolutely convergent and
\begin{align*}
\sum_{\D}&\langle u,\m_{(p_1,\dots,p_n,q)}(a_1,\dots,a_n, v_\D)\rangle\\
&=\sum_{\D} \langle u, e^{\sum P_\mu q_\mu}\m_{(p_1-q,\dots,p_n-q,0)}(a_1,\dots,a_n, \pr_\D \m_{(p_{n+1},\dots,p_{n+m})}(a_{n+1},\dots,a_{n+m}))\rangle\\
&=\langle u, e^{\sum P_\mu q_\mu}\m_{(p_1-q,\dots,p_n-q,p_{n+1},\dots,p_{n+m})}
(a_1,\dots,a_{n+m}))\rangle\\
&=\langle u, \m_{(p_1,\dots,p_n,p_{n+1}+q,\dots,p_{n+m}+q)}
(a_1,\dots,a_{n+m}))\rangle,
\end{align*}
where we used Lemma \ref{lem_translation} in the last line.
To prove the uniform convergence, it suffices to show that there is a compact neighborhood $(p_1,\dots,p_{n+m},q) \in \tilde{K}$ in $U_{n,m}$ such that \eqref{eq_cluster} is uniformly convergent in $\tilde{K}$.
By the assumption, there is a compact neighborhood $K' \subset U_{n,m}^0$
with $(p_1',\dots,p_{n+m}') \in K'$ such that 
\begin{align*}
\lim_{M \to \infty}\sup_{x'_{[n+m]} \in K'} 
|\langle u,\m_{x_{[n+m]}'(a_{[n+m]})}\rangle -
\sum_{\D <M}
\langle u, \m_{(x_1',\dots,x_n',0)}(a_1,\dots,a_n, \pr_\D \m_{(x_{n+1}',\dots,x_{n+m}')}(a_{n+1},\dots,a_{n+m})\rangle|=0.
%S_{n+m}(z'_1,\dots,z'_{n+m}) - \sum_{h,\h <H}\sum_{i \in I_h}S_n(\cdots,e_i;z'_1,\dots,z'_n,0)S_m(\cdots,e^i,z'_{n+1},\dots,z'_{n+m})|=0.
\end{align*}
For $r \in \R^d$ and $(s_1,\dots,s_{n+m}) \in (\R^d)^{n+m}$, set
\begin{align*}
T_r(s_1,\dots,s_{n+m})=(s_1-r,\dots,s_n-r,s_{n+1},\dots,s_{n+m}) \in (\R^d)^{n+m}.
\end{align*}
Take a compact neighborhood $L$ of $(p_1,\dots,p_{n+m})$ and an open subset $V$ of $K'$ such that $T_q(L) \subset V \subset \overline{V} \subset K'$.
Since $L$ is compact and $V$ is open, there is $\ep >0$ such that 
\begin{align}
T_r(L) \subset V
\label{eq_LVK}
\end{align}
for any $r \in B_\ep(q) = \{x\in\R^d \mid |x-q|\leq  \ep \}$.
Set $\tilde{K}= L \times B_{\ep}(q)$, which is a compact neighborhood of $(p_1,\dots,p_{n+m},q)$ in $\Conf_{n+m}(\R^d) \times \R^d$.
By \eqref{eq_LVK}, 
for each $k_\mu \geq 0$, $\sum_{\D\in\R} 
\langle \Pi_\mu(P_\mu^*)^{k_\mu} u, \m_{(x_1-w,\dots,x_n-w,0)}(a_1,\dots,a_n, \pr_\D \m_{(x_{n+1},\dots,x_{n+m})}(a_{n+1},\dots,a_{n+m})\rangle
$ uniformly convergent to $
\langle \Pi_\mu(P_\mu^*)^{k_\mu} u, \m_{(x_1-w,\dots,x_n-w,x_{n+1},\dots,x_{n+m})}(a_{[n+m]})\rangle$ in $(x_1,\dots,x_{n+m},w)\in\tilde{K}$.
%Hence, \eqref{eq_cluster} convergent uniformly to $\langle u,\m_{x_{[n+m]}}(a_{[n+m]})\ranlge$ in $\tilde{K}$ since
%\begin{align*}
Hence, the assertion follows from
\begin{align*}
\langle u, \m_{x_1,\dots,x_n,x_{n+1}+w,\dots,x_{n+m}+w}(a_{[n+m]})\rangle
&=\sum_{k_\mu \geq 0} \left(\Pi_\mu\frac{(w^\mu)^{k_\mu}}{k_\mu!}\right)
\langle \Pi_\mu(P_\mu*)^{k_\mu} u, \m_{x_1-w,\dots,x_n-w,x_{n+1},\dots,x_{n+m}}(a_{[n+m]})\rangle.
\end{align*}

\end{proof}

Let $(H,\m,\va)$ be a $d$-CC system.
We give some immediate consequences of conformal covariance. Then, we have:
\begin{prop}\label{prop_covariance_explicit}
For any $\ar \in H^{\otimes n}$, the following properties hold:
\begin{align*}
P_\mu \m_{\xr}(\ar) &= \left(\sum_{q =1}^n \frac{d}{dx_\mu^{(q)}} \right)\m_{\xr}(\ar)\\
J_{\mu,\nu} \m_{\xr}(\ar) 
&=\sum_{q=1}^n \left(\m_{\xr}((J_{\mu,\nu})\cdot_q\ar)+  \left(x_\nu^{(q)} \frac{d}{dx_\mu^{(q)}}-x_\mu^{(q)} \frac{d}{dx_\nu^{(q)}}\right) \m_{\xr}(\ar)\right)\\
D \m_{\xr}(\ar) &=\sum_{q =1}^n \left(\m_{\xr}(D\cdot_q \ar)+E(x^{(q)}) \m_{\xr}(\ar)\right)\\
K_\mu \m_{\xr}(\ar)
 &=\sum_{q =1}^n 
 \m_{\xr}((K_\mu+2x_\mu^{(q)} D+2 \sum_{\rho}x_\rho^{(q)} J_{\rho,\mu})\cdot_q \ar)\\
&+ 
\sum_{q =1}^n\left(2x_\mu^{(q)} E(x^{(q)})-||x^{(q)}||^2 \frac{d}{dx_q^\mu}\right)\m_{\xr}(\ar),
\end{align*}
where $E(x^{(q)}) = \sum_{\rho} x_\rho^{(q)} \frac{d}{dx_\rho^{(q)}}$ is the Euler operator on the $q$-th point.
%In particular,
%\begin{align*}
%R^D\m_{\xr}(\ar) &=\m_{(Rx_1,\dots,Rx_r)}(R^D a_1,\dots,R^D a_r)
%\end{align*}
%for any $R >0$, where $R^D a =R^\D a$ for $a\in H_\D$.
\end{prop}

By Remark \ref{rem_chiral_cov}, we have:
%We end this section by explicitly giving the conformal covariance in the case $d=2$.
\begin{cor}\label{cor_cov_2d}
In the case of $d=2$, under the assumption of translation invariance, the conformal invariance is equivalent to
\begin{align*}
L(-1) \m_{\zr}(\ar) &= \left(\sum_{q =1}^n \frac{d}{dz_q} \right)\m_{\zr}(\ar)\\
\Ld(-1) \m_{\zr}(\ar) &= \left(\sum_{q =1}^n \frac{d}{d\z_q} \right)\m_{\zr}(\ar)\\
L(0) \m_{\zr}(\ar) &= \sum_{q =1}^n \left( z_q\frac{d}{dz_q}\m_{\zr}(\ar)+\m_{\zr}(L(0)\cdot_q\ar) \right)\\
\Ld(0) \m_{\zr}(\ar) &= \sum_{q =1}^n \left( \z_q\frac{d}{d\z_q}\m_{\zr}(\ar)+\m_{\zr}(\Ld(0)\cdot_q\ar) \right)\\
L(1) \m_{\zr}(\ar) &= \sum_{q =1}^n \left( z_q^2\frac{d}{dz_q}\m_{\zr}(\ar)+2z_q\m_{\zr}(L(0)\cdot_q\ar)+\m_{\zr}(L(1)\cdot_q\ar) \right)\\
\Ld(1) \m_{\zr}(\ar) &= \sum_{q =1}^n \left( \z_q^2\frac{d}{d\z_q}\m_{\zr}(\ar)+2\z_q\m_{\zr}(\Ld(0)\cdot_q\ar)+\m_{\zr}(\Ld(1)\cdot_q\ar) \right)
\end{align*}
for any $\zr=(z_1,\dots,z_n) \in \Conf_n(\C)$.
\end{cor}
This covariance is well known in two-dimensional chiral and full conformal
field theories; see, for example, \cite{FHL,M8}.

\subsection{From CC system to vertex operator}\label{sec_from_d}

Let $(H,\m,\va)$ be a $d$-CC system.

For $x \in \R^d \setminus \{0\}$, define a linear map $\tilde{Y}(-,x)-: H\otimes H \rightarrow \F$ by
\begin{align*}
\tilde{Y}(a,x)b = \m_{(x,0)}(a,b) \in \F,
\end{align*}
which is real analytic for $x \in \R^d \setminus \{0\}$.
There is essentially no loss of information in putting $x^{(2)}=0$ from the following lemma.
\begin{lem}\label{lem_trans_2}
As real analytic functions, 
%for any $\mu =1,\dots,d$,
%\begin{align*}
%\left(\frac{d}{dx_\mu^{(1)}}+\frac{d}{dx_\mu^{(2)}}\right)\left(
%\exp(-\sum_\rho x_\rho^{(2)} P_\rho) \m_{(x^{(1)},x^{(2)})}(a,b)\right)=0
%\end{align*}
%and 
\begin{align}
\m_{(x^{(1)},x^{(2)})}(a,b) = \exp(\sum_\rho x_\rho^{(2)} P_\rho) \tilde{Y}(a,x^{(1)}-x^{(2)})b.\label{eq_M_Y_formula}
\end{align}
\end{lem}
\begin{proof}
%We first note that $\exp(-\sum_\rho x_2^\rho P_\rho) \m_{(x_1,x_2)}(a,b)$ is a well-defined real analytic function on $\Conf_2(\R^d)$ since $P_\rho$ is an operator of degree $1$ with respect to $D$, from (C2) for any $u \in H^\vee$ the following sum for $k \geq 0$ is finite:
%\begin{align*}
%&\langle u, \exp(-\sum_\rho x_2^\rho P_\rho) \m_{(x_1,x_2)}(a,b)\rangle=\sum_{k \geq 0}\langle u, \frac{1}{k!}\left(-\sum_\rho x_2^\rho P_\rho\right) \m_{(x_1,x_2)}(a,b)\rangle.
%\end{align*}
%%$P_\rho$は次数1の作用素であるから(C2)よりfor any $u \in F^\vee$式1における$k$についての和は有限である。
%By the translation invariance and the conformal invariance, we have
%\begin{align*}
%&\left(\frac{d}{dx_\mu^{(1)}}+\frac{d}{dx_\mu^{(2)}}\right)
%\exp\left(-\sum_\rho x_\rho^{(2)} P_\rho\right) \m_{(x^{(1)},x^{(2)})}(a,b)\\
%&=- \exp\left(-\sum_\rho x_\rho^{(2)} P_\rho\right) P_\mu \m_{(x^{(1)},x^{(2)})}(a,b)+
%\exp\left(-\sum_\rho x_\rho^{(2)} P_\rho\right) \m_{(x^{(1)},x^{(2)})}(P_\mu a,b)\\
%&+ \exp\left(-\sum_\rho x_\rho^{(2)} P_\rho\right) \m_{(x^{(1)},x^{(2)})}(a,P_\mu b)=0,
%\end{align*}
By Lemma \ref{lem_translation}, $\exp(-\sum_\rho x_\rho^{(2)} P_\rho) \m_{(x^{(1)},x^{(2)})}(a,b)$ is translation invariant. Hence, it coincides with $\tilde{Y}(a,x^{(1)}-x^{(2)})b$ as real analytic functions.
\end{proof}

By the conformal invariance, we have
\begin{align}
\begin{split}
P_\mu \tilde{Y}(a,x)b &= \tilde{Y}(P_\mu a,x)b +\tilde{Y}(a,x)P_\mu b\\
J_{\mu,\nu} \tilde{Y}(a,x)b
&=\tilde{Y}(J_{\mu,\nu}a,x)b+ \tilde{Y}(a,x)J_{\mu,\nu} b+
\left(x_\nu \frac{d}{dx_\mu}-x_\mu \frac{d}{dx_\nu}\right)\tilde{Y}(a,x)b\\
D \tilde{Y}(a,x)b &=
\tilde{Y}(D a,x)b+\tilde{Y}(a,x) Db+E(x)\tilde{Y}(a,x)b\\
K_\mu \tilde{Y}(a,x)b
 &= \tilde{Y}(K_\mu + 2x_\mu D+2 \sum_\rho x^\rho J_{\rho,\mu}a,x)b + \tilde{Y}(a,x)K_\mu b+\left(2x_\mu E(x)-||x||^2 \frac{d}{dx_\mu}\right)\tilde{Y}(a,x)b
\end{split}
\label{eq_conf_inv0}
\end{align}
and by the translation invariance
\begin{align}
\frac{d}{dx_\mu}\tilde{Y}(a,x)b = \tilde{Y}(P_\mu a,x)b.
\label{eq_trans_inv0}
\end{align}

Let $\D_0,\D_1,\D_2\in \R$ and $u_0 \in H_{\D_0}^*$, $a \in H_{\D_1}$ and $b\in H_{\D_2}$.
By \eqref{eq_conf_inv0},
\begin{align}
E(x) \langle u_0, \tilde{Y}(a,x)b \rangle = (\Delta_0 -\D_1-\D_2) \langle u_0, \tilde{Y}(a,x)b \rangle.
\label{eq_vertex_E}
\end{align}
Hence, $\langle u_0, \tilde{Y}(a,x)b \rangle ||x||^{\D_1+\D_2-\D_0}$ is a scale invariant real analytic function on $\R^d\setminus\{0\}$.
We may regard it as a linear map
\begin{align}
C:H_{\D_0}^*\otimes H_{\D_1}\otimes H_{\D_2}\rightarrow C^\infty(S^{d-1},\C),\quad\quad (u_0,a,b)\mapsto \langle u_0,\tilde{Y}(a,x)b \rangle.\label{eq_sod_HHS}
\end{align}
%where $C^\infty(S^{d-1},\C)$ is the vector space of smooth functions on $(d-1)$-dimensional sphere $S^{d-1}$.
%Note that $C^\infty(S^{d-1},\C)$ is naturally a $\mathfrak{so}(d)$-module (see Proposition \ref{prop_rep_func_space})
%and the left-hand-side of \eqref{eq_sod_HHS} is also a (tensor product) representation of $\mathfrak{so}(d)$.
Since $H_\D$ is finite dimensional by (C2)
and \eqref{eq_sod_HHS} is a $\mathfrak{so}(d)$-module homomorphism by \eqref{eq_conf_inv0}, by Proposition \ref{prop_analytic_Laurent},
we can naturally regard $\pr_{\Delta_0} \tilde{Y}(a,x)b$ as an element in
\begin{align*}
H[x_1,\dots x_d,|x|^\R]_{\Delta_0-\Delta_1-\Delta_2}
\end{align*}
and $\tilde{Y}(a,x)b$ as
\begin{align*}
H[[x_1,\dots x_d,|x|^\R]].
\end{align*}
Denote this linear map $H \otimes H \rightarrow H[[x_1,\dots d_x,|x|^\R]]$ by $Y(\bullet,x)\bullet$.
Moreover, by (C2), we have:
\begin{lem}\label{lem_vertex_bound}
For any $a,b \in H$, $Y(a,x)b \in H((x_1,\dots x_d,|x|^\R))$.
\end{lem}

To summarize, by Lemma \ref{lem_trans_2}, \eqref{eq_conf_inv0}, \eqref{eq_trans_inv0}, Lemma \ref{lem_vertex_bound}, we have:
\begin{prop}
\label{prop_vertex}
Let $(H,\m,\va)$ be a $d$-CC system. 
There exists a unique linear map
\begin{align*}
Y(-,x):H \otimes H \rightarrow H((x_1,\dots,x_d,|x|^\R))
\end{align*}
such that for any $u_0^* \in H_{\D_0}^*$, $a \in H_{\D_1}$ and $b\in H_{\D_2}$,
\begin{align}
\langle u_0^*,Y(a,x)b\rangle \in |x|^{\Delta_0-\Delta_1-\D_2}\C\left[\frac{x_1}{|x|},\dots,\frac{x_d}{|x|}\right]
\label{eq_just_polynomial}
\end{align}
and
\begin{align*}
Y(a,x)b|_{x=x_1} = \m_{(x_1,0)}(a,b)
\end{align*}
holds as real analytic functions on $\R^d \setminus \{0\}$, where we regard the left-hand-side as a real analytic function by \eqref{eq_polar_regard}. Furthermore, $Y(\bullet,x)$ is a conformally covariant vertex operator 
 and it satisfies
%\begin{align*}
%\frac{d}{dx^\mu} Y(a,x) = Y(P_\mu a,x)
%\end{align*}
%for any $\mu=1,\dots,d$ and
\begin{align*}
\m_{(x^{(1)},x^{(2)})}(a,b)
=
\exp\left(\sum_\rho x_\rho^{(2)}P_\rho\right)
Y(a,x)b|_{x=x^{(1)}-x^{(2)}}.
%\m_{(x_1,x_2)}(a,b) = \exp(\sum_\rho x_\rho^{(2)} P_\rho) Y(a,x_1-x_2)b.
\end{align*}
\end{prop}

We call the linear map $Y(-,x)$ a \textbf{($d$-dimensional) vertex operator} associated with the $d$-CC system $(H,\m,\va)$.

\subsection{Definition of conformal $d$-vertex algebra}
\label{sec_def_vertex}
In this section, we examine the relationship between $Y(-,x)$ and correlation functions.
Applying the cluster decomposition for $(n,m)=(1,2)$, $w=0$, $x^{(3)}=0$ and $a_i \in H_{\D_i}$ ($i=1,2,3$), we have:
\begin{align}
\m_{(x^{(1)},x^{(2)},0)}(a_1,a_2,a_3) 
&= \sum_{\D \in \R}\m_{(x^{(1)},0)} \left(a_1,\pr_\D \m_{(x^{(2)},0)}(a_2,a_3) \right),
\label{eq_from_B1}
\end{align}
where the right-hand side is absolutely convergent in
\begin{align}
\{|x^{(1)}|>|x^{(2)}|>0\} = \{(x^{(1)},x^{(2)}) \in (\R^d)^2 \mid |x^{(1)}|>|x^{(2)}|>0\}.
\label{eq_domain_add_1}
\end{align}
Note that by \eqref{eq_just_polynomial} in  Proposition \ref{prop_vertex} for any $u_0 \in H_{\D_0}^*$,
$\langle u_0,\m_{(x^{(1)},0)} \left(a_1,\pr_\D \m_{(x^{(2)},0)}(a_2,a_3) \right)\rangle$ is in
\begin{align*}
\left( \frac{|x^{(2)}|}{|x^{(1)}|} \right)^\Delta
 |x^{(1)}|^{\Delta_0-\Delta_1}|x^{(2)}|^{-\Delta_2-\D_3}\C\left[\frac{x_1^{(1)}}{|x^{(1)}|},\dots,\frac{x_d^{(1)}}{|x^{(1)}|},
\frac{x_1^{(2)}}{|x^{(2)}|},\dots,\frac{x_d^{(2)}}{|x^{(2)}|}
\right],
\end{align*}
and thus, is a real analytic function on \eqref{eq_domain_add_1}.
Applying the cluster decomposition for $(n,m)=(1,2)$, we also have:
\begin{align*}
\m_{(0,x^{(1)},x^{(2)})}(a_3,a_1,a_2) 
&= \sum_{\D \in \R}\m_{(0,x^{(2)})} \left(a_3,\pr_\D \m_{(x^{(1)}-x^{(2)},0)}(a_1,a_2) \right),
\end{align*}
where the right-hand side is absolutely convergent in
\begin{align*}
\{|x^{(2)}|>|x^{(1)}-x^{(2)}|>0\}= \{(x^{(1)},x^{(2)}) \in (\R^d)^2 \mid |x^{(2)}|>|x^{(1)}-x^{(2)}|>0\}.
\end{align*}
%Hence, by setting $w=x_2$, we have:
%\begin{align*}
%\m_{(0,x_1,x_2)}(a_3,a_1,a_2) 
%&= \sum_{\D \in \R}\m_{(0,x_2)} \left(a_3,\pr_\D \m_{(x_1-x_2,0)}(a_1,a_2) \right),
%\end{align*}
%in $\{(x_1,x_2) \in (\R^d)^2 \mid |x_2|>|x_1-x_2|>0\}$.
By the permutation invariance, we have:
\begin{align}
\m_{(x^{(1)},x^{(2)},0)}(a_1,a_2,a_3)
&= \sum_{\D \in \R}\m_{(x^{(2)},0)} \left(\pr_\D \m_{(x^{(1)}-x^{(2)},0)}(a_1,a_2),a_3 \right).
\label{eq_from_B2}
\end{align}

Let $(x_\mu^{(0)})_{\mu =1,\dots,d}$ be formal variables.
Note that by \eqref{eq_just_polynomial} in 
Proposition \ref{prop_vertex}, the composition of vertex operators
\begin{align*}
Y\left(Y(a_1,x^{(0)})a_2,x^{(2)}\right)a_3
\end{align*}
is formally well-defined since 
$\langle u_0, Y\left(\pr_\D Y(a_1,x^{(0)})a_2,x^{(2)}\right)a_3\rangle $ for $u_0 \in H_{\D_0}^*$ and $\D\in \R$ is in 
\begin{align}
\left( \frac{|x^{(0)}|}{|x^{(2)}|} \right)^\Delta
 |x^{(0)}|^{-\Delta_1-\Delta_2}|x^{(2)}|^{\Delta_0-\D_3}\C\left[\frac{x_1^{(2)}}{|x^{(2)}|},\dots,\frac{x_d^{(2)}}{|x^{(2)}|},\frac{x_1^{(0)}}{|x^{(0)}|},\dots,\frac{x_d^{(0)}}{|x^{(0)}|}
\right]
\label{eq_D_vertex_comp}
\end{align}
(The degree of $|x^{(0)}|$ is $\Delta - (\Delta_1+\Delta_2)$).
Moreover, by \eqref{eq_D_vertex_comp}, $\langle u_0, Y\left(\pr_\D Y(a_1,x^{(0)})a_2,x^{(2)}\right)a_3\rangle$ is a real analytic function on
\begin{align*}
\{(x^{(1)},x^{(2)}) \in (\R^d)^2 \mid x^{(1)}-x^{(2)} \neq 0 \text{ and } x^{(2)}\neq 0\}
\end{align*}
by $x^{(0)}=x^{(1)}-x^{(2)}$, which coincides with $\m_{(x^{(2)},0)} \left(\pr_\D \m_{(x^{(1)}-x^{(2)},0)}(a_1,a_2),a_3 \right)$ as real analytic functions.
Set
\begin{align*}
Y_2(\R^d)=\{(x^{(1)},x^{(2)}) \in (\R^d)^2 \mid x^{(1)} \neq 0, x^{(2)}\neq 0, x^{(1)}-x^{(2)} \neq 0\}.
\end{align*}
Then, we have:
\begin{prop}\label{prop_borcherds}
Let $(H,\m,\va)$ be a $d$-CC system and $Y(-,x):H \otimes H \rightarrow H((x_1,\dots,x_d,|x|^\R))$ the associated vertex operator. Then, for any $u_0 \in H^\vee$ and $a_1,a_2,a_3 \in H$, the sums $\sum_{\Delta\in \R} \langle u_0,Y(a_1,x^{(1)}) \pr_\D Y(a_2,x^{(2)})a_3 \rangle$
and $\sum_{\Delta\in \R} \langle u_0,Y(\pr_\D Y(a_1,x^{(0)})a_2,x^{(2)})a_3 \rangle$ are 
absolutely convergent if $|x^{(1)}|>|x^{(2)}|>0$ and $|x^{(2)}|>|x^{(0)}|>0$, respectively.
Moreover, there exists a real analytic function $\phi(x^{(1)},x^{(2)}):Y_2(\R^d)\rightarrow \C$
such that 
\begin{enumerate}
\item The sums
$\sum_{\Delta\in \R} \langle u_0,Y(a_1,x^{(1)}) \pr_\D Y(a_2,x^{(2)})a_3 \rangle$ and $\sum_{\Delta\in \R} \langle u_0,Y(a_2,x^{(2)}) \pr_\D Y(a_1,x^{(1)})a_3 \rangle$ are locally uniformly convergent to $\phi(x^{(1)},x^{(2)})$ in $|x^{(1)}|>|x^{(2)}|>0$ and $|x^{(2}|>|x^{(1)}|>0$, respectively.
\item The sum $
\sum_{\Delta\in \R} \langle u_0,Y(\pr_\D Y(a_1,x^{(0)})a_2,x^{(2)})a_3 \rangle$ is locally uniformly convergent to $\phi(x^{(1)},x^{(2)})$ in $|x^{(2)}|>|x^{(1)}-x^{(2)}|>0$ with $x^{(0)}=x^{(1)}-x^{(2)}$.
\end{enumerate}
\end{prop}

We write (1) and (2) in Proposition \ref{prop_borcherds} simply as 
\begin{align*}
\langle u_0,Y(a_1,x^{(1)}) Y(a_2,x^{(2)})a_3 \rangle&= \phi(x^{(1)},x^{(2)})|_{|x^{(1)}|>|x^{(2)}|} \\
\langle u_0,Y(a_2,x^{(2)}) Y(a_1,x^{(1)})a_3 \rangle &= \phi(x^{(1)},x^{(2)})|_{|x^{(2)}|>|x^{(1)}|}\\
\langle u_0,Y(Y(a_1,x^{(0)})a_2,x^{(2)})a_3 \rangle &= \phi(x^{(0)}+x^{(2)},x^{(2)})|_{|x^{(2)}|>|x^{(0)}|}.
\end{align*}

\begin{prop}
\label{prop_vacuum}
Let $(H,\m,\va)$ be a $d$-CC system and $Y(-,x):H \otimes H \rightarrow H((x_1,\dots,x_d,|x|^\R))$ the associated vertex operator. Then, the following properties hold:
\begin{enumerate}
\item
$Y(\va,x) = \id_H$;
\item
For any $a,b \in H$, $Y(a,x)b = \exp(\sum_\rho P_\rho x_\rho)Y(b,-x)a$;
\item
For any $a\in H$, $Y(a,x)\va = \exp(\sum_\rho P_\rho x_\rho)a$ and in particular $\lim_{x \to 0}Y(a,x)\va =a$.
\end{enumerate}
\end{prop}
\begin{proof}
By the unit and the vacuum property,
\begin{align*}
Y(\va,x)a =\m_{(x,0)}(\va,a) = \m_{(0)}(a) =a.
\end{align*}
By setting $x^{(1)}=0$ in \eqref{eq_M_Y_formula}, we have:
\begin{align}
\exp(\sum_{\rho} x_\rho^{(2)} P_\rho)\tilde{Y}(a,-x^{(2)})b =  \m_{(0,x^{(2)})}(a,b)=\m_{(x^{(2)},0)}(b,a) = \tilde{Y}(b,x^{(2)})a,
\label{eq_vacuum_coin_tilde}
\end{align}
as real analytic function on $x^{(2)} \in \R^d\setminus \{0\}$.
By \eqref{eq_just_polynomial}, \eqref{eq_vacuum_coin_tilde} holds formally. Hence, (2) follows.
(3) follows from (1) and (2).
\end{proof}

\begin{dfn}\label{def_vertex_prealgebra}
A conformal $d$-vertex algebra is a compact $\sod$-module $H$ equipped with a conformally covariant linear map
\begin{align*}
Y(-,x): H \otimes H \rightarrow H((x_1,\dots,x_d,|x|^\R)),\quad\quad (a,b)\mapsto Y(a,x)b
\end{align*}
and $\va \in H_0$ such that:
\begin{enumerate}
\item[V1)]
%For any $a\in H$, $\frac{d}{dx} Y(a,x) = Y(P_\mu a,x)$;
$\C\va$ is a trivial representation of $\sod$, $Y(\va,x) = \id_H$ and for any $a\in H$, $Y(a,x)\va \in H[[x_1,\dots,x_d]]$ and $\lim_{x \to 0}Y(a,x)\va =a$;
\item[V2)]
For any $u_0 \in H^\vee$ and $a_1,a_2,a_3 \in H$, the sum $\sum_{\Delta\in \R} \langle u_0,Y(a_1,x^{(1)}) \pr_\D Y(a_2,x^{(2)})a_3 \rangle$
and $\sum_{\Delta\in \R} \langle u_0,Y(\pr_\D Y(a_1,x^{(0)})a_2,x^{(2)})a_3 \rangle$
are absolutely convergent if $|x^{(1)}|>|x^{(2)}|>0$
and $|x^{(2)}|>|x^{(0)}|>0$, respectively.
Moreover, there exists a real analytic function $\phi(x^{(1)},x^{(2)}):Y_2(\R^d)\rightarrow \C$
such that 
\begin{align}
\begin{split}
\langle u_0,Y(a_1,x^{(1)}) Y(a_2,x^{(2)})a_3 \rangle&= \phi(x^{(1)},x^{(2)})|_{|x^{(1)}|>|x^{(2)}|} \\
\langle u_0,Y(a_2,x^{(2)}) Y(a_1,x^{(1)})a_3 \rangle &= \phi(x^{(1)},x^{(2)})|_{|x^{(2)}|>|x^{(1)}|},\\
\langle u_0,Y(Y(a_1,x^{(0)})a_2,x^{(2)})a_3 \rangle &= \phi(x^{(0)}+x^{(2)},x^{(2)})|_{|x^{(2)}|>|x^{(0)}|},
\end{split}
\label{eq_def_V3}
\end{align}
in the sense of Proposition \ref{prop_borcherds}.
\end{enumerate}
\end{dfn}

By Proposition \ref{prop_vertex}, Proposition \ref{prop_borcherds} and Proposition \ref{prop_vacuum}, we have:
\begin{thm}\label{thm_from}
Let $(H,\m,\va)$ be a $d$-CC system  and $Y(-,x):H \otimes H \rightarrow H((x_1,\dots,x_d,|x|^\R))$ the associated vertex operator.
Then, $(H,Y(-,x),\va)$ is a conformal $d$-vertex algebra.
\end{thm}

%\subsection{Derivation of OPEs}\label{sec_OPE}
Hereafter, we will look at consequences from the axioms of conformal $d$-vertex algebra. All of the results here are generalizations of results known for chiral vertex algebras (see, for example, \cite[Section 3]{LL}).
The proof in the case of full vertex algebras (conformal $2$-vertex algebras) can be found in our previous paper \cite{M1}.
%この章では vertex $d$-algebra の公理からの様々な帰結を見る。この章の結果は全てchiral 頂点代数に対して知られて結果の拡張である(たとえば\cite[Section 3]{LL}を参照)。カイラル頂点代数の場合の証明の類似が vertex $2$-algebra に対しても適用できることは我々の以前の論文\cite{M1}にある。
%次の命題で使われるテクニックはカイラル頂点代数における結果の類似である (see \cite[Section 3]{LL} for chiral vertex algebra and \cite[Proposition 3.7]{M1} for the full case):
\begin{prop}\label{prop_vacuum_ind}
Let $(H,Y(-,x),\va)$ be a conformal $d$-vertex algebra. Then, for any $a\in H$, 
\begin{align*}
Y(a,x)\va = \exp\left(\sum_\rho P_\rho x^\rho\right)a
\end{align*}
holds.
\end{prop}
\begin{proof}
By conformal covariance and (V1),
\begin{align}
\frac{d}{dx^\mu} Y(a,x)\va = Y(P_\mu a,x)\va =[P_\mu,Y(a,x)]\va =P_\mu Y(a,x)\va.
\label{eq_skew_diff}
\end{align}
Since $Y(a,x)\va \in H[[x_1,\dots,x_d]]$ by (V1), by solving \eqref{eq_skew_diff} inductively with the initial condition
$\lim_{x\to 0}Y(a,x)\va =a$, we have the equality.
\end{proof}

\begin{prop}\label{prop_skew_symmetry}
Let $(H,Y(-,x),\va)$ be a conformal $d$-vertex algebra. Then, for any $a,b \in H$, 
\begin{align*}
Y(a,x)b = \exp\left(\sum_\rho P_\rho x^\rho\right)Y(b,-x)a
\end{align*}
formally holds.
\end{prop}
\begin{proof}
Let $a,b\in H$ and $u\in H^\vee$. By Proposition \ref{prop_vacuum_ind}, we have
\begin{align}
\langle u, Y(a,x^{(1)})Y(b,x^{(2)}) \va \rangle &= \langle u, Y(a,x^{(1)})\exp(\sum_\rho P_\rho x_\rho^{(2)})b\rangle.
\label{eq_above_exp_add}
\end{align}
By conformal covariance,
as formal power series, the right-hand side is formally equal to $\langle u, \exp\left(\sum_\rho P_\rho x_\rho^{(2)}\right)Y(a,x^{(1)}-x^{(2)})|_{|x^{(1)}|>|x^{(2)}|}b\rangle$.
Let \(x^{(0)}\) be formal variables and consider 
\begin{align}
\langle u, \exp\left(\sum_\rho P_\rho x_\rho^{(2)}\right)Y(a,x^{(0)})b\rangle
=\langle u, Y(Y(a,x^{(0)})b,x^{(2)})\va \rangle
 \label{eq_skew_x0}
\end{align}
Then \eqref{eq_above_exp_add} is obtained by substituting $x^{(0)}=(x^{(1)}-x^{(2)}) |_{|x^{(1)}|>|x^{(2)}|}$ into 
 \eqref{eq_skew_x0}.
By (V2), \eqref{eq_skew_x0} is equal to
\begin{align}
\begin{split}
\langle u, Y(b,x^{(2)})Y(a,x^{(1)}) \va \rangle &=
\langle u, \exp\left(\sum_\rho P_\rho x_\rho^{(1)}\right)Y(b,-x^{(0)})a\rangle,\\
 &=
\langle u, \exp\left(\sum_\rho P_\rho (x_\rho^{(2)}+x_\rho^{(0)})\right)Y(b,-x^{(0)})a\rangle
\end{split}
\label{eq_skew_x1}
\end{align}
as analytic functions.
By (C2), both \eqref{eq_skew_x0} and \eqref{eq_skew_x1} are in 
\begin{align*}
\C[x_1^{(2)},\dots,x_d^{(2)},x_1^{(0)},\dots,x_d^{(0)},|x^{(0)}|^\R].
\end{align*}
Thus, both sides are formally equal. Hence, the assertion holds.
\end{proof}

\subsection{Relation with vertex algebras}
\label{sec_rel_vertex}
In this section, we explain how vertex algebras arise from conformal \(2\)-vertex algebras when \(d=2\). The material in this section is
essentially due to \cite{M1}, but we include the details for the reader's convenience.
%この章では$d=2$において、$2$-vertex algebra からどのように vertex algebra が生じるかを説明する。この章の内容は本質的に\cite{}による。
Let $H$ be a conformal 2-vertex algebra.
Then, for each $\Delta \in \R$, $H_\Delta$ is a representation of $\mathrm{SO}(2)$. Set
\begin{align*}
H_{\Delta,s}= \{v \in H_\Delta\mid \rho(e^{i \theta})v =e^{is\theta}\}
\end{align*}
for $s\in\Z$. Then, we have:
$H_\Delta = \bigoplus_{s \in \Z} H_{\Delta,s}$.
Set
\begin{align*}
H_{h,\h} = H_{\Delta=h+\h, s=h-\h}.
\end{align*}
Then, $H=\bigoplus_{h,\h \in \R} H_{h,\h}$, which is precisely the simultaneous eigenspace decomposition with respect
to \(L(0)\) and \(\Ld(0)\).
For formal variables $x_1,x_2$, we set
\begin{align*}
z = x_1+ i x_2,\qquad &\z= x_1- i x_2\\
\frac{d}{dz}= \ft\left(\frac{d}{dx_1}-i \frac{d}{dx_2}\right)\qquad
&\frac{d}{d\z}= \ft\left(\frac{d}{dx_1}+i \frac{d}{dx_2}\right).
\end{align*}
Let $u\in H_{h_0,\h_0}^\vee$, $a \in H_{h,\h}$ and $b\in H_{h',\h'}$. Then,
$h_0\langle u, Y(a,x)b \rangle = \langle  u,L_0 Y(a,x)b \rangle = \langle  u,[L_0, Y(a,x)]b \rangle 
+\langle  u,Y(a,x)L_0 b \rangle$.
Hence, by the conformal covariance, we have
\begin{align*}
z\frac{d}{dz} \langle u, Y(a,x)b\rangle = (h_0-h_1-h_2) \langle u, Y(a,x)b\rangle,
\end{align*}
and similarly,
\begin{align*}
\z\frac{d}{d\z} \langle u, Y(a,x)b\rangle = (\h_0-\h_1-\h_2) \langle u, Y(a,x)b\rangle.
\end{align*}
This implies 
\begin{align*}
\langle u, Y(a,x)b\rangle \in \C z^{h_0-h_1-h_2}\z^{\h_0-\h_1-\h_2}.
\end{align*}
Hence, we have:
\begin{lem}\label{lem_2d_vertex}
For any $a\in H$, there are uniquely determined operators $a(r,s)\in \End H$ ($r,s\in\R$) such that
\begin{align*}
Y(a,x) = \sum_{r,s\in\R}a(r,s)z^{-r-1}\z^{-s-1}.
\end{align*}
Moreover, if $a\in H_{h,\h}$, then it satisfies
\begin{align}
a(r,s) H_{h',\h'} \subset H_{h+h'-r-1,\h+\h'-s-1}
\label{eq_deg_cov_lem}
\end{align}
for any $r,s\in\R$ and $h',\h' \in\R$.
\end{lem}
In view of this, in the two-dimensional case we shall write vertex
operators as
\begin{align*}
Y(a,z,\z) \in \End H[[z,\z,|z|^\R]].
\end{align*}
We note that, when considering compositions of vertex operators such as
\(Y(a,z_1,\z_1)Y(b,z_2,\z_2)c\), we assume convergence only for the sum in
the direction of \(D=L_0+\Ld_0\).
%このことから二次元では以下、我々は vertex operator を
%\begin{align*}
%Y(a,z,\z) \in \End H[[z,\z,|z|^\R]]
%\end{align*}
%とかくことにする。
%ここで $Y(a,z_1,\z_1)Y(b,z_2,\z_2)c$のような vertex operator の合成を考えるとき、我々は$D=L_0+\Ld_0$の方向に関する和の収束性しか仮定していないことに注意をしておく。

\begin{lem}
\label{lem_D_add}
Let $u \in H_{h_0,\h_0}^\vee$ and  $a_i \in H_{h_i,\h_i}$ ($i=1,2,3$) and $\phi(z_1,z_2)$ be the real analytic function on $Y_2(\R^2)$ in \eqref{eq_def_V3}.
Then, there is a real analytic function
$F: \C\setminus \{0,1\} \rightarrow \C$
such that
\begin{align*}
\phi(z_1,z_2) = z_1^{h_0-h_1-h_2-h_3}\z_1^{\h_0-\h_1-\h_2-\h_3}F\left(\frac{z_1}{z_2}\right)
\end{align*}
\end{lem}
\begin{proof}
Since
\begin{align*}
&h_0 \langle u, Y(a_1,z_1)Y(a_2,z_2)a_3 \rangle\\
&=
\langle L_0^* u, Y(a_1,z_1)Y(a_2,z_2)a_3 \rangle\\
&= \langle u, L_0 Y(a_1,z_1)Y(a_2,z_2)a_3 \rangle\\
&= \langle u, [L_0,Y(a_1,z_1)]Y(a_2,z_2)a_3 \rangle +
\langle u, Y(a_1,z_1)[L_0,Y(a_2,z_2)]a_3 \rangle
+ \langle u, Y(a_1,z_1)Y(a_2,z_2)L_0a_3 \rangle\\
&=(h_1+h_2+h_3) \langle u, Y(a_1,z_1)Y(a_2,z_2)a_3 \rangle +
\left( z_1\frac{d}{dz_1} + z_2\frac{d}{dz_2}\right)
\langle u, Y(a_1,z_1)Y(a_2,z_2)a_3 \rangle,
\end{align*}
we have:
\begin{align}
\left( z_1\frac{d}{dz_1} + z_2\frac{d}{dz_2}\right)\phi(z_1,z_2) = (h_0-h_1-h_2-h_3)\phi(z_1,z_2)
\label{eq_cov_D_add}
\end{align}
and similarly,
\begin{align}
\left( \z_1\frac{d}{d\z_1} + \z_2\frac{d}{d\z_2}\right)\phi(z_1,z_2) = (\h_0-\h_1-\h_2-\h_3)\phi(z_1,z_2).
\label{eq_cov_D_add2}
\end{align}
Hence,
\begin{align*}
z_1^{h_1+h_2+h_3-h_0}\z_1^{\h_1+\h_2+\h_3-\h_0}
\phi(z_1,z_2)
\end{align*}
is a real analytic function on \(Y_2(\R^2)\) which is invariant under
rotations and dilations.
Here we have used the fact that
\begin{align*}
z_1^{h_1+h_2+h_3-h_0}\z_1^{\h_1+\h_2+\h_3-\h_0} = |z_1|^{2(h_1+h_2+h_3-h_0)}\z_1^{(\h_1-h_1)+(\h_2-h_2)+(\h_3-h_3)-(\h_0-h_0)}
\end{align*}
is single-valued since $H_{h,\h}=0$ if $h-\h \notin \Z$. The desired assertion follows.
%
%Hence,
%\begin{align*}
%z_1^{h_1+h_2+h_3-h_0}\z_1^{\h_1+\h_2+\h_3-\h_0}
%\phi(z_1,z_2)
%\end{align*}
%は回転とdilation に対して不変な$Y_2(\R^2)$上の実解析的関数である。
%ここで
%\begin{align*}
%z_1^{h_1+h_2+h_3-h_0}\z_1^{\h_1+\h_2+\h_3-\h_0} = (x_1^{(1)}+x_2^{(1)})^{h_1+h_2+h_3-h_0}\z_1^{(\h_1-h_1)+(\h_2-h_2)+(\h_3-h_3)-(\h_0-h_0)}
%\end{align*}
%が一価であることを用いた。よって題意が従う。
\end{proof}

A vector $v \in H$ is called a \textbf{chiral vector} (resp. \textbf{anti-chiral vector}) if $\Ld(-1)v=0$ (resp. $L(-1)v=0$) (see \cite{M1}).

Assume that $a \in H$ is a chiral vector.
Then, since $0= Y(\Ld(-1)a,z,\z) = \frac{d}{d\z} Y(a,z,\z)$, 
we have $Y(a,z,\z)b \in H((z))$ for any $b \in H$.
Hence, there are uniquely determined operators $a(n) \in \End H$ such that
\begin{align*}
Y(a,z,\z) = \sum_{n\in\Z}a(n)z^{-n-1},
\end{align*}
which we denote by $Y(a,z)$.
Similarly, if $b$ is an anti-chiral vector, then
\begin{align*}
Y(b,z,\z) = \sum_{n\in\Z}b(n)\z^{-n-1}.
\end{align*}
Moreover, by \eqref{eq_deg_cov_lem}, if $a \in H_{h,\h}$, then
\begin{align}
a(n)H_{h',\h'} \subset H_{h+h'-n-1,\h+\h'}.
\label{eq_anb_L}
\end{align}

\begin{prop}\cite[Lemma 2.10]{M1}
\label{lem_chiral_der}
Let $H$ be a conformal 2-vertex algebra
and $u \in H^\vee$, $a_1,a_2,a_3 \in H$.
Assume that $a_1$ is a chiral vector. Then,
$\phi(z_1,z_2)$ in \eqref{eq_def_V3} is a function in $\C[z_1^\pm,(z_1-z_2)^\pm, z_2,\z_2,|z_2|^\R]$. 
\end{prop}
\begin{proof}
We may assume that $u \in H_{h_0,\h_0}^\vee$ and $a_i \in H_{h_i,\h_i}$. Set
\begin{align*}
h=h_1+h_2+h_3-h_0.
\end{align*}
By Lemma \ref{lem_D_add}, there is a real analytic function $G:\C \setminus \{0,1\} \rightarrow \C$ such that
\begin{align*}
\phi(z_1,z_2) =z_2^{h_0-h_1-h_2-h_3}
\z_2^{\h_0-\h_1-\h_2-\h_3} G\left(\frac{z_1}{z_2}\right).
\end{align*}
Since $a_1$ is a chiral vector,
$
\frac{d}{d\z_1}G\left(\frac{z_1}{z_2}\right)=0$,
and thus, $G$ is a holomorphic function on $\C \setminus \{0,1\}$, which satisfies
\begin{align}
\begin{split}
z_2^{h}\z_2^{\h} \langle u, Y(a_1,z_1) Y(a_2,z_2,\z_2) a_3 \rangle &= G\left(\frac{z_1}{z_2}\right)\Bigl|_{|z_1|>|z_2|},\\
z_2^{h}\z_2^{\h} \langle u, Y(a_2,z_2,\z_2)Y(a_1,z_1)a_3 \rangle &= G\left(\frac{z_1}{z_2}\right)\Bigl|_{|z_2|>|z_1|},\\
z_2^{h}\z_2^{\h} \langle u, Y(Y(a_1,z_0)a_2,z_2,\z_2)a_3 \rangle &= G\left(\frac{z_0+z_2}{z_2}\right)\Bigl|_{|z_2|>|z_0|}.
\end{split}
\label{eq_pre_bor}
\end{align}
By definition, the sum
\[
\sum_{\Delta}
z_2^{h}\z_2^{\h} \langle u, Y(a_2,z_2,\z_2)\pr_\Delta Y(a_1,z_1)a_3 \rangle
\]
is absolutely convergent.
By \eqref{eq_deg_cov_lem}, this sum is equal to
\begin{align*}
\sum_{n\in \Z}
z_2^{h}\z_2^{\h} \langle u, Y(a_2,z_2,\z_2)a_1(n)a_3 \rangle&=
z_2^{h}\z_2^{\h}\sum_{n \in \Z} \langle u, a_2(h-n-2,\h-1)a_1(n)a_3 \rangle z_1^{-n-1}z_2^{-h+n-1}\z_2^{-\h}\\
&=
\sum_{n \in \Z} \langle u, a_2(h-n-2,\h-1)a_1(n)a_3 \rangle \left(\frac{z_1}{z_2}\right)^{-n-1}
\end{align*}
%Again by Lemma \ref{lem_deg_L}, we have
%\begin{align*}
%z_2^h\z_2^\h\sum_{n\in\Z} \langle u, Y(a_2,z_2,\z_2)a_1(n)a_3 \rangle &=
%z_2^h\z_2^\h\sum_{n\in\Z}
%\langle u, a_2(h-n-2,\h-1)a_1(n)a_3 \rangle z_1^{-n-1}z_2^{-h+n+1}\z_2^{-\h}\\
%&=\sum_{n\in\Z}
%\langle u, a_2(h-n-2,\h-1)a_1(n)a_3 \rangle
%\left(\frac{z_1}{z_2}\right)^{-n-1}.
%\end{align*}
By (C2), \(a(n)a_3=0\) for sufficiently large \(n\). Hence this series,
with \(z=\frac{z_1}{z_2}\), belongs to \(\C((z))\). Moreover, by assumption, after setting \(z_2=1\),
this sum is absolutely convergent in the region \(0<|z_1|<1\).
It follows that \(G(z)\) has at most a pole at \(z=0\).
Similarly, using \eqref{eq_pre_bor}, we see that \(G(z)\) has at most poles
also at \(z=1\) and \(z=\infty\). Therefore \(G(z)\) defines a meromorphic
function on \(\CP\) with possible poles at $\{0,1,\infty\}$.
Hence \(G(z) \in \C[z^\pm,(1-z)^\pm]\).
%ここで定義より和$\sum_{\Delta \geq 0}
%z_2^{h_1+h_2+h_3-h_0}\z_2^{\h_1+\h_2+\h_3-\h_0} \langle u, Y(a_2,z_2,\z_2)\pr_\Delta Y(a_1,z_1)a_3 \rangle$は絶対収束する。
%ここで Lemma \ref{lem_deg_L} からこの和は
%\begin{align*}
%\sum_{n \in \Z} \langle u, Y(a_2,z_2,\z_2)a_1(n)a_3 \rangle z_1^{-n-1}
%\end{align*}
%に等しく、再びLemma \ref{lem_deg_L}より
%\begin{align*}
%z_2^h\z_2^\h\sum_{n\in\Z} \langle u, Y(a_2,z_2,\z_2)a_1(n)a_3 \rangle &=
%z_2^h\z_2^\h\sum_{n\in\Z}
%\langle u, a_2(h-n-2,\h-1)a_1(n)a_3 \rangle z_1^{-n-1}z_2^{-h+n+1}\z_2^{-\h}\\
%&=\sum_{n\in\Z}
%\langle u, a_2(h-n-2,\h-1)a_1(n)a_3 \rangle
%\left(\frac{z_1}{z_2}\right)^{-n-1}.
%\end{align*}
%ここで (C2) より、$a(n)a_3=0$が$n$が十分に大きいときに成り立つため、これは下に有界な Lorentz series の空間$\C((z))$に$z=\frac{z_1}{z_2}$として含まれている。さらに仮定からこの和は、$z_2=1$とおいたとき、$0<|z_1|<1$において絶対収束している。
%よって、$G(z)$は$z=0$で高々極を持つことが分かる。
%同様にして、\eqref{eq_pre_bor}から$G(z)$は$z=1,\infty$でも高々極を持つことが分かるため、$\CP$上の有理形関数を定める。よって$G(z) \in \C[z^\pm,(1-z)^\pm]$が分かる。
\end{proof}

The following corollary follows from
the Cauchy integral formula and the above proposition (see \cite[Lemma 3.11]{M1}):
% the residue theorem, a standard method in the theory of vertex algebras .
\begin{cor}
Let $a,b \in H$ and assume that $a$ is a chiral vector. Then, for any $n \in \Z$
\begin{align*}
[a(n),Y(b,z,\z)] &= \sum_{k \geq 0}\binom{n}{k} Y(a(k)b,z,\z)z^{n-k}\\
Y(a(n)b,z,\z) &= \sum_{k \geq 0}(-1)^k\binom{n}{k}a(n-k)Y(b,z,\z)z^{k} -  \sum_{k \geq 0}\binom{n}{k}(-1)^{n+k}
Y(b,z,\z) a(k) z^{n-k}
\end{align*}
\end{cor}

\begin{prop}\cite[Proposition 3.12]{M1}
\label{prop_vertex_algebra}
The spaces \(\ker \Ld(-1)\) and \(\ker L(-1)\) carry vertex algebra
structures, and \(H\) is a module over these vertex algebras.
\end{prop}

\section{Construction of examples}
In this section, we construct a family of examples of $d$-CC systems.
The definitions of $d$-CC systems and conformal $d$-vertex algebras also make sense over $\R$, which we take as the base field in this section.
The underlying vector space of a $d$-CC system is a compact $\sod$-module.
In Section \ref{sec_Verma}, using parabolic induction, we study the highest and lowest weight $\sod$-modules $V(\al)$ and $V(\al)^\da$, depending on a
parameter $\al\in\R$, together with their simple quotients
$L(\al)$ and $L(\al)^\da$, and their realizations in spaces of polynomials.
In Section \ref{sec_construction}, using the canonical pairing between
$L(\al)$ and $L(\al)^\da$, we introduce an infinite-dimensional Heisenberg
Lie algebra which may be regarded as a higher-dimensional analogue of the
affine Heisenberg Lie algebra. Using this Lie algebra together with the
polynomial realizations above, we construct a scalar field and study its
conformal covariance. In Section \ref{sec_construction3}, we use these
constructions to obtain $d$-CC systems over $\R$.

In Appendix~\ref{sec_singular}, we derive an explicit determinant formula for the invariant bilinear form on these parabolic Verma modules and use it to establish the irreducibility.
%The irreducibility of the parabolic Verma modules $V(\al)$ and
%$V(\al)^\da$, as well as the explicit form of their singular vectors, is
%studied in Appendix \ref{sec_singular}.
By this formula, the irreducible highest and lowest weight modules $L(\al)$ and $L(\al)^\da$ admit positive-definite invariant inner products precisely for
$\al\geq \frac{d-2}{2}$ or $\al=0$,
and, for $d>2$, the distinguished value
$\al=\frac{d-2}{2}$
corresponds to the conformal field theory of the massless free scalar field.
In Appendix \ref{sec_massless}, we show that the Hilbert-space completion of $L\left(\frac{d-2}{2}\right)^\da$
carries a projective unitary representation of $\SO_e(d,2)$
whose $(\so(d,2)_\C,K)$-module agrees with the
$\sod\cong\so(d,2)_\C$-module structure on
$L\left(\frac{d-2}{2}\right)^\da$.
Finally, in Appendix \ref{sec_polynomial}, we derive an explicit formula
expressing the invariant inner product on $L\left(\frac{d-2}{2}\right)^\da$ in terms of the $L^2$-inner product on the standard
sphere $S^{d-1}$.

%この章では、$d$-CC system の例の族を構成する。
%$d$-CC system の underlying vector space は compact $\sod$-module であり、section \ref{sec_Verma}では、parabolic induction を用いて、定義される $\sod$ の parameter $\al \in\R$を持った highest および lowest weight modules $V(\al),V(\al)^\da$、およびそれらの単純商$L(\al), L(\al)^\da$
%の多項式への実現を調べる。
%Section \ref{sec_construction}では、$L(\al)$と$L(\al)^\da$の canonical pairing を用いて、affine Heisenberg Lie algebra の高次元への拡張となる 無限次元の Heisenberg Lie algebra を導入する。またこのリー代数と多項式への実現を用いて、scalar 場を構成し、その共形共変性を調べる。Section \ref{sec_construction2}では、これらを用いて$d$-CC system を構成する。
%
%parabolic Verma modules $V(\al),V(\al)^\da$の既約性や singular vector の具体的な形などは Appendix \ref{sec_singular} において調べられる。
%既約な highest and lowest weight modules $L(\al),L(\al)^\da$ は特に$\al \geq \frac{d-2}{2}$のときに限り、正定値な不変内積を持ち、
%$\al = \frac{d-2}{2}$のパラメータは$d>2$において、massless free scalar 場のなす共形場理論に対応する。
%Appendix \ref{sec_massless} では$L(\frac{d-2}{2})$のヒルベルト空間としての完備化が$\SO_e(d,2)$の射影的ユニタリ表現の構造を持ち、その $(\so(d,2)_\C,K)$加群が$L(\frac{d-2}{2})$の$\sod \cong \so(d,2)_\C$加群の構造と一致することが示される。
%また$\al=\frac{d-2}{2}$の場合の内積を標準球面$S^{d-1}$上の$L^2$内積を用いて計算する公式は Appendix \ref{sec_polynomial} で示す。
 
\subsection{Harmonic polynomials and representation}\label{sec_Verma}
In this section, we introduce and study a family of representations $V(\al),L(\al),V(\al)^\da,L(\al)^\da$ of
$\mathfrak{so}(d+1,1)$.

For any $\al \in \R$,
let $d_\al:\mathfrak{so}(d+1,1) \rightarrow \R[x_1,\partial_1,\dots,
x_d,\partial_d]$ be a linear map defined by
\begin{align*}
d_\al(D)&= -E_x - \al \\
d_\al(K_\mu) &= ||x||^2\pa_\mu -2 x_\mu (E_x +\al)\\
d_\al(P_\mu) &= -\pa_\mu \\
d_\al(J_{\mu,\nu}) &= x_\mu \pa_\nu-x_\nu \pa_\mu
\end{align*}
for $\mu,\nu \in \{1,2,\dots,d \}$.
It is easy to show that the linear map $d_\al:\mathfrak{so}(d+1,1) \rightarrow \R[x_1,\partial_1,\dots,x_d,\partial_d]$ gives a Lie algebra homomorphism.
Thus, $\R((x_1,\dots,x_d,|x|^\R))$ is a $\sod$-module via $d_\al$.
Since for any $\be \in \R$
\begin{align*}
d_\al(K_\mu) ||x||^{2\be} =  \left(||x||^2\pa_\mu -2 x_\mu (E_x +\al) \right)||x||^{2\be}= 2x_\mu (\be -2\be-\al) ||x||^{2\be},
\end{align*}
we have:
\begin{lem}\label{lem_highest}
$1 \in \R[x^1,\dots,x^d]$
satisfies
\begin{align*}
d_\al(J_{\mu,\nu})1&= 0 \\
d_\al(P_\mu)1 &= 0 \\
d_\al(D) 1 &= - \al 1
\end{align*}
and
$\frac{1}{||x||^{2\al}} \in \R[x_1,\dots,x_d,|x|^\R]$
satisfies
\begin{align*}
d_\al(J_{\mu,\nu})||x||^{-2\al}&= 0 \\
d_\al(K_\mu)||x||^{-2\al}&= 0 \\
d_\al(D)||x||^{-2\al} &= \al ||x||^{-2\al}.
\end{align*}
\end{lem}
By Lemma \ref{lem_highest}, $1$ and $||x||^{-2\al}$ are highest and lowest weight vectors of $\sod$, respectively.
Let $\sod^+$ (resp. $\sod^-$) be a subalgebra of $\sod$ spanned by
$\{P_\mu, D, J_{\mu,\nu}\}_{\mu,\nu \in \{1,\dots,d\}}$ (resp. $\{K_\mu, D, J_{\mu,\nu}\}_{\mu,\nu \in \{1,\dots,d\}}$).
%(resp. $\{P_\mu, D, J_{\mu,\nu}\}_{\mu,\nu \in \{1,\dots,d\}}$).
Let us consider the one-dimensional representation $\R v_\al$ (resp. $v_\al^\da$) of $\sod^+$ (resp. $\sod^-$) defined by
\begin{align}
\begin{split}
P_\mu v_\al &=0,\\
D v_\al &= -\al v_\al,\\
J_{\mu,\nu} v_\al &= 0
\end{split}
\label{eq_univ_verma}
\end{align}
and
\begin{align}
\begin{split}
K_\mu v_\al^\da &=0,\\
D v_\al^\da &= \al v_\al^\da,\\
J_{\mu,\nu} v_\al^\da &= 0
\end{split}
\label{eq_dual_Verma}
\end{align}
for $\mu,\nu\in \{1,\dots,d\}$.
Let $V(\al) = \mathrm{Ind}^{\sod} \R v_\al$ (resp.  $V(\al)^\da = \mathrm{Ind}^{\sod} \R v_\al^\da$)
 be the induced representation of $\sod$ from $\R v_\al$ (resp. $\R v_\al^\da$), which is a parabolic Verma module.

Let $\theta:\mathfrak{so}(d+1,1) \rightarrow \mathfrak{so}(d+1,1)$
be an involution of $\mathfrak{so}(d+1,1)$ defined by
\begin{align}
\begin{split}
\theta(P_\mu)&= -K_\mu \\
\theta(K_\mu)&= -P_\mu \\
\theta(D)&=-D \\
\theta(J_{\mu,\nu})&=J_{\mu,\nu},
\end{split}
\label{def_theta_add}
\end{align}
for $\mu,\nu \in \{1,2,\dots,d\}$.
It is easy to confirm that the linear map $\theta$ is an automorphism of the Lie algebra
$\mathfrak{so}(d+1,1)$.
For any $\sod$-module $M$,
we can define a new action by 
$A\cdot_\theta m = \theta(A)\cdot m$
for $A \in \sod$ and $m\in M$.
%$h \circ \theta: \sod \rightarrow \End \;M$. 
We denote this module by $\theta^* M$. Then, it is clear that $\theta^* V(\al) \cong V(\al)^\da$ as $\sod$-modules. Denote this isomorphism by
\begin{align}
\theta_\al:\theta^* V(\al) \cong V(\al)^\da,\quad v_\al \mapsto v_\al^\da \label{eq_theta_al}
\end{align}

%Denote $v_\al \in V(\al)^\dagger$ by $v_\al^\dagger$. Then, it satisfies
%\begin{align}
%\begin{split}
%P_\mu v_\al^\dagger &=0\\
%D v_\al^\dagger &= -\al v_\al^\dagger\\
%J_{\mu,\nu} v_\al^\dagger &= 0,
%\end{split}\label{eq_dual_Verma}
%\end{align}
%which characterize the $\sod$-module $V(\al)^\dagger$ by the universal property as an induced module.

For $n \in \Z_{\geq 0}$, set
\begin{align*}
V(\al)_n^\da = \{v \in V(\al)^\da \mid D(v) = (\al+n)v \},\\
V(\al)_n = \{v \in V(\al) \mid D(v) = (-\al-n)v \},
\end{align*}
which are spanned by
\begin{align}
\begin{split}
\{P_{\mu_1}P_{\mu_2}\cdots P_{\mu_n}v_\al^\da\}\quad \text{ and }\quad
\{K_{\mu_1}K_{\mu_2}\cdots K_{\mu_n}v_\al\}
\end{split}
\label{eq_basis_Verma}
\end{align}
with $1 \leq \mu_1 \leq \mu_2 \leq \dots \mu_n \leq d$, respectively.
Then, $V(\al) = \bigoplus_{n\geq 0}V(\al)_n$. Set
\begin{align*}
V(\al)^\vee= \bigoplus_{n\geq 0}V(\al)_n^*,
\end{align*}
the restricted dual of $V(\al)$, which is a $\sod$-module by $(A.f)(\bullet)=-f(A\bullet)$ for $f(\bullet) \in V(\al)^\vee$  and $A \in \sod$. Let $u_\al \in V(\al)_0^*$ satisfy $u_\al(v_\al)=1$.
Since $u_\al$ satisfies \eqref{eq_dual_Verma}, there is a unique $\sod$-module homomorphism
\begin{align*}
f_\al: V(\al)^\da \rightarrow V(\al)^\vee
\end{align*}
such that $f_\al(v_\al^\da) = u_\al$.
Define a bilinear form
\begin{align*}
(-,-):V(\al)^\da \otimes V(\al)^\da \rightarrow \R
\end{align*}
by
\begin{align*}
V(\al)^\da \otimes \theta^* V(\al)^\da \overset{f_\al \otimes \theta_\al}{\rightarrow} V(\al)^\vee \otimes V(\al) \overset{\langle-,-\rangle}{\rightarrow}\R,
\end{align*}
where the last map is the canonical pairing.
Then, we have:
\begin{prop}\label{prop_bilinear_Verma}
The bilinear form $(-,-):V(\al)^\da \otimes V(\al)^\da \rightarrow \R$ satisfies
\begin{enumerate}
\item
$(v_\al^\da,v_\al^\da)=1$;
\item
$(Au,v) = -(u,\theta(A)v)$ for any $u,v\in V(\al)^\da$ and $A\in \sod$;
\item
$(u,v)=(v,u)$  for any $u,v\in V(\al)^\da$.
\end{enumerate}
Moreover, $(-,-)$ is the unique bilinear form which satisfies (1) and (2).
\end{prop}
\begin{proof}
(1) and (2) follow from the definition. By \eqref{eq_basis_Verma}, (1) and (2) uniquely characterize the bilinear form.
Let $(-,-)'$ be the bilinear form defined by $(u,v)' = (v,u)$ for $u,v \in V(\al)^\da$. Since $\theta^2=\id$, $(-,-)'$ also satisfies (1) and (2). Hence, the bilinear form is symmetric.
\end{proof}
The following lemma is standard \cite{Hum}:
\begin{lem}\label{lem_irreducible_non_deg}
The bilinear form $(-,-)$ on $V(\al)^\da$ is non-degenerate if and only if $V(\al)^\da$ is an irreducible $\sod$-module.
\end{lem}

Denote by $N(\al)^\da$ the (unique) maximal submodule of $V(\al)^\da$ which does not contain $v_\al^\da$. Then, 
\begin{align*}
L(\al)^\da =V(\al)^\da / N(\al)^\da
\end{align*}
is an irreducible $\sod$-module. Denote also by $L(\al)$ the irreducible quotient of $V(\al)$. Then, we have:
%Denote by $L(\al)^\da$ the proper maximal quotient of $V(\al)^\da$, which is irreducible.
%Verma module についての一般的性質から次の命題が従う:
\begin{prop}\label{prop_bilinear_L}
The $\sod$-module homomorphism $f_\al:V(\al)^\dagger\rightarrow V(\al)^\vee$ induces an isomorphism of $\sod$-modules:
\begin{align}
f_\al:L(\al)^\dagger \rightarrow L(\al)^\vee,
\label{eq_Lf_al}
\end{align}
which defines a non-degenerate bilinear form $(-,-):L(\al)^\da\otimes L(\al)^\da \rightarrow \R$ 
satisfying (1), (2) and (3) in Proposition \ref{prop_bilinear_Verma}.
Moreover, \eqref{eq_Lf_al} gives a non-degenerate pairing $L(\al)^\da \otimes L(\al) \overset{f_\al\otimes \id}{\rightarrow} L(\al)^\vee \otimes L(\al) \overset{\langle-,-\rangle}{\rightarrow} \R$, denoted by $\langle -,-\rangle$ which satisfies: $\langle v_\al^\da,v_\al \rangle =1$ and
\begin{align*}
\langle A u^\da,v \rangle = -\langle u^\da, A v \rangle
\end{align*}
for any $A\in\sod$, $u^\da\in L(\al)^\da$ and $v \in L(\al)$.
%,
%where $L(\al)^\dagger$ is the maximal quotient of the Verma module $V(\al)^\dagger$.
\end{prop}

It is important to determine $N(\al)^\da$, which is given in Appendix A.
Here we briefly examine the structure of $N(\al)_2^\da =N(\al)^\da\cap V(\al)_2^\da$.
% ここでは手短に$N(\al)_2^\da =N(\al)^\da\cap V(\al)_2^\da$の構造を調べる。
%well-known that any submodule of $V(\al)$ is generated by the \textbf{singular vectors} (of weight $k \in \Z_{\geq 0}$) which is a vector $v\in V(\al)_k$ satisfying $K_\mu v= J_{\mu,\nu}v =0$.
%$V(\al)$は既約か？またそうでない場合は singular vector はどんなものがあるか？を知ることは本論文において重要である。
By an easy computation, we have
\begin{align}
(P_\mu v_\al^\da, P_\nu v_\al^\da) = 2 \al \delta_{\mu,\nu} \label{eq_form_deg1}
\end{align}
and, for $\mu_1 \leq \mu_2$ and $\nu_1 \leq \nu_2$,
\begin{align}
\begin{split}
(P_{\mu_1}P_{\mu_2} v_\al^\da, P_{\nu_1}P_{\nu_2} v_\al^\da) =
\begin{cases}
4\al (\al+1) \delta_{\mu_1,\nu_1}\delta_{\mu_2,\nu_2} & \text{ if } \mu_1 < \mu_2,\\
\delta_{\nu_1,\nu_2}\left(8\al (\al+1) \delta_{\mu_1,\nu_1}-4\al\right) & \text{ if } \mu_1 = \mu_2.
\end{cases}
\end{split}\label{eq_form_deg2}
\end{align}
By $\dim V(\al)_2^\da = \binom{d+1}{2}$, the eigenvalues of the matrix representation of $(-,-)$ on $V(\al)_2^\da$ for the basis \eqref{eq_form_deg2} consist of $ 4\al (\al+1)$ and $8\al(\al+1)$ (multiplicity $\binom{d}{2}$ and $d-1$) and $4\al (2\al - d+2)$ (multiplicity one).
%$\dim V(\al)_2^\da = \binom{d+1}{2}$に注意すると、$(-,-)$の$V(\al)_2^\da$への制限を\eqref{eq_form_deg2}の基底について行列表示したとき、その固有値は$\binom{d+1}{2}-1$個の$4\al (\al+1)$と1個の $4\al (2\al - d+2)$からなる。
Hence, $V(\al)^\da$ is reducible if
\begin{align*}
\al = 0,-1, \frac{d-2}{2}.
\end{align*}
This can be seen directly through the representation to polynomial rings as follows:
By Lemma \ref{lem_highest}, the universality \eqref{eq_dual_Verma} and \eqref{eq_univ_verma}, there exists a unique $\sod$-module homomorphism 
\begin{align}
\Psi_\al^\da: V(\al)^\da \rightarrow \R[x^1,\dots,x^d,|x|^\R] \label{eq_psi_dagger}
\end{align}
and
\begin{align}
\Psi_\al: V(\al) \rightarrow \R[x^1,\dots,x^d]\label{eq_psi_al}
\end{align}
such that $\Psi_\al^\da(v_\al^\da) = ||x||^{-2\al}$ and $\Psi_\al(v_\al) = 1$.

It is important to note that for any $\be \in \R$,
\begin{align}
\left(\sum_{\rho=1}^d \pa_\rho^2\right) ||x||^{-2\be} = 4\be \left(\be - \frac{d-2}{2} \right)||x||^{-2\be-2}
\label{eq_laplace_beta}
\end{align}
and 
\begin{align}
\begin{split}
\left(\sum_{\rho=1}^d \left(||x||^2\pa_\rho -2x_\rho(E_x+\al)\right)^2\right) 1 &=
-2\al \sum_{\rho=1}^d \left(||x||^2\pa_\rho -2x_\rho(E_x+\al)  \right)x_\rho\\
&=-2\al(d-2(\al+1))||x||^2.
\label{eq_laplace_beta2}
\end{split}
\end{align}
Thus, if $\al = \frac{d-2}{2}$, then $\Psi_{\frac{d-2}{2}}^\da:V\left(\frac{d-2}{2}\right)^\da \rightarrow \R[x^1,\dots,x^d,|x|^\R]$ has a non-zero kernel by $\Psi_{\frac{d-2}{2}}\left( \left(\sum_{\rho=1}^d P_\rho^2\right)v_{\frac{d-2}{2}}^\da \right) =
\left(\sum_{\rho=1}^d \pa_\rho^2\right) ||x||^{-d+2} =0$.
Hence, $\left(\sum_{\rho=1}^d P_\rho^2\right)v_{\frac{d-2}{2}}^\da$ is a non-zero vector in $N(\dn)^\da$, which corresponds to the eigenvalue in \eqref{eq_form_deg2} with multiplicity one.
Note that from $V(\al)^\da$ we also know $V(\al)$ by \eqref{eq_theta_al}.
%同様に$(\sum_\rho K_\rho^2) v_\dn$は$$
% singular vector in $V(\dn)^\dagger$である。
Appendix A gives the following theorem:
\begin{thm}\label{prop_singular_vector}
Let $d \geq 2$.
The following properties hold for $V(\al)^\da$:
\begin{enumerate} 
\item 
$V(\al)^\da$ is irreducible if and only if
\begin{align*} 
\al \in \R \setminus\left\{ \{0,-1,-2 ,\dots \} \cup \{\frac{d-2}{2}, \frac{d-2}{2}-1,\frac{d-2}{2}-2,\dots \}\right\}.
\end{align*}
In this case $\Psi_\al:V(\al) \rightarrow \R[x^1,\dots,x^d]$ is an isomorphism of $\sod$-modules.
Also if $\al > \frac{d-2}{2}$, then the bilinear form $(-,-)$ on $V(\al)^\da$ is positive-definite.
\item
%$d \geq 3$の奇数ならば任意の$N=0,1,2,\dots$に対して$V\left(\frac{d-2}{2}-N\right)^\da$ を $\left(\sum_\rho P_\rho^2\right)^{N+1}v_{\frac{d-2}{2}-N}^\da$で生成される部分加群で割った加群は既約である。さらに既約加群
%\begin{align*}
%L\left(\dn\right)^\da = V\left(\frac{d-2}{2}\right)^\da / \left(\sum_\rho P_\rho^2\right) v_{\frac{d-2}{2}}^\da
%\end{align*}
%上に誘導される双線形形式$(-,-)$は正定値である。
If $d \geq 3$ is an odd integer, for any $N=0,1,2,\dots$, the quotient module of $V\left(\frac{d-2}{2}-N\right)^\da$ 
by the submodule generated by $\left(\sum_\rho P_\rho^2\right)^{N+1}v_{\frac{d-2}{2}-N}^\da$
is irreducible. Moreover, the induced bilinear form $(-,-)$ on the irreducible module
\begin{align*}
L\left(\dn\right)^\da = V\left(\frac{d-2}{2}\right)^\da / \left(\sum_\rho P_\rho^2\right) v_{\frac{d-2}{2}}^\da
\end{align*}
is positive-definite.
\item 
If $d \geq 4$ is an even integer, for any 
$K=0,1,2,\dots, \frac{d-2}{2}-1$, 
the quotient module of $V\left(\dn - K \right)^\da$ by the submodule generated by 
$\left(\sum_\rho P_\rho^2\right)^{K+1}v_{\frac{d-2}{2}-K}^\da$ is irreducible. Moreover, the induced bilinear form $(-,-)$ on the irreducible module
\begin{align*}
L\left(\dn\right)^\da =V\left(\frac{d-2}{2}\right)^\da / \left(\sum_\rho P_\rho^2\right) v_{\frac{d-2}{2}}^\da
\end{align*}
is positive-definite.
%
%
%
%
%$d \geq 4$ for any $K=0,1,2,\dots, \frac{d-2}{2}-1$ if even, $V\left(\dn - K \right)^\da$ for $\left(\sum_\rho P_\rho^2\right)^{K+1}v_{\frac{ d-2}{2}{2}-K}^\da$ is irreducible. Furthermore, the irreducible additive group 
%\begin{align*} 
%L\left(\dn\right)^\da =V\left(\frac{d-2}{2}\right)^\da / \left(\sum_\rho P_\rho^2\right) v_{\frac{d-2}{2}}^\da 
%\end {align*} 
%The bilinear form $(-,-)$ induced above is positive definite.
%\item
%$V(\al)^\da$が既約であることと$\al$が次の値を取ることは同値である:
%\begin{align*}
%\al \in \R \setminus \{0,-1,-2,\dots \} \cup \{\frac{d-2}{2}, \frac{d-2}{2}-1,\frac{d-2}{2}-2,\dots \}.
%\end{align*}
%このとき、$\Psi_\al:V(\al) \rightarrow \R[x^1,\dots,x^d]$は$\sod$-module の同型である。
%また $\al > \frac{d-2}{2}$ならば $V(\al)^\da$上の bilinear form $(-,-)$は positive-definite である。
%\item
\end{enumerate}
Furthermore, the non-degenerate bilinear form on $L(\al)^\da$ is positive-definite if and only if $\al \geq \dn$ or $\al=0$.
\end{thm}
In this paper, we call $\al \in \R$ a generic case when $V(\al)^\da$ is irreducible (see Theorem\ref{prop_singular_vector}).
%The following proposition is also shown in Appendix:
% follows from the argument below \eqref{eq_form_deg2} (or Theorem \ref{thm_determinant} in Appendix):
%\begin{prop}\label{prop_indefinite}
%Let $d \geq 3$. Then, 
%%, then the non-degenerate bilinear form $(-,-)$ on $V(\al)$ is indefinite. 
%\end{prop}
%本論文では$\al \in \R$が 命題\ref{prop_singular_vector} の (1) のときを generic case と呼ぶ。
%Section \ref{}では generic な $\al$に対して、$d$-vertex algebra and conformally flat $d$-algebra 
%$(H_{d,\al},\m,\va)$を構成する。
%また$\al= \dn$ (with $d \geq 3$)の場合は Section \ref{}で扱う。
Hereafter, we will study $\sod$-module homomorphisms $\Psi_\al^\da:V(\al)^\da \rightarrow \R[x_1,\dots,x_d,|x|^\R]$ and $\Psi_\al:V(\al) \rightarrow \R[x_1,\dots,x_d]$.
Let $n\geq 0$ and $\R[x_1,\dots,x_d]_n$ be the subspace of $\R[x_1,\dots,x_d]_n$ consisting of polynomials of degree $n$.
The following lemma is useful:
\begin{lem}\label{transform_explicit}
For $\al \in \R$, define a linear map $T_\al: \R[x_1,\dots,x_d] \rightarrow \theta^* \R[x_1,\dots,x_d,|x|^\R]$ by
\begin{align}
T_\al(P(x)) =P(x)||x||^{-2\al-2n}= P\left(\frac{x}{||x||^2}\right)||x||^{-2\al}\label{eq_Tal_def}
\end{align}
for any $P(x) \in \R[x_1,\dots,x_d]_n$. Then, $T_\al$ is a $\sod$-module homomorphism,
where we regard $\R[x_1,\dots,x_d]$ and $\R[x_1,\dots,x_d,|x|^\R]$ as $\sod$-modules by $d_\al$.
\end{lem}
\begin{proof}
Let $\mu,\nu \in \{1,\dots,d \}$ and $P(x) \in \R[x_1,\dots,x_d]_n$.
Since $d_\al(J_{\mu,\nu})||x||^{-2\al-2n}=0$,
we have $d_\al(J_{\mu,\nu}) T_\al(P(x))=T_\al(d_\al(J_{\mu,\nu})P(x))$.
Similarly, by
\begin{align*}
d_\al(P_\mu)T_\al(P(x))&= - \pa_\mu \Bigl(P(x)||x||^{-2\al-2n} \Bigr)\\
&= ||x||^{-2\al-2n-2}\Bigl(-||x||^2\pa_\mu+2 x_\mu(\al+n)  \Bigr)P(x)\\
&= ||x||^{-2\al-2n-2}\left(d_\al(-K_\mu)P(x)\right)\\
&= T_\al \left(d_\al(\theta(P_\mu))P(x)\right)
\end{align*}
and
\begin{align*}
d_\al(K_\mu)T_\al(P(x))&=\Bigl(||x||^2\pa_\mu- 2x_\mu(E_d+\al)\Bigr)\Bigl(||x||^{-2\al-2n}P(x) \Bigr)\\
&= ||x||^{-2\al-2n} (-2x_\mu (\al+n) - 2x_\mu(-2\al-n+\al)+||x||^2\pa_\mu )P(x)\\
&=||x||^{-2\al-2(n-1)} \left(\pa_\mu P(x)\right)\\
&= T_\al\left(d_\al(\theta(K_\mu))P(x) \right)
\end{align*}
and 
\begin{align*}
d_\al(D)T_\al(P(x))&=-( E_d+\al)\Bigl(||x||^{-2\al-2n}P(x) \Bigr)\\
&=(\al+n)\Bigl(||x||^{-2\al-2n}P(x) \Bigr)\\
&=||x||^{-2\al-2n}\Bigl(d_\al(-D) P(x) \Bigr)\\
&=T_\al\Bigl(d_\al(\theta(D)) P(x) \Bigr).
\end{align*}
Hence, $T_\al$ is a $\sod$-module homomorphism.
%from $H_d$ to $\theta^* \R[x_1,\dots,x_d,|x|^\R]$
%and $T_H(1)= |x|^{-(d-2)}$.
%Thus, the image of $T_H$ is equal to $H_d^\vee$.
%Since $H_d$ and $H_d^\vee$ is irreducible,
%$T_H$ must coincide with $P^\vee \circ T_S \circ P^{-1}$.
\end{proof}

First, consider the case when $\al$ is generic.
Since $V(\al)$ is irreducible in this case, from the comparison of dimension, $\Psi_\al:V(\al)\rightarrow \R[x_1,\dots,x_d]$ is an isomorphism of $\sod$-modules.
Also, $\Psi_\al^\da:V(\al)^\da \rightarrow \R[x_1,\dots,x_d,|x|^\R]$ is an isomorphism onto the image, which is from Lemma \ref{transform_explicit} given by 
\begin{align} 
\bigoplus_{n\geq 0} ||x||^{-2\al-2n}\R[x_1,\dots,x_d]_n.
\end{align}
%On $\R[x^1,\dots,x^d]$, we have a non-degenerate bilinear form induced from the invariant bilinear form on $V(\al)^\dagger$.
%まず$\al$が generic な場合を考える。
%このとき$V(\al)$は既約であるから\eqref{eq_basis_Verma}と次元の比較から、$\Psi_\al:V(\al) \rightarrow \R[x^1,\dots,x^d]$は $\sod$-module の同型写像である。
%また$\Psi_\al^\da:V(\al)^\da \rightarrow \R[x^1,\dots,x^d,|x|^\R]$は像への同型写像であり、その像は、Lemma \ref{transform_explicit}より、
%\begin{align}
%\bigoplus_{n\geq 0} ||x||^{-\al-n}\R[x_1,\dots,x_d]_n
%\end{align}
%で与えられる。
%$\R[x^1,\dots,x^d]$には$V(\al)^\dagger$上の invariant bilinear form から誘導される、non-degenerate bilinear form が入る. 
%We denote it by $(-,-)_P$.
Hence, we have:
\begin{prop}\label{prop_psi_generic}
Assume that $\al$ is generic.
Then, $\Psi_\al:V(\al) \rightarrow \R[x_1,\dots,x_d]$ and $\Psi_\al^\da:V(\al)^\da \rightarrow \bigoplus_{n\geq 0} ||x||^{-2\al-2n}\R[x_1,\dots,x_d]_n$ are $\sod$-module isomorphisms.
Moreover, 
the dual basis of 
\begin{align*}
\{d_\al(P_{\mu_1}) \cdots d_\al(P_{\mu_n})v_\al^\da \}_{1\leq \mu_1 \leq \dots\leq \mu_n\leq d} \subset V(\al)_n^\da
\end{align*}
with respect to the non-degenerate invariant bilinear pairing $V(\al)^\da \otimes V(\al) \rightarrow \R$ is 
\begin{align}
\left\{\frac{1}{k_1! k_2!\cdots k_l!}x_{i_1}^{k_1}x_{i_2}^{k_2} \dots x_{i_l}^{k_l}\right\}
\label{eq_dual_basis}
\end{align}
where $n=\sum_{i=1}^l k_i$, $1 \leq l\leq d$, $1\leq i_1<i_2 < \dots <i_l \leq d$ and $k_i \geq 0$
and we identify $V(\al)$ as $\R[x^1,\dots,x^d]$.
\end{prop}
\begin{proof}
It suffices to show that \eqref{eq_dual_basis} gives the dual basis. Since
\begin{align*}
(d_\al(P_{\mu_1}) \cdots d_\al(P_{\mu_n})v_\al^\da, x_{\nu_1}\dots,x_{\nu_n}) &=
(v_\al^\da, d_\al(P_{\mu_n}) \cdots d_\al(P_{\mu_1})
 x_{\nu_1}\dots,x_{\nu_n})_P \\
 &=
(v_\al^\da, \pa_{\mu_1}\cdots \pa_{\mu_n} x_{\nu_1}\dots,x_{\nu_n})_P,
\end{align*}
the assertion holds.
\end{proof}

Hereafter in this section, we assume that $\al = \dn$.
Then, by Theorem \ref{prop_singular_vector}, \eqref{eq_laplace_beta} and \eqref{eq_laplace_beta2}, $\Psi_\al$ and $\Psi_\al^\dagger$ factor through
$\Psi_\al^\da: L(\al)^\da \rightarrow \R[x_1,\dots,x_d,|x|^\R]$ and $\Psi_\al: L(\al) \rightarrow \R[x_1,\dots,x_d]$.
Denote the images of $\Psi_\al^\da$ and $\Psi_\al$ by $H(\al)^\dagger$ and $H(\al)$.
By Proposition \ref{prop_bilinear_L}, we have the following $\sod$-module isomorphisms:
\begin{align}
\begin{split}
\Psi_\al^\dagger&:L(\al)^\dagger {\rightarrow} H(\al)^\dagger\\
\Psi_\al&:L(\al) \rightarrow H(\al)\\
f_\al&: L(\al)^\dagger \rightarrow L(\al)^\vee.
\end{split}\label{eq_three_isom}
\end{align}
%とくに多項式の空間$H(\al)$と$H(\al)^\dagger$には non-degenerate pairing が入る。
%$H(\al), H(\al)^\dagger$を明示的に与えて、この章を終える。
Set
\begin{align}
\mathrm{Harm}_{d,n} = \left\{P(x)\in \R[x_1,\dots,x_d]_n \mid \left(\sum_\rho \left(\frac{d}{dx_\rho}\right)^2\right) P(x) =0\right\},
\end{align}
which consists of harmonic polynomials of degree $n$.
Then, we have:
\begin{prop}\label{prop_psi_dn}
%$\al \in \R$が generic ならば、
%\begin{align}
%H(\al)^\dagger &= \R[x_1,\dots,x_d]\\
%H(\al) &= \bigoplus_{n \geq 0}||x||^{-\al-n} \R[x_1,\dots,x_d]_n.
%\end{align}
Let $d \geq 3$. Then, $H\left(\dn\right) \subset \R[x_1,\dots,x_d]$
coincides with the set of harmonic polynomials, that is,
\begin{align}
H\left(\dn\right) &= \bigoplus_{n \geq 0} \mathrm{Harm}_{d,n}\\
H\left(\dn\right)^\da &= \bigoplus_{n \geq 0} ||x||^{-d+2-2n} \mathrm{Harm}_{d,n}.\label{eq_Harm_dagger}
\end{align}
\end{prop}
\begin{proof}
%Assume that $\al$ is generic. Then, by Theorem \ref{thm_singular_vector},  $\Psi_\al:V(\al)^\dagger \rightarrow \R[x_1,\dots,x_d]$ is injective.
%Since the degree $n$ subspace of $V(\al)^\dagger$ has a basis $\{K_{i_1}\dots K_{i_d}\}_{1\leq i_1 \leq i_2\leq \dots\leq i_n\leq d}$, the dimension of $V(\al)_n^\dagger$ and $\R[x_1,\dots,x_d]_n$ coincides.
%Since the image of $\Psi_\al:V(\al) \rightarrow \R[x_1,\dots,x_d,|x|^\R]$ is spanned by
%\begin{align*}
%\frac{d}{dx_{i_1}}\dots\frac{d}{dx_{i_n}} ||x||^{-\al},
%\end{align*}
%$H(\al)$ is a subspace of $\bigoplus_{n \geq 0}||x||^{-\al-n} \R[x_1,\dots,x_d]_n$. Then, by the above argument, the assertion follows.
First, we observe that, for any $\mu \in \{1,\dots,d\}$ and $\al \in \R$,
\begin{align*}
\left[\sum_{\rho}d_\al(P_\rho)^2, d_\al(K_\mu)\right]
&=\left[\Delta_x, ({||x||^2}\pa_\mu-2x_\mu (E_d+\al))\right] \\
%&=\pa_\mu (2E_d+2\al) +2x_\mu \Delta - (2E_d +d)\pa_\mu \\
&=2(d-2-2\al)\pa_\mu -4x_\mu \Delta.
\end{align*}
Thus, the space of harmonic polynomials $\mathrm{Harm}_d = \bigoplus_{n\geq 0}\mathrm{Harm}_{d,n} \subset \R[x_1,\dots,x_d]$ is preserved by the action of $\sod$ if $\al=\frac{d-2}{2}$.
Since $1 \in \mathrm{Harm}_d$, the image of $\Psi_\al: L\left(\dn\right) \rightarrow \R[x_1,\dots,x_d]$ is in $\mathrm{Harm}_d$.

It is a well-known fact in the theory of spherical harmonics that $\mathrm{Harm}_{d,n}$ is an irreducible representation of $\mathrm{SO}(d)$.
Assume $P(x) \in \mathrm{Harm}_{d,n}$ satisfies
\begin{align*}
0= d_\dn(K_\mu)P(x) = \left(||x||^2\pa_\mu -2(n+\dn) x_\mu  \right)P(x)
\end{align*}
for all $\mu =1,\dots,d$. Then, $\sum_\mu x_\mu \pa_\mu P(x) = 2(n+\dn)P(x)$.
Since the degree of $P(x)$ is $n$, $(n+d-2)P(x) =0$. Hence, $P(x) =0$ if $d\geq 3$, and thus, 
$\Psi_\al: L\left(\dn\right) \rightarrow \mathrm{Harm}_d$ is an isomorphism for $d\geq 3$.
\eqref{eq_Harm_dagger} is a direct consequence of Lemma \ref{transform_explicit}.
\end{proof}

\begin{rem}
\label{rem_harmonic_dagger}
Since $||x||^{-d+2}$ is in the kernel of $\sum_\rho(\pa_\rho^2)$ by \eqref{eq_laplace_beta}
and $\sum_\rho(\pa_\rho^2)$ and $p_\mu$ commute,
the image of $\Psi_\al^\da$ consists of harmonic functions.
\end{rem}

\subsection{Construction of scalar fields}\label{sec_construction}

%\begin{itemize}
%\item
%\todo
%生成演算子として正しい$D$の次数を持っているのは、$v_\al$すなわち$L(\al)$側であり、これを$L(\al)^\da$と書くのが正しい!!
%\item
%なので $\Psi^\al$などの記号も逆にするべきか？
%\item
%以下、そのように修正されたと思ってかく。
%\end{itemize}
Let \(d\geq 2\). Throughout Sections \ref{sec_construction} and  \ref{sec_construction3}, we assume that \(\alpha>0\) and that either \(\alpha\) is generic or $\alpha=\frac{d-2}{2}$ with \(d>2\) (For the reason for imposing the condition \(\alpha>0\), see Remark~\ref{rem_upper}.)
%($\al >0$を課す理由は Remark \ref{rem_upper}を参照).
%Let $d \geq 2$ and $\al \in \R$ be generic or $\al=\dn$ in the case of $d >2$.
%Assume $$.
%Let $\mathfrak{H}$ be a finite dimensional vector space over $\R$ equipped with a symmetric non-degenerate bilinear form 
%$(-,-):\mathfrak{H}\times \mathfrak{H} \rightarrow \R$.
Set $$\hat{\mathfrak{A}}_{d,\al}= L(\al)^\dagger
\oplus L(\al) \oplus \R c.$$
Define a Lie bracket on $\hat{\mathfrak{A}}_{d,\al}$ by
\begin{align*}
[a,a'] &= 0 \\
[a,a^\da] &= -[a^\da,a]= \langle a^\da,a\rangle c \\
[a^\da,a'^\da] &= 0 \\
[a,c] &= 0 \\
[a^\da,c] &= 0,
\end{align*}
for $a,a' \in L(\al)$ and $a^\da,a'^\da \in L(\al)^\da$,
where $\langle-,-\rangle:L(\al)^\da \times L(\al) \rightarrow \R$ is the non-degenerate pairing, which satisfies
\begin{align}
\langle a^\da,Aa \rangle = - \langle A a^\da, a \rangle
\label{eq_pairing1}
\end{align}
for any $A \in \sod$ (see Proposition \ref{prop_bilinear_L}).
Then, $\hat{\mathfrak{A}}_{d,\al}$ is a Lie algebra.
Set
\begin{align*}
\Ad^+ = L(\al) \oplus \R c,
\end{align*}
which is a subalgebra of $\Ad$.
%which we call a $d$-dimensional affine Heisenberg algebra
%associated with $\mathfrak{H}$.
Define an action of $\sod$ on $\hat{\mathfrak{A}}_{d,\al}$ by
\begin{align*}
A\cdot a &= Aa \\
A \cdot a^\da &= A a^\da \\
A\cdot c&=0,
\end{align*}
for $A \in \sod$. 
Denote by $\mathrm{Der}(\hat{\mathfrak{A}}_{d,\al})$
the set of derivations of the Lie algebra $\hat{\mathfrak{A}}_{d,\al}$.
For $A\in \sod$, by \eqref{eq_pairing1},
\begin{align*}
[A a,a^\da]
+[a,A a^\da]
 &= (\langle A a^\da,a \rangle+ \langle a^\da,A a\rangle) c =0.
\end{align*}
Then, we have:
\begin{prop}\label{derivation_affine}
The above map $\sod \rightarrow \End\,\Ad$ gives a
Lie algebra homomorphism $\sod \rightarrow \mathrm{Der}\;\Ad$. The subalgebra $\Ad^+$ is preserved by this action.
\end{prop}
Let $\R \1$ be a one-dimensional  representation of
$\Ad^+$ defined by
\begin{align}
c. \1&=\1 \nonumber \\
a .\1&=0 \label{eq_verma}
\end{align}
for any $a \in L(\al)$.
We denote the induced module $\mathrm{Ind}^\Ad \R\va$
by $H_{d,\al}$
and the action of $\Ad$ on $\Hd$ by $\chi:\Ad \rightarrow \End \Hd$.
Define a $\sod$-module structure on $\R\va$ by
\begin{align}
X \va =0\quad\quad\text{ for any }X\in \sod. \label{eq_trivial_sod_vac}
\end{align}
Since the relation (\ref{eq_verma}) and \eqref{eq_trivial_sod_vac} are preserved by
the $\sod$-action, $\Hd$ is also a $\sod$-module.
More explicitly, for $a_1^\da,\dots,a_k^\da \in L(\al)^\da$ and $A\in \sod$, we have
$A (a_1^\da a_2^\da\dots a_k^\da \va)=(A a_1^\da)a_2^\da \dots a_k^\da \va+a_1^\da(Aa_2^\da)\dots a_k^\da\va+\dots+ a_1^\da a_2^\da\dots(Aa_k^\da)\va$.
In fact, by the PBW theorem, $\Hd$ is isomorphic to the symmetric algebra over $L(\al)^\da$ as $\sod$-modules.
Then, we have:
\begin{lem}\label{chi_covariant}
For $A\in \sod$ and $a \in \Ad$ and $v \in \Hd$,
$[A,\chi(a)]v=\chi(Aa)v$.
\end{lem}
\begin{rem}\label{rem_upper}
We note that \(H_{d,\alpha}\) is a compact \(\mathfrak{so}(d+1,1)\)-module if and only if \(\alpha>0\).
The lower boundedness of the \(D\)-spectrum and the finite-dimensionality of its eigenspaces are not strictly necessary for developing the theory. Just as the notion of a conformal vertex algebra provides a relaxation of that of a vertex operator algebra, one could appropriately weaken the compactness condition, in which case much of the discussion would continue to apply to generic parameters \(\alpha<0\). We nevertheless restrict ourselves to compact modules in the present paper, since working in this greater generality would make the reconstruction of vertex operators from correlation functions in Section~2 substantially more cumbersome.
%We note that $\Hd$ is a compact $\sod$-module if and only if $\al>0$.
%頂点作用素代数を緩めた概念として、共形頂点代数という概念があるように、理論を構築する上で$D$の固有値が下に有界であることや固有空間の有限性は必ずしも必須ではなく、適切に定義を緩めることで、$\al <0$の generic parameter についても多くの議論はそのまま適用できる。しかし、多くの定理がより煩雑になってしまうため、我々はより便利な compact 性を考える。
\end{rem}
Hereafter, we will define a conformally covariant vertex operator on $\Hd$.
For $n \in \Z_{\geq 0}$,
let $\{e_1^n,\dots,e_{l_n}^n \}$ be a basis of $L(\al)_n$
and $\{f_1^n,\dots,f_{l_n}^n\}$ be the dual basis of $L(\al)_n^\da$ with respect to the pairing \eqref{eq_pairing1}.
Set $I_n=\sum_{i=1}^{l_n} f_i^n \otimes e_i^n \in L(\al)^\da \otimes L(\al)$,
which is independent of a choice of basis.
Since $\Psi_\al (e_i^n) \in \R[x_1,\dots,x_d]_n$ and $\chi(f_i^n) \in \End \Hd$,
\begin{align*}
(\chi \otimes \Psi_\al) I_n \in \End \Hd[x_1,\dots,x_d]_n
\end{align*}
and similarly by Lemma \ref{transform_explicit}
\begin{align*}
(\Psi_\al^\da \otimes \chi) I_n \in ||x||^{-2\al-2n} \End H_{d,\al}[x_1,\dots,x_d]_n.
\end{align*}

Set
\begin{align*}
\phi^-(x)= \sum_{n=0}^\infty (\chi \otimes \Psi_\al) I_n,&\qquad
\phi^+(x)=\sum_{n=0}^\infty (\Psi_\al^\da \otimes \chi)I_n  \\
\phi(x)&= \phi^+(x)+\phi^-(x).
\end{align*}

%\begin{rem}
%In physics, $\frac{1}{2l+d-2}$ appears
%Because the map $(S_d)_l \rightarrow (H_d)_l$ is chosen in correct.
%Thus, issue are clarified in section \ref{sec_polynomial} (see also Remark \ref{rem_correct_inner}).
%\end{rem}

Since $\phi^+(x)$ consists of annihilation operators,
for any $v \in \Hd$, $\phi^+(x)v \in \frac{1}{||x||^{2\al}}\Hd[x_1,\dots,x_d,||x||^{-1}]$, a finite sum.
For example,
$\phi^+(x) \phi=\frac{1}{||x||^{2\al}}1$ and 
$\phi^+(x) P_\mu \phi= 2\al \frac{x_\mu}{||x||^{2\al+2}}1$.
Thus, $\phi(x)v \in \Hd((x_1,\dots,x_d,|x|^\R))$ for any $v \in \Hd$.
\begin{prop}\label{phi_covariance}
The vertex operator $\phi(x)$ satisfies 
$[A,\phi(x)] = -d_\al(A) \phi(x)$ for any $A\in\sod$, that is, for any $\mu,\nu \in \{1,2,\dots,d\}$,
\begin{align*}
[P_\mu, \phi(x)] &= \frac{d}{dx_\mu}\phi(x),\\
[J_{\mu,\nu}, \phi(x)]&=(x_\nu\pa_\mu-x_\mu\pa_\nu)\phi(x),\\
[D, \phi(x)] &= (E_x +\al)\phi(x),\\
%\Delta \phi(x) &= 0.
[K_\mu,\phi(x)]&=
(-||x||^2\frac{d}{dx^\mu}+2 x_\mu (E_x+\al))\phi(x),
\end{align*}
and $\phi^+(x)v \in \Hd[x_1,\dots,x_d,|x|^\R]$ for any $v \in \Hd$.
Moreover, if $\al=\dn$ with $d>2$, then
\begin{align}
\sum_\rho \left(\frac{d}{dx_\rho}\right)^2 \phi(x) =0.\label{eq_laplace_phi}
\end{align}
\end{prop}
In order to prove the proposition, we need the following lemma:
\begin{lem}\label{lem_In}
For $n \in \Z_{\geq 0}$ and $\mu,\nu \in \{1,\dots,d \}$,
\begin{align*}
(J_{\mu,\nu} \otimes 1+1\otimes J_{\mu,\nu}) I_n &=  0\\
(D\otimes 1 + 1\otimes D) I_n&=0 \\
(P_\mu \otimes 1) I_n +(1\otimes P_\mu) I_{n+1} &=0 \\
(K_\mu \otimes 1) I_{n+1} +(1\otimes K_\mu) I_{n}&=0.
\end{align*}
\end{lem}
\begin{proof}
For $n\in \Z_{\geq 0}$,
let $\{e_1^l,\dots,e_{l_n}^n \}$ be a basis of $L(\al)_n$
and $\{f_1^l,\dots,f_{l_n}^n\}$ be the dual basis of $L(\al)_n^\da$.
Let $\mu,\nu \in \{1,\dots,d \}$.
For any $a \in L(\al)_n$, by \eqref{eq_pairing1}, we have:
\begin{align*}
\sum_{i=1}^{l_n} \langle J_{\mu,\nu}f_i^n, a\rangle e_i^n
=\sum_{i=1}^{l_n} -\langle f_i^n, J_{\mu,\nu} a\rangle e_i^n =-J_{\mu,\nu}a = - \sum_{i=1}^{l_n} \langle f_i^n, a\rangle J_{\mu,\nu}e_i^n,
\end{align*}
which implies $(J_{\mu,\nu}\otimes 1 +1\otimes J_{\mu,\nu})I_n=0$
and similarly $(D\otimes 1 +1\otimes D)I_n=0$.
For any $a \in L(\al)_{n+1}$,
\begin{align*}
\sum_{i=1}^{l_n} \langle P_{\mu}f_i^n, a\rangle 
e_i^n
=\sum_{i=1}^{l_n} - \langle f_i^n, P_\mu a\rangle 
e_i^n =-P_\mu a = - \sum_{i=1}^{l_{n+1}} \langle f_i^{n+1},a\rangle P_\mu e_i^{n+1},
\end{align*}
which implies that
$(P_\mu \otimes 1) I_n +(1\otimes P_\mu) I_{n+1}=0$
and similarly 
$(K_\mu \otimes 1) I_{n+1} +(1\otimes K_\mu) I_{n}=0$.
\end{proof}

\begin{proof}[Proof of Proposition \ref{phi_covariance}]
Let $\mu,\nu \in \{1,\dots,d \}$.
By Lemma \ref{chi_covariant} and Lemma \ref{lem_In}, 
\begin{align*}
[J_{\mu,\nu},\phi(x)]
&=[J_{\mu,\nu}, \sum_{n=0}^\infty \Bigr( (\chi \otimes \Psi_\al) I_n+ (\Psi_\al^\da \otimes \chi)I_n \Bigr) \\
&=\sum_{n=0}^\infty \Bigr( (\chi \otimes \Psi_\al)(J_{\mu,\nu}\otimes 1) I_n+ (\Psi_\al^\da \otimes \chi)(1\otimes J_{\mu,\nu})I_n \Bigr) \\
&=- \sum_{n=0}^\infty \Bigr( (\chi \otimes 
\Psi_\al)(1\otimes J_{\mu,\nu}) I_n+ (\Psi_\al^\da \otimes \chi)(J_{\mu,\nu}\otimes 1)I_n \Bigr) \\
&=- d_\al (J_{\mu,\nu}) \sum_{n=0}^\infty \Bigr( (\chi \otimes \Psi_\al) I_n+ (\Psi_\al^\da \otimes \chi)I_n \Bigr) \\
&= -(x_\mu \pa_\nu- x_\nu \pa_\mu) \phi(x)
\end{align*}
and similarly $[D,\phi(x)] = (\sum_{\rho =1}^d x_\rho \pa_\rho +\al) \phi(x)$.
Set $I_{-1}=0$.
Then,
\begin{align*}
[P_{\mu},\phi(x)]
&=[P_{\mu}, \sum_{n=0}^\infty \Bigr( (\chi \otimes \Psi_\al) I_n+ (\Psi_\al^\da \otimes \chi)I_n \Bigr) \\
&=\sum_{n=0}^\infty (\chi \otimes \Psi_\al)(P_{\mu} \otimes 1) I_n+ (\Psi_\al^\da \otimes \chi)(1\otimes P_{\mu})I_n \\
&=- \sum_{n=0}^\infty (\chi \otimes \Psi_\al)(1 \otimes P_{\mu}) I_{n+1}+
(\Psi_\al^\da \otimes \chi)(P_{\mu} \otimes 1 )I_{n-1}.
\end{align*}
by $(1 \otimes P_\mu) I_0=0$
\begin{align*}
%- &\sum_{n=0}^\infty (\chi \otimes \Psi_\al)(1 \otimes P_{\mu}) I_{n+1}+
%(\Psi_\al^\da \otimes \chi)(P_{\mu} \otimes 1 )I_{n-1}\\
&=- \sum_{n=0}^\infty (\chi \otimes \Psi_\al)(1 \otimes P_{\mu}) I_{n}+
(\Psi_\al^\da \otimes \chi)(P_{\mu} \otimes 1 )I_{n} \\
&=-d_\al(P_\mu) \sum_{n=0}^\infty \Bigr( (\chi \otimes \Psi_\al) I_n+ (\Psi_\al^\da \otimes \chi)I_n \Bigr) \\
&= \pa_\mu \phi(x).
\end{align*}
Thus, $[P_\mu, \phi(x)]= \pa_\mu \phi(x)$
and similarly, $[K_\mu,\phi(x)]=-d_\al(K_\mu) \phi(x).$
%By Lemma \ref{H_dual_delta},
%$\Delta H_d = \Delta H_d^\vee=0$.
%Thus, $\Delta \phi(x)=0$.
If $\al=\dn$ with $d>2$, then by Proposition \ref{prop_psi_dn} and Remark \ref{rem_harmonic_dagger}, the assertion holds.
\end{proof}

Letting $y$ be another formal variable, we can consider $\sod$-module homomorphisms $\Psi_{\al,y}: L(\al) \rightarrow \R[y_1,\dots,y_d]$, $\Psi_{\al,y}^\da: L(\al)^\da \rightarrow \R[y_1,\dots,y_d,|y|^\R]$ and vertex operator $\phi(y)$.
When no confusion arises, we omit the subscript $\al$ from the notation.
The following proposition is important:
\begin{prop}\label{prop_commutator}
Let $\al$ be generic or $\al=\dn$ with $d>2$. Then,
\begin{align*}
[\phi^+(x),\phi^-(y)] = \exp\left(-\sum_\rho y_\rho \frac{d}{dx_\rho}\right)\frac{1}{||x||^{2\al}}.
\end{align*}
\end{prop}
\begin{proof}
Let $\{e_1^n,\dots,e_{l_n}^n \}$ be a basis of $L(\al)_n$
and $\{f_1^n,\dots,f_{l_n}^n\}$ be the dual basis of $L(\al)_n^\da$ for $n \in \Z_{\geq 0}$.
% Then,
%\begin{align*}
%[\phi^+(x),\phi^-(y)] = \sum_{n \geq 0}\sum_{i=1}^{l_n} ||x||^{-\al-n}
% \Psi_{\al,x}(e_i^n)\Psi_{\al,y}(e_n^i).
%\end{align*}
%Recall that there is a tautological $\sod$-module isomorphism $f_\al:L(\al) \rightarrow L(\al)^\da=\theta^*L(\al)$. Set
%$f_i^n = f_\al(e_n^i)$. 
%Then, by the definitions of the invariant bilinear forms in Proposition \ref{prop_bilinear_L},
%$\{f_i^n \in L(\al)_n^\dagger\}_{i=1,\dots,l_n}$ is the dual basis of $\{e_i^n\}$. 
%
%Let $\{e_1^l,\dots,e_m^l \}$ be a basis of $(S_d)_l$
%and $\{f_1^l,\dots,f_m^l\}$ be the dual basis of $(S_d)_l^*$ for $l \in \Z_{\geq 0}$.
Then,
\begin{align}
\begin{split}
[\phi^+(x), \phi^-(y)]
&=\sum_{n,m=0}^\infty [(\Psi_x^\da \otimes \chi)I_n, 
(\chi \otimes \Psi_y) I_{m}] \\
%&=\sum_{n,m=0}^\infty [(\Psi_x^\vee \otimes \chi)I_n, 
%(\chi \otimes \Psi_y) I_{m}] \\
&=\sum_{n,m=0}^\infty \sum_{i,j} [(\Psi_x^\da \otimes \chi)f_i^n\otimes e_i^n, 
(\chi \otimes \Psi_y) f_j^m\otimes e_j^m] \\
%&=\sum_{l=0}^\infty \sum_{i,j} [P_x^\vee(f_i^l) \chi(e_i^l), 
%P_y(e_j^l) \chi(f_j^l) ] \\
&=\sum_{n,m=0}^\infty \sum_{i,j} \Psi_x^\da(f_i^n)\Psi_y(e_j^m) [\chi(e_i^n), \chi(f_j^m) ] \\
&=\sum_{n,m=0}^\infty \sum_{i,j} \Psi_x^\da(f_i^n)\Psi_y(e_j^m) \langle f_j^m,e_i^n \rangle \\
&=\sum_{n=0}^\infty \sum_{i} \Psi_x^\da(f_i^n)\Psi_y(e_i^n).
\end{split}
\label{eq_commutator_psi}
\end{align}
Assume that $\al$ is generic.
Taking the basis $\{f_i^n\}$ as $\{d_\al(P_{\mu_1})\dots d_\al(P_{\mu_n})||x||^{-2\al}\}$, then the assertion follows from Proposition \ref{prop_psi_generic}. The case $\al =\dn$ is shown in Appendix B.
\end{proof}

We next construct a natural invariant bilinear form $(-,-)_{\Hd}$ on $\Hd$.
Recall that we have an isomorphism \eqref{eq_theta_al}
\begin{align*}
\theta_\al: \theta^* L(\al) \rightarrow L(\al)^\da, v_\al \mapsto v_\al^\da.
\end{align*}
It is easy to see that $\theta_\al$ gives an anti-automorphism $\Theta$ of Lie algebra $\Ad$ by
\begin{align*}
\Theta: a^\da \mapsto  \theta_{\al}^{-1}(a^\da),\quad a\mapsto \theta_\al(a),\quad c\mapsto c,
\end{align*}
which is an involution $\Theta^2=\id_{\Ad}$.
Let $\Hd^\vee$ be the restricted dual of $\Hd$ as in \eqref{eq_restricted_dual}, and regard it as an $\Ad$-module by
$X.u(\bullet) = u(\Theta(X)\bullet)$ for $u\in\Hd^\vee$ and $X\in \Ad$.
Let $\va^* \in \Hd^\vee$ be the unique vector such that $\va^*(\va)=1$ and $\va^*(X\bullet)=0$ for any $X \in L(\al)^\da$.
Then, by the universality of $\Hd$, there is a unique $\Ad$-module homomorphism
$\Hd \rightarrow \Hd^\vee$ which sends $\va$ to $\va^*$.
Since $\Hd$ and $\Hd^\vee$ are simple $\Ad$-modules, we have:
\begin{prop}\label{prop_Hd_bilinear}
There is a unique non-degenerate bilinear form $(-,-)_{\Hd}:\Hd \otimes \Hd \rightarrow \R$ such that:
\begin{enumerate}
\item
$(\va,\va)_{\Hd}=1$;
\item
$(X v,w)_{\Hd}=(v,\Theta(X)w)_{\Hd}$ for any $v,w\in \Hd$ and $X\in \Ad$.
\end{enumerate}
Moreover, $(-,-)_{\Hd}$ is symmetric and it satisfies
\begin{align}
\begin{split}
(v,\phi^-(x)w) &= ||x||^{-2\al}(\phi^+\left(\frac{x}{||x||^2}\right)v,w)\\
(v,\phi^+(x)w) &= ||x||^{-2\al}(\phi^-\left(\frac{x}{||x||^2}\right)v,w)
\end{split}
\label{eq_dual_phi}
\end{align}
for any $v,w\in\Hd$.
\end{prop}
\begin{proof}
We will check \eqref{eq_dual_phi}.
By \eqref{eq_Tal_def}, 
\begin{align*}
(v, (\chi \otimes \Psi_x )(I_n) w)& = ((\chi \otimes \Psi_x )((\Theta\otimes \id)I_n) v, w)\\
&=||x||^{-2\al} ((\chi \otimes \Psi_{x/||x||^2}^\da )((\Theta\otimes \Theta)I_n) v, w)\\
&=||x||^{-2\al} ((\Psi_{x/||x||^2}^\da \otimes \chi )(I_n) v, w).
\end{align*}
Hence, the assertion holds.
\end{proof}

Now, we will construct the vertex operator on $\Hd$. We identify $\Hd$ as the symmetric algebra 
\begin{align*}
\mathrm{Sym}\,L(\al)^\da = \bigoplus_{N \geq 0}
\mathrm{Sym}^N L(\al)^\da
\end{align*}
as $\sod$-modules,
where $\mathrm{Sym}^N L(\al)^\da$ is the space of symmetric polynomials of degree $N$ in $L(\al)^\da$.
By using $\phi(x)$, we construct the vertex operator
\begin{align} 
Y(-,x):\Hd \rightarrow \End\Hd[[x_1,\dots,x_d,|x|^\R]]\label{eq_vertex_def} 
\end{align} 
by induction on $N$.

%\subsection{Construction of vertex operator}
%\label{sec_construction2}
%We continue with the notations from the previous section.
%Let $\al$ be generic or $\al=\dn$ with $d\geq 3$.
%We will identify $\Hd$ as the symmetric algebra 
%\begin{align*}
%\mathrm{Sym}\,L(\al)^\da = \bigoplus_{N \geq 0}
%\mathrm{Sym}^N L(\al)^\da
%\end{align*}
%as $\sod$-modules,
%where $\mathrm{Sym}^N L(\al)^\da$ is the space of symmetric polynomials of degree $N$ in $L(\al)^\da$.
%In this section, we use $\phi(x)$ to construct the vertex operator: 
%\begin{align} 
%Y(-,x):\Hd \rightarrow \End\Hd[[x_1,\dots,x_d,|x|^\R]]\label{eq_vertex_def} 
%\end{align} 
%by induction on $N$.

%この章ではまず初めに$\phi(x)$を用いて、頂点作用素:
%\begin{align}
%Y(-,x):\Hd \rightarrow \End\Hd[[x^1,\dots,x^d,|x|^\R]]\label{eq_vertex_def}
%\end{align}
%を次数$N$についての帰納法により構成する。
In the case of $N=0$, $\R \1 = \mathrm{Sym}^0 L(\al)^\da$, set
\begin{align*}
Y(\1, x) =\mathrm{id}_{\Hd}.
\end{align*}
Define a linear map $F_x^\pm: L(\al)^\dagger = \mathrm{Sym}^1 L(\al)^\dagger \rightarrow \Hom_\R\left(\Hd, \Hd((x_1,\dots,x_d,|x|^\R)) \right)$ by
\begin{align*}
F_x^\pm(P_{\mu_1}\dots P_{\mu_n}v_\al^\da) &= \frac{d}{dx_{\mu_1}}\dots \frac{d}{dx_{\mu_n}} \phi^\pm(x)
\end{align*}
for any $1\leq \mu_1 \leq \dots \leq \mu_n \leq d$,
which is well-defined as a linear map for $\al=\dn$ by \eqref{eq_laplace_phi} and Theorem \ref{prop_singular_vector}.

Then, define the linear map \eqref{eq_vertex_def} by induction on $N$ as follows:
Assume that $N>0$ and that the vertex operator is defined up to $N-1$.
%$N>0$であって、$N-1$までは頂点作用素が定義されていると仮定する。
Then, for $a_1, a_2, a_N \in L(\al)^\da$, set
\begin{align}
Y(a_1 a_2\dots a_N,x) = Y(a_2 \dots a_N,x) F_x^+(a_1) + F_x^-(a_1)Y(a_2 \dots a_N,x).\label{eq_normally}
\end{align}
Since $\phi^+(x)v \in \mathrm{Sym}^{M-1} L(\al)^\da[x_1,\dots,x_d,|x|^\R]$ for any $v \in \mathrm{Sym}^{M} L(\al)^\da$ and $M>0$,
\eqref{eq_normally} is well-defined, that is, each coefficient of $x$ is a finite sum as elements in $\End\Hd$.
Since, for example,
\begin{align}
Y(a_1a_2,x) = F_x^-(a_1)F_x^-(a_2)+F_x^-(a_1)F_x^+(a_2)+F_x^+(a_1)F_x^+(a_2)+F_x^-(a_2)F_x^+(a_1),
\label{eq_example_normal_order}
\end{align}
and $[F_x^+(a),F_x^+(a')]=0$, $[F_x^-(a),F_x^-(a')]=0$,
\eqref{eq_normally} is independent of the order of the products of $a_1 \dots a_N$.
Hence, \eqref{eq_normally} defines a linear map:
\begin{align*}
\mathrm{Sym}^N L(\al)^\da \rightarrow \End\, \Hd[[x_1,\dots,x_d,|x|^\R]].
\end{align*}

%$F_x^+(a)$は次数を下げる微分作用素からなり消滅演算子と呼ばれ、$F_x^-(a)$は生成演算子と呼ばれる。
%\eqref{eq_example_normal_order}にあるように、頂点作用素は消滅演算子を右側に、生成演算子を左側に並べなおすことで定義される。
%このような合成は正規順序積と呼ばれる。

%In this section, we will prove the following result:
%\begin{prop}\label{thm_vertex_d}
%The vertex operator $Y(-,x):\Hd \rightarrow \End\,\Hd[[x_1,\dots,x_d,|x|^\R]]$ given in \eqref{eq_normally} and $\1 \in \Hd$ satisfies Definition \ref{def_vertex_prealgebra}, and thus, defines a conformal b $d$-vertex algebra.
%\end{prop}

\begin{lem}\label{lem_vertex_creation}
The vertex operator $Y(-,x)$ satisfies Definition (V1), that is, $Y(v,x)\1 \in \Hd[[x_1,\dots,x_d]]$ and
\begin{align*}
\lim_{x\to 0}Y(v,x)\1 =v
\end{align*}
for any $v\in \Hd$.
\end{lem}
\begin{proof}
Since $\pa_{\mu_1}\dots \pa_{\mu_n}\phi^+(x)\1=0$,
\begin{align*}
Y(a_1\dots a_n,x)\1= F_x^-(a_1)\dots F_x^-(a_n)\1 \in \Hd[[x^1,\dots,x^d]].
\end{align*}
By Proposition \ref{phi_covariance} and \eqref{eq_trivial_sod_vac}, we have
\begin{align*}
\lim_{x\to 0}Y(P_{\mu_1}\dots P_{\mu_n} v_\al^\da,x)\1&=
\lim_{x \to 0} \pa_{\mu_1}\dots \pa_{\mu_n}\phi^-(x)\1\\
&=\lim_{x \to 0} [P_{\mu_1},[P_{\mu_2},\dots [P_{\mu_n},\phi^-(x)]\dots ]\1\\
&=P_{\mu_1} \dots P_{\mu_n}(\lim_{x\to 0}\phi^-(x)\va) = P_{\mu_1}\dots P_{\mu_n} v_\al^\da.
\end{align*}
Hence, the assertion holds.
\end{proof}

\begin{lem}\label{lem_vertex_cov}
The vertex operator $Y(-,x)$ is conformally covariant.
\end{lem}
\begin{proof}
We first show $Y(P_\mu v,x)=\frac{d}{dx_\mu} Y(v,x)$ for any $v\in \Sym^N L(\al)^\da$ and $\mu =1,\dots,d$ by the induction on $N$. In the case of $N=0$, the claim follows from $Y(P_\mu \1,x)=\frac{d}{dx_\mu} \id_{\Hd}=0$. The case of $N=1$ follows from the definition. By the induction assumption, we have
\begin{align*}
\frac{d}{dx_\mu} Y(a_1 a_2\dots a_N,x)&= \frac{d}{dx_\mu} (Y(a_2 \dots a_N,x) F_x^+(a_1) + F_x^-(a_1)Y(a_2 \dots a_N,x))\\
&=Y(P_\mu \cdot (a_2 \dots a_N),x) F_x^+(a_1) + F_x^-(a_1)Y(P_\mu\cdot( a_2 \dots a_N),x)\\
&+Y(a_2 \dots a_N,x) F_x^+(P_\mu a_1) + F_x^-(P_\mu a_1)Y(a_2 \dots a_N,x)\\
&=Y(P_\mu(a_1 a_2\dots a_N),x).
\end{align*}
Hence, the claim holds.
Next, for any $A \in \sod$, we will show that $[A, Y(v,x)]-Y(\exp(-\sum_\rho x_\rho \ad P_\rho)A v,x)=0$ for any $v \in \Sym^N L(\al)^\da$ by the induction on $N$. The case of $N=0$ is trivial.
By Proposition \ref{phi_covariance} and Lemma \ref{lem_covariance_Lie},
\begin{align*}
[A,F_x^\pm(v_\al^\da)] = F_x^\pm((\exp(-\sum_\rho x_\rho \ad P_\rho)A) v_\al^\da)
\end{align*}
holds since $K_\mu v_\al^\da= J_{\mu\nu}v_\al^\da=0$ and $D v_\al^\da=\al v_\al^\da$.
Hence,
\begin{align}
\begin{split}
[A, F_x^\pm(P_{\mu_1} \dots P_{\mu_n} v_\al^\da)]&= \pa_{\mu_1}\dots \pa_{\mu_n} [A,F_x^\pm(v_\al^\da)]\\
&= \pa_{\mu_1}\dots \pa_{\mu_n} F_x^\pm((\exp(-\sum_\rho x_\rho \ad P_\rho)A) v_\al^\da)\\
&= F_x^\pm((\exp(-\sum_\rho x_\rho \ad P_\rho)A) P_{\mu_1} \dots P_{\mu_n}v_\al^\da).
\end{split}
\label{eq_com_A_cov}
\end{align}
Thus, the case of $N=1$ holds. For $N>1$, by the induction assumption, we have
\begin{align*}
&[A, Y(a_1 a_2\dots a_N,x)]\\
&= [A, (Y(a_2 \dots a_N,x) F^+(a_1) + F^-(a_1)Y(a_2 \dots a_N,x))]\\
&=Y((\exp(-\sum_\rho x_\rho \ad P_\rho)A) \cdot (a_2 \dots a_N),x) F_x^+(a_1) + F_x^-(a_1)Y((\exp(-\sum_\rho x_\rho \ad P_\rho)A)\cdot( a_2 \dots a_N),x)\\
&+Y(a_2 \dots a_N,x) F_x^+((\exp(-\sum_\rho x_\rho \ad P_\rho)A) a_1) + F_x^-((\exp(-\sum_\rho x_\rho \ad P_\rho)A) a_1)Y(a_2 \dots a_N,x)\\
&=Y((\exp(-\sum_\rho x_\rho \ad P_\rho)A)(a_1 a_2\dots a_N),x).
\end{align*}
Hence, the assertion holds.
\end{proof}

\subsection{Construction of \(d\)-CC systems}
\label{sec_construction3}
In this section, based on the vertex operator constructed in the preceding section, we construct a $d$-CC system.
The most technical part is the convergence of the sums appearing in the cluster decomposition. We first study this convergence problem using Gegenbauer polynomials.

%この章では先ほどの章で構成した vertex operator を基に、$d$-CC system を構成する。
%最もテクニカルなパートは cluster decomposition に現れる和の収束性である。まず初めにこの収束性の問題を Gegenbauer polynomial を用いて考える。

For $\al \in \R$,
the Gegenbauer polynomials $\{C_n^{(\al)}(t) \in \R[t] \}_{n =0,1,\dots}$
are defined in terms of their generating function
\begin{align}
\frac{1}{(1-2rt+r^2)^{\al}}= \sum_{n=0}^\infty C_n^{(\al)}(t)r^n.
\label{eq_app_Gegen1}
\end{align}
\begin{lem}\label{lem_conv_prim}
For any $t \in [-1,1]$, the series $\sum_{n=0}^\infty C_n^{(\al)}(t)z^n$  is absolutely and locally uniformly convergent in $\{z\in\C\mid |z|<1\}$.
\end{lem}
\begin{proof}
Set $t=\cos(\theta)$. The assertion follows from $(1-2zt+z^2)=(1-e^{i\theta}z)(1-e^{-i\theta}z)$.
\end{proof}
The importance of these polynomials comes from the following observation.
Set
\begin{align*}
(x,y)=\sum_{i=1}^d x_i y_i.
\end{align*}
Since 
\begin{align}
\frac{1}{|x-y|^{2\al}}&= ((x_1-y_1)^2+\dots+(x_d-y_d)^2)^{-\al} =||x||^{-2\al}\left(1-2\frac{(x,y)}{|x|^2}+\frac{|y|^2}{|x|^2}\right)^{-\al},
\label{eq_app_Gegen2}
\end{align}
applying the generating function of the Gegenbauer polynomials
with $t=\frac{(x,y)}{|x||y|}, r=\frac{|y|}{|x|}$,
\begin{align}
\frac{1}{|x-y|^{2\al}}\Bigl|_{|x|>|y|}=\sum_{l=0}^\infty C_l^{(\al)}\left(\frac{(x,y)}{|x||y|}\right)\frac{|y|^l}{|x|^{l+2\al}},
\label{eq_app_Gegen3}
\end{align}
where the left-hand-side converges in $|x|>|y|$ by Lemma \ref{lem_conv_prim}.
To obtain the strong cluster decomposition, we need the following convergence property, which is stronger than Lemma \ref{lem_conv_prim}.

We note that, for each \(x\neq 0\),
$C_l^{(\alpha)}
\left(
\frac{(x,y)}{|x||y|}
\right)
\frac{|y|^l}{|x|^l}
$
extends to a homogeneous polynomial of degree \(l\) in \(y\). Indeed, since \(C_l^{(\alpha)}(t)\) is a linear combination of \(t^{l-2k}\), the apparent singularity at \(y=0\) cancels. In particular, the right-hand side of \eqref{eq_app_Gegen3} is well defined also at \(y=0\).
%strong cluster decomposition を得るためには、我々はLemma \ref{lem_conv_prim}より強い収束性が必要である。
\begin{lem}\label{lem_Gegenbauer_convergence}
Set
\[
 A_n^{(\alpha)}(x,y)
 =
 |x|^{-2\alpha}
 \left(\frac{|y|}{|x|}\right)^n
 C_n^{(\alpha)}
 \left(\frac{(x,y)}{|x||y|}\right).
\]
%which is a polynomial of degree $n$ in $y$.
%where the right-hand side is understood as its homogeneous
%polynomial extension in $y$.
Then, for any multi-indices $I,J$,
\[
 \sum_{n\geq0}
 \partial_x^I\partial_y^J
 A_n^{(\alpha)}(x,y)
\]
is absolutely and locally uniformly convergent in $|x|>|y|$.
Moreover,
\[
 \partial_x^I\partial_y^J
 \frac1{|x-y|^{2\alpha}}
 =
 \sum_{n\geq0}
 \partial_x^I\partial_y^J
 A_n^{(\alpha)}(x,y).
\]
\end{lem}
\begin{proof}

Let $K$ be a compact subset of $\{|x|>|y|\}$ and put
$q=\sup_{(x,y)\in K}\frac{|y|}{|x|}<1$.
Choose $R>1$ such that $Rq<1$.
For a complex variable $z$, set
\begin{align*}
Q_{x,y}(z)= |x|^2-2z(x,y)+z^2|y|^2,
\end{align*}
which is a holomorphic function on $z$ and 
$Q_{x,y}(z) \neq 0$ for any $|z|<R$ by the proof of Lemma \ref{lem_conv_prim}.
Thus, by choosing the branch of $Q_{x,y}(z)^{-\al}$ so that its value at $z=0$ is $\frac{1}{|x|^{2\al}} \in \R_{>0}$, we obtain
\begin{align*}
F(z;x,y) = \frac{1}{Q_{x,y}(z)^\al}
\end{align*}
as a holomorphic function on $|z|<R$. Moreover, this function depends smoothly on $(x,y)\in K$.
%よって $Q_{x,y}(z)^{-\al}$の分岐を$z=0$において$\frac{1}{|x-y|^{2\al}} \in \R_{>0}$になるように選ぶことで、
%\begin{align*}
%F(z;x,y) = \frac{1}{Q_{x,y}(z)^\al}
%\end{align*}
%が、$|z|<R$上の正則関数として定まる。また、この関数は$(x,y)\in K$には滑らかに依存する。
%Since $|x|^2+2z(x,y)+|y|^2 \neq 0$ if $|z|<R$,
%$F(z;x,y)$ is a holomorphic function on $|z|<R$.
%ここで我々は$F(0;x,y) = \frac{1}{|x-y|^{2\al}}$となる branch を選んだ。
By applying \eqref{eq_app_Gegen1}
with $t=\frac{(x,y)}{|x||y|}$ and $r=z\frac{|y|}{|x|}$
and by Lemma \ref{lem_conv_prim},
\[
F(z;x,y)
 =
 \sum_{n\geq0} A_n^{(\alpha)}(x,y)z^n,
\]
where the right-hand side converges uniformly in $\{|z|<R\}$.
For any multi-indices $I,J$, the function
$\partial_x^I\partial_y^J F(z;x,y)$ is uniformly bounded for $(x,y)\in K$ and $|z|=R$.
Hence, by the Cauchy estimate for the coefficients,
there is a constant $C_{I,J,K}>0$ such that
\[
 \sup_{(x,y)\in K}
 \left|
 \partial_x^I\partial_y^J A_n^{(\alpha)}(x,y)
 \right|
 \leq C_{I,J,K}R^{-n}.
\]
Consequently,
$\sum_{n\geq0}
 \partial_x^I\partial_y^J A_n^{(\alpha)}(x,y)$
is absolutely and uniformly convergent on $K$.
\end{proof}
We now return to the construction of the $d$-CC system. By Proposition \ref{prop_commutator} and Lemma \ref{lem_Gegenbauer_convergence},
\begin{align*}
[F_x^+(v_\al^\da),F_y^-(v_\al^\da)] = \sum_{n \geq 0} A_n^{(\al)}(x,y)
\end{align*}
and hence, for any $\mu_1,\dots,\mu_N, \nu_1,\dots,\nu_M\in \{1,\dots,d\}$,
\begin{align}
[F_x^+(P_{\mu_1}\cdots P_{\mu_N} v_\al^\da),F_y^-(P_{\nu_1}\cdots P_{\nu_M} v_\al^\da)] = \sum_{n \geq 0}  \pa_{x_{\mu_1}}\cdots \pa_{x_{\mu_N}}\pa_{y_{\nu_1}}\cdots \pa_{y_{\nu_M}} A_n^{(\al)}(x,y)
\label{eq_general_com}
\end{align}
follows. The right-hand side of \eqref{eq_general_com} converges, for $|x|>|y|$, to $\pa_{x_{\mu_1}}\cdots \pa_{x_{\mu_N}}\pa_{y_{\nu_1}}\cdots \pa_{y_{\nu_M}}\frac{1}{|x-y|^{2\al}}$, and in particular admits an analytic continuation to a real analytic function on $\Conf_2(\R^d)$. Therefore, one can define the linear map
\begin{align*}
\kappa: L(\alpha)^\dagger\otimes L(\alpha)^\dagger \rightarrow C^\om(\Conf_2(\R^d),\C),\quad (a,b)\mapsto \kappa_{a,b}(x,y)
\end{align*}
by $\kappa_{a,b}(x,y)|_{|x|>|y|} = [F_x^+(a),F_y^-(b)]$.
% We note that, by Lemma \ref{lem_Gegenbauer_convergence}, the sum $[F_x^+(a),F_y^-(b)]$, taken with respect to the degree in $y$, converges absolutely and uniformly for $|x|>|y|$.
We denote by $A_{a,b}^{|x|>|y|}(x,y)$ the formal series defined by $[F_x^+(a),F_y^-(b)]$.
The following is immediate from the definition.

\begin{lem}\label{lem_kappa_sym}
For any $a,b\in L(\al)^\da$,
\begin{align*}
\kappa_{a,b}(x,y)=\kappa_{b,a}(y,x).
\end{align*}
\end{lem}
We emphasize that although $A_{a,b}^{|x|>|y|}(x,y)$ and $A_{b,a}^{|y|>|x|}(y,x)$ have the same analytic continuation, they are entirely different series.

\begin{rem}\label{rem_completion}
%Let $x_0,y_0 \in \R^d$ and $r,s>0$ with
%\begin{align*}
%\overline{B_r(x_0)} \cap \overline{B_s(y_0)}=\emptyset,\quad B_r(x_0), B_s(y_0)\subset B_1(0),
%\end{align*}
%where $B_t(z)$ is an open ball with radius $t>0$ and center $z \in\R^d$.
%In \cite[Theorem 2.24]{M9}, $\al=\frac{d-2}{2}$の場合に we proved that the linear map
%\begin{align*}
%L(\alpha)^\dagger\otimes L(\alpha)^\dagger \rightarrow \C,\quad (a,b)\mapsto \kappa_{r^{D}a,s^D b}(x,y)
%\end{align*}
%is bounded with respect to the invariant bilinear form on $L(\alpha)^\dagger$.
%さらに、これを用いて$\kappa$はd次元 unit disk のある幾何学的条件を満たす任意の共形埋め込みに対して拡張され、$d$次元の共形場理論の幾何学的な定義を可能にする。
%こうした幾何学的拡張と共形ラプラシアンに付随する因子化代数との関係は\cite{M10}もみよ。
Assume $d>2$ and set $\alpha=\frac{d-2}{2}$.
Let $x_0,y_0\in\R^d$ and $r,s>0$ satisfy
\[
\overline{B_r(x_0)}\cap\overline{B_s(y_0)}
=\emptyset,\quad
B_r(x_0),B_s(y_0)\subset B_1(0),
\]
where $B_r(x)$ denotes the open ball of radius $r$
centered at $x$.
In \cite[Theorem 2.24]{M9}, we proved that the linear map
\[
L(\alpha)^\dagger\otimes L(\alpha)^\dagger
\rightarrow \R,\quad
a\otimes b\mapsto
\kappa_{r^D a,s^D b}(x_0,y_0)
\]
extends to a bounded linear functional on the complete tensor product of the Hilbert space completions of
$L(\alpha)^\dagger$ with respect to its positive-definite invariant bilinear form.
Moreover, this construction extends to any pair of conformal embeddings of the unit ball
into itself whose images have disjoint closures \cite[Theorem 2.31]{M9}.
For its relation to the prefactorization algebra
associated with the conformal Laplacian,
see \cite{Mfactorization}.
\end{rem}

Using this map $\kappa$, we construct a $d$-CC system.
For $a_i \in L(\al)^\da$, from the definition, the vertex operator $Y(a_1 \dots a_r,x)$ is the product of $F_{x}(a_i)$. The order of the product is such that the annihilation operator $F_x^+(a_i)$ is on the right and the creation operator $F_x^-(a_i)$ is on the left (see \eqref{eq_example_normal_order}). The product of vertex operators of such an order is called the {\bf normal ordered product}.
In general, we can also consider normal ordered products between vertex operators with different formal variables. Such a normal ordered product is denoted by $:Y(v_1,x_1)Y(v_2,x_2) \dots Y(v_n,x_n):$, e.g.,
%一般に異なる変数を持つ頂点作用素の間の正規順序積を考えることもできる。こうした正規順序積を$:Y(v_1,x_1)Y(v_2,x_2) \dots Y(v_n,x_n):$とかく。
\begin{align*}
:Y(a_1\dots a_n,x_1) Y(b_1\dots b_m,x_2): 
= \sum_{[m]=I_1 \sqcup I_2} (\Pi_{i \in I_1} F_{x_2}^-(b_i)) Y(a_1\dots a_n,x_1) 
(\Pi_{j \in I_2} F_{x_2}^+(b_j)).
\end{align*}

The formula that connects the ordinary product and the normal ordered product is called the {\bf Wick formula}, which is well-known in physics, and can describe their relationship inductively. For example,
\begin{align*}
Y(a_1\dots a_n,x_1) Y(b,x_2) &= Y(a_1\dots a_n,x_1) (F_{x_2}^+(b)+F_{x_2}^-(b))\\
&=Y(a_1\dots a_n,x_1) F_{x_2}^+(b)+[Y(a_1\dots a_n,x_1) ,F_{x_2}^-(b)]+F_{x_2}^-(b)Y(a_1\dots a_n,x_1)\\
&=:Y(a_1\dots a_n,x_1)Y(b,x_2): +[Y(a_1\dots a_n,x_1) ,F_{x_2}^-(b)]\\
&=:Y(a_1\dots a_n,x_1)Y(b,x_2): +
 \sum_{i =1}^n A_{a_i,b}^{|x_1|>|x_2|}(x_1,x_2)Y(a_1\dots \hat{a_i}\dots a_n,x_1).
\end{align*}
%Since
%\begin{align*}
%[Y(a_1\dots a_n,x_1) ,Y(b,x_2)] = \sum_{i =1}^n \kappa_{a_i,b}(x_1,x_2) Y(a_1\dots \hat{a_i}\dots a_n,x_1),
%\end{align*}
%we have
%\begin{align*}
%[Y(a_1\dots a_n,x_1) ,F_{x_2}^+(P_{\mu_1} \dots P_{\mu_k}v_\al^\da)] = \pa_{\mu_1}^{x_2} \dots \pa_{\mu_k}^{x_2} [Y(a_1\dots a_n,x_1),\phi^+(x_2)]
%\end{align*}
%and $[F_{x_1}^-(a_i),\phi^+(x_2)]$ is the derivative of $\frac{1}{||x_1-x_2||^\al}|_{|x_1|>|x_2|}$ by Proposition \ref{prop_commutator}, we have
%\begin{align*}
%&Y(a_1\dots a_n,x_1) Y(b,x_2) =:Y(a_1\dots a_n,x_1)Y(b,x_2): \\
%&+
%\sum_{k=1}^n Y(a_1\dots \hat{a_k}\dots a_n,x_1) \times \text{the derivative of $\frac{1}{||x_1-x_2||^\al}\Bigl|_{|x_1|>|x_2|}$ determined from $a_k$ and $b$}.
%\end{align*}
In general, $Y(a_1\dots a_n,x_1) Y(b_1\dots b_m,x_2)$ is the sum of all possible contractions of $a_i$ and $b_j$ and the rest of the normal ordered products.
%Although the Wick formula can be written explicitly in combinatorial
%terms, we omit it here.
It is important to note that the coefficients appearing in the normal
ordered products are independent of the ordering of the fields after
analytic continuation by Lemma \ref{lem_kappa_sym}.
%It is important to note that the coefficients of the normal ordered product in the Wick formula does not depend on the order of the product up to analytic continuations, that is, they differs only by the domain of the expansions.
For example, 
\begin{align*}
Y(a,x)Y(b,y) &=:Y(a,x)Y(b,y): +A_{a,b}^{|x|>|y|}(x,y),\\
Y(b,y)Y(a,x) &=:Y(a,x)Y(b,y): + A_{b,a}^{|y|>|x|}(y,x).
\end{align*}
%$Y(v_\al^\da v_\al^\da,x_2) Y((P_\mu v_\al^\da) v_\al^\da,x_1)
%=:Y((P_\mu v_\al^\da) v_\al^\da,x_1)Y(v_\al^\da v_\al^\da,x_2): 
%+ 2:Y(P_\mu v_\al^\da,x_1)Y(v_\al^\da,x_2):\frac{1}{||x_1-x_2||^\al}\Bigl|_{|x_2|>|x_1|}
%+ 2:Y(v_\al^\da,x_1)Y(v_\al^\da,x_2):\frac{-2\al x_1^\mu}{||x_1-x_2||^{\al+1}}\Bigl|_{|x_2|>|x_1|}+\frac{-4\al x_1^\mu}{||x_1-x_2||^{2\al+1}}\Bigl|_{|x_1|>|x_2|}.$
%Here, we used $\kappa_{a,b}(x,y)=\kappa_{b,a}(y,x)$ as functions on $\Conf_2(\R^d)$.
%
%.
%More precisely, if 
%$Y(a_1\dots a_n,x_1) Y(b_1\dots b_m,x_2)$ and $Y(b_1\dots b_m,x_2)Y(a_1\dots a_n,x_1)$ is a coefficient 
%$\frac{ 1}{||x_1-x_2|||^\al}$, which are the same except for the domain of the expansion of the derivative.
%In other words, they define the same operator-valued function except for the analytic connection.
%Wick formula を組み合わせ論的に明示的与えることは簡単であるが、ここでは省略する。
%重要なことは正規順序積の中身は積の順序に寄らないことであり、よって
%$Y(a_1\dots a_n,x_1) Y(b_1\dots b_m,x_2)$と$Y(b_1\dots b_m,x_2)Y(a_1\dots a_n,x_1)$の正規順序積は、係数である
%$\frac{1}{||x_1-x_2||^\al}$の微分の展開の領域を除けば同じである。
%言い換えると、これらは解析接続を除いて同じ作用素値の関数を定めている。
We define
\begin{align}
\m:\Conf_n(\R^d) \times \Hd^{\otimes n} \rightarrow \overline{\Hd}\label{eq_m_nop}
\end{align}
by thinking of the coefficients in the Wick formula as real analytic functions without series expansion, e.g.,
\begin{align*}
\m_{x_{[3]}}&(a,b,c)=:Y(a,x_1)Y(b,x_2)Y(c,x_3):\va+
\kappa_{a,b}(x_1,x_2) Y(c,x_3)\va \\
&+\kappa_{a,c}(x_1,x_3) Y(b,x_2)\va
+\kappa_{b,c}(x_2,x_3) Y(a,x_1)\va.
\end{align*}
%\begin{align*}
%\m_{x_{[3]}}&(v_\al^\da v_\al^\da,v_\al,v_\al)=:\phi(x_1)^2 \phi(x_2) \phi(x_3):\va+\frac{1}{||x_2-x_3||^\al}:\phi(x_1)^2 :\va\\
%&+\frac{2}{||x_1-x_2||^\al}:\phi(x_1)\phi(x_3):\va
%+\frac{2}{||x_1-x_3||^\al}:\phi(x_1)\phi(x_2):\va
%+\frac{2}{||x_1-x_2||^\al||x_1-x_3||^\al}\va.
%\end{align*}

\begin{prop}\label{prop_wick_correlation}
For any \(n\geq 1\), the map
\[
\m:\Conf_n(\R^d)\rightarrow
\Hom_\R(H_{d,\alpha}^{\otimes n},\overline{H_{d,\alpha}}),
\qquad
x_{[n]}\mapsto \m_{x_{[n]}}
\]
has the following properties.
\begin{enumerate}
\item For any \(u\in H_{d,\alpha}^\vee\) and \(v_i\in H_{d,\alpha}\),
\[
x_{[n]}\mapsto
\langle u,\m_{x_{[n]}}(v_1,\ldots,v_n)\rangle
\]
is real analytic on \(\Conf_n(\R^d)\).
\item The map $\m$ is permutation invariant (see Definition \ref{def_conformal_algebra}).
\item The composition of vertex operators $Y(v_1,x_1)\cdots Y(v_n,x_n)\va$ is absolutely and 
locally uniformly convergent in \(|x_1|>\cdots>|x_n|\)
and satisfies
\[
\m_{x_{[n]}}(v_1,\ldots,v_n)|_{|x_1|>\cdots>|x_n|}
=
Y(v_1,x_1)\cdots Y(v_n,x_n)\va
\]
for any $v_i \in H_{d,\alpha}$.
\end{enumerate}
\end{prop}
\begin{proof}
For any $u \in H_{d,\al}^\vee$ and $v_1,\dots,v_n\in H_{d,\al}$,
\begin{align*}
\langle u, :Y(v_1,x_1) \dots Y(v_n,x_n):\va\rangle
\end{align*}
is just a polynomial, and in particular, a real analytic function on $\Conf_n(\R^d)$.
Hence, (1) holds. 
Since normal ordered products are independent of the order, (2) follows from Lemma \ref{lem_kappa_sym}.
(3) follows from the definition and Lemma \ref{lem_Gegenbauer_convergence}.
\end{proof}

\begin{lem}\label{lem_wick_covariance}
The maps \(\m_{x_{[n]}}\) satisfy the translation invariance and conformal
covariance axioms in Definition \ref{def_conformal_algebra}.
Moreover,
\begin{align}
\m_{(0)}(v)=v,\qquad
\m_{x_{[n]},y} (v_1,\ldots,v_n,\mathbf 1)=\m_{x_{[n]}}(v_1,\ldots,v_n).
\label{eq_vacuum_prop_lem}
\end{align}
\end{lem}

\begin{proof}
\eqref{eq_vacuum_prop_lem} follows immediately from the definition of
normal ordered products and from \(F_x^+(a)\mathbf 1=0\) and Lemma \ref{lem_vertex_creation}.
The translation and conformal covariance follow from Lemma \ref{lem_vertex_cov}.
In fact, the required identities hold in the non-empty region \(|x_1|>\cdots>|x_n|>0\) by Proposition \ref{prop_wick_correlation}.
Since both sides are real analytic functions on \(\Conf_n(\R^d)\), they hold everywhere.
\end{proof}

\begin{lem}\label{lem_wick_cluster}
The maps \(\m_{x_{[n]}}\) satisfy the strong cluster decomposition in the
form of Proposition \ref{prop_scd0}.
\end{lem}

\begin{proof}
Let $a_1,\dots,a_{n+m}\in H_{d,\al}$ with $Da_i=\delta_i a_i$.
Set $\delta=\sum_{j=1}^m \delta_{n+j}$.
For $u\in\Hd^\vee$, consider
\[
F(t)=t^\delta
\left\langle u,
\m_{x_{[n]},t y_{[m]}}
(a_1,\ldots,a_n,a_{n+1},\ldots,a_{n+m})
\right\rangle,
\qquad t>0.
\]

By the Wick formula, $F(t)$ is a finite sum
of products of contractions and normal ordered
products. Contractions involving only the
$x$-variables are independent of $t$, while
those involving only the $y$-variables are
homogeneous in $t$. Each mixed contraction
admits an absolutely and locally uniformly
convergent expansion in $t$ for
\begin{align}
|t|\max_j|y_j|<\min_i|x_i|
\label{eq_add_final_region}
\end{align}
by Lemma~\ref{lem_Gegenbauer_convergence}. Moreover, after pairing with
$u$, the normal ordered terms are polynomials
in $t$.
Hence, we obtain an expansion
\[
F(t)=\sum_{\Delta}c_\Delta(x_{[n]},y_{[m]})t^\Delta,
\]
which is absolutely and locally uniformly
convergent in \eqref{eq_add_final_region}.
In particular, it converges absolutely and
locally uniformly at $t=1$ on $U^0_{n,m}$.
%It remains to identify the coefficients.
On the open region
\begin{align}
|x_1|>\cdots>|x_n|>|y_1|>\cdots>|y_m|>0,
\label{eq_add_cluster_open_sp}
\end{align}
Proposition~\ref{prop_wick_correlation} (3) and Lemma~\ref{lem_vertex_creation} give
\[
\begin{aligned}
F(t)&=t^\delta
\left\langle u,
Y(a_1,x_1)\cdots Y(a_n,x_n)Y(a_{n+1},ty_1)\cdots Y(a_{n+m},ty_m)\va \right\rangle\\
&=t^\delta \sum_\Delta \left\langle u,
Y(a_1,x_1)\cdots Y(a_n,x_n)\pr_\Delta Y(a_{n+1},ty_1)\cdots Y(a_{n+m},ty_m)\va \right\rangle\\
&=
t^\delta \sum_\Delta \big\langle u,
\m_{x_{[n]},0}\big(
a_1,\ldots,a_n,
\operatorname{pr}_\Delta
\m_{ty_{[m]}}(a_{n+1},\ldots,a_{n+m})
\big)
\big\rangle\\
&=
\sum_\Delta t^{\Delta} \big\langle u,
\m_{x_{[n]},0}\big(
a_1,\ldots,a_n,
\operatorname{pr}_\Delta
\m_{y_{[m]}}(a_{n+1},\ldots,a_{n+m})
\big)
\big\rangle,
\end{aligned}
\]
where we used $\operatorname{pr}_\Delta
\m_{t y_{[m]}}(w_1,\ldots,w_m)
=
t^{\Delta-\delta}
\operatorname{pr}_\Delta
\m_{y_{[m]}}(w_1,\ldots,w_m)$ in the last line, which follows from the dilation covariance.
By uniqueness of the expansion in $t$,
\[
c_\Delta(x_{[n]},y_{[m]})
=
\big\langle u,
\m_{x_{[n]},0}\big(
a_1,\ldots,a_n,
\operatorname{pr}_\Delta
\m_{y_{[m]}}(a_{n+1},\ldots,a_{n+m})
\big)
\big\rangle
\]
holds on \eqref{eq_add_cluster_open_sp}. Since both sides are real analytic on $U^0_{n,m}$, the equality holds on $U^0_{n,m}$. Setting $t=1$ proves the assertion.

%It suffices to prove the assertion for monomials
%\(v_i=a_{i,1}\cdots a_{i,N_i}\).
%We consider the region
%\[
%U^0_{n,m}=
%\{(x_1,\ldots,x_n,y_1,\ldots,y_m)\mid
%\min_i |x_i|>\max_j |y_j|\}.
%\]
%The Wick expansion of
%\[
%\m_{x_{[ n]},y_{[m]}}(v_1,\ldots,v_n,w_1,\ldots,w_m)
%\]
%is a finite sum of products of contractions and normal ordered products. The contractions are of three
%types: those among the \(x\)-variables, those among the \(y\)-variables, and
%those connecting an \(x\)-variable with a \(y\)-variable.
%The first two types are unchanged in the cluster expansion. The only
%difference is that each mixed contraction
%$\kappa_{a,b}(x_i,y_j)$ is replaced by its expansion in the region \(|x_i|>|y_j|\). By Lemma
%\ref{lem_Gegenbauer_convergence}, this expansion is locally uniformly and
%absolutely convergent on \(U^0_{n,m}\).
%Grouping the resulting terms according to the sums over the total degrees gives precisely
%\[
%\sum_{\Delta}
%\m_{x_{[n]},0}
%\left(
%v_1,\ldots,v_n,
%\pr_\Delta \m_{y_{[m]}}(w_1,\ldots,w_m)
%\right).
%\]
\end{proof}

Now, we have the following result:
\begin{thm}\label{thm_CF_al}
Let $\alpha>0$ and $d\geq 2$. Assume that either \(\alpha\) is generic or $\alpha=\frac{d-2}{2}$.
Then, the triple \((H_{d,\alpha},\m,\mathbf 1)\) is a \(d\)-CC system.
In particular, \((H_{d,\alpha},Y(-,x),\mathbf 1)\) is a conformal \(d\)-vertex algebra.
\end{thm}

\appendix

\section{Determinant formula and singular vectors}\label{sec_singular}
In this appendix we prove Theorem \ref{prop_singular_vector}.
The reducibility and singular vectors of the scalar parabolic Verma modules considered here are known from the general theory.
After complexification to $\mathfrak{so}(d+2)_\C$, they are described explicitly in \cite[Proposition~4.14 and Corollary~4.15]{KOSS};
see also \cite[Section~3.1]{PTY} for the corresponding null primary descendants in physics terminology.
In particular, the two exceptional families of parameters occurring in Theorem~\ref{prop_singular_vector} agree with these classifications.

For our purposes, however, we need an explicit determinant formula for the invariant bilinear form, together with its
positivity properties. We therefore give a direct derivation below. The general determinant formula for parabolic Verma modules goes back to
Jantzen \cite{Jantzen77}; see also \cite[Section~2.2]{YamazakiDet} for its application to conformal field theory.

%We believe that this result has already been shown, but the positivity of the bilinear form and the explicit presentation of the singular vectors are important for the purpose of this paper that a complete proof is given here.
All notation in this appendix follows Section \ref{sec_Verma}. Let $d \geq 2$ and $N \geq 0$ be integers. For $\al\in\R$, let $V(\al)^\da$ be the parabolic Verma modules introduced in Section \ref{sec_Verma}, equipped with the invariant bilinear form $(-,-)$ of Proposition \ref{prop_bilinear_Verma}.
%
%この章では Proposition \ref{prop_singular_vector}を証明する。この章の記号は全て Section \ref{sec_Verma}に従うものとする。
%Let $d \geq 2$ and $N \geq 0$ integers, $V(\al)$ を parabolic Verma module for $\al\in\R$, $(-,-)$を不変双線形形式 (Proposition \ref{prop_bilinear_Verma})とする。
Let $\R[P_1,\dots,P_d]$ be the polynomial ring with variables $P_1,\dots,P_d$.
Denote by $V_{N,d}$ the degree $N$ subspace of $\R[P_1,\dots,P_d]$,
which has a basis of the form:
\begin{align}
\left\{P_{\mu_1}P_{\mu_2}\cdots P_{\mu_N}\right\}\quad \quad\quad \text{with } 1 \leq \mu_1 \leq \mu_2 \leq \dots \mu_N \leq d.  \label{eq_app_basis_Verma}
\end{align}
The following lemma is clear:
\begin{lem}\label{lem_app_dim}
$A_{N} = \dim V_{N,d} = \binom{d-1+N}{d-1} = \frac{(d-1+N)(d-2+N)\cdots (1+N)}{(d-1)!}$.
\end{lem}
For each $\al \in \R$, let $g_\al: V_{N,d} \rightarrow V(\al)_N^\da$ be the $\R$-linear isomorphism given by $g_\al(P_{\mu_1}P_{\mu_2}\cdots P_{\mu_N}) = P _{\mu_1}P_{\mu_2}\cdots P_{\mu_N}v_\al^\da$.
By $g_\al$ pulling back the invariant bilinear form on $V(\al)^\da$, a family of bilinear forms $(-,-)_\al$ is defined on $V_{N,d}$.
%
%
%各$\al \in \R$に対して、let $g_\al: V_{N,d} \rightarrow V(\al)_N$ be the $\R$-linear isomorphism given by $g_\al(P_{\mu_1}P_{\mu_2}\cdots P_{\mu_N}) = P_{\mu_1}P_{\mu_2}\cdots P_{\mu_N}v_\al$.
%$g_\al$によって$V(\al)^\da$上の不変内積を引き戻すことで、$V_{N,d}$上に内積 $(-,-)_\al$が定まる。
The following recursive relation is fundamental to the study of the bilinear forms:
%双線形形式$(-,-)$を調べるうえで次の帰納的な関係が基本的である:
\begin{lem}\label{lem_app_inductive}
For any $N \geq 1$, $\al\in \R$ and $\mu_i,\nu_j \in \{1,\dots,d\}$,
\begin{align*}
&(P_{\mu_1}P_{\mu_2}\cdots P_{\mu_N}, P_{\nu_1}P_{\nu_2}\cdots P_{\nu_N})_\al \\
&=2(\al+N-1)\sum_{j=1}^N \delta_{\mu_1,\nu_j} (P_{\mu_2}\cdots P_{\mu_N}, P_{\nu_1}\cdots\hat{P_{\nu_j}} \cdots P_{\nu_N})_\al\\
&-2 \sum_{k<l} \delta_{\nu_k,\nu_l} (P_{\mu_2}\cdots P_{\mu_N}, P_{\mu_1} P_{\nu_1}\cdots\hat{P_{\nu_k}} \cdots \hat{P_{\nu_l}} \cdots P_{\nu_N})_\al.
\end{align*}
\end{lem}
\begin{proof}
By Proposition \ref{prop_bilinear_Verma}, we have:
\begin{align*}
&(P_{\mu_1}P_{\mu_2}\cdots P_{\mu_N}v_\al^\da, P_{\nu_1}P_{\nu_2}\cdots P_{\nu_N}v_\al^\da)\\
&=(P_{\mu_2}\cdots P_{\mu_N}v_\al^\da, K_{\mu_1}P_{\nu_1}P_{\nu_2}\cdots P_{\nu_N}v_\al^\da)\\
&=\sum_{j=1}^N (P_{\mu_2}\cdots P_{\mu_N}v_\al^\da, P_{\nu_1}\cdots (2D \delta_{\mu_1,\nu_j} -J_{\mu_1,\nu_j}) \cdots P_{\nu_N}v_\al^\da)\\
&=2(\al+N-1)\sum_{j=1}^N \delta_{\mu_1,\nu_j} (P_{\mu_2}\cdots P_{\mu_N}v_\al^\da, P_{\nu_1}\cdots\hat{P_{\nu_j}} \cdots P_{\nu_N}v_\al^\da)\\
&-2 \sum_{k<l} \delta_{\nu_k,\nu_l} (P_{\mu_2}\cdots P_{\mu_N}v_\al^\da, P_{\mu_1} P_{\nu_1}\cdots\hat{P_{\nu_k}} \cdots \hat{P_{\nu_l}} \cdots P_{\nu_N}v_\al^\da).
\end{align*}
\end{proof}
From Lemma \ref{lem_app_inductive} we immediately see that
the dependence of the bilinear forms $(-,-)_\al$ on $V_{N,d}$ with respect to $\al$ is a polynomial of degree at most $N$.
% the bilinear form $(-,-)_\al$ on $V_{N,d}$ depends on the polynomial of degree at most $N$ with respect to $\al$. 
Therefore, by thinking of $\al$ as a variable, we get the polynomial-valued bilinear form 
\begin{align*} 
(-,-)_N:V_{N,d} \otimes V_{N,d} \rightarrow \R[\al] 
\end{align*} The purpose of this section is to examine $(-,-)_N$.
%Lemma \ref{lem_app_inductive}から直ちに分かることは$V_{N,d}$上の bilinear form $(-,-)_\al$ は、$\al$について高々$N$次の多項式で依存している点である。よって$\al$を変数と思うことで多項式値の bilinear form
%\begin{align*}
%(-,-)_N:V_{N,d} \otimes V_{N,d} \rightarrow \R[\al]
%\end{align*}
%を得る。この章の目的は$(-,-)_N$を調べることである。
The following lemma follows easily from Lemma \ref{lem_app_inductive}:
\begin{lem}\label{lem_app_leading}
The image of $(-,-)_N$ is a polynomial of degree at most $N$ and 
\begin{align*}
(P_{\mu_1}^{k_1}P_{\mu_2}^{k_2}\cdots P_{\mu_p}^{k_p}, P_{\nu_1}^{l_1}P_{\nu_2}^{l_2}\cdots P_{\nu_q}^{l_q})_N = 2^N \al(\al+1)\dots (\al+N-1) \delta_{p,q} \Pi_{i=1}^p \left((k_i!)\delta_{k_i,l_i}\delta_{\mu_i,\nu_i}\right)+\mathrm O(\al^{N-1})
\end{align*}
holds for any $1 \leq p,q \leq d$, $0<k_i, l_j $, $1 \leq \mu_1 < \mu_2 < \cdots < \mu_p\leq d$ and $1 \leq \nu_1 < \nu_2 < \cdots < \nu_q\leq d$
with $N= \sum_{i=1}^p k_i = \sum_{j=1}^q l_j$.
\end{lem}

\begin{dfn}\label{def_determinant}
Consider a basis of $V_{N,d}$ of the following form:
%$V_{N,d}$の基底として次の形のものを考える:
\begin{align}
\frac{1}{\sqrt{2}^N \sqrt{k_1!k_2!\cdots k_p!}} P_{\mu_1}^{k_1}P_{\mu_2}^{k_2}\cdots P_{\mu_p}^{k_p}
\label{eq_app_basis_normal}
\end{align}
for any $1 \leq p \leq d$, $0<k_i$, $1 \leq \mu_1 < \mu_2 < \cdots < \mu_p\leq d$ with $N= \sum_{i=1}^p k_i$.
Let $D_{N,d}(\al) \in \R[\al]$ be the determinant of the matrix representation of $(-,-)_N$ with respect to the basis of the form \eqref{eq_app_basis_normal}.
\end{dfn}

\begin{cor}\label{cor_app_det}
The degree of the polynomial $D_{N,d}(\al) \in \R[\al]$ is $N A_{N}$ and the coefficient of $\al^{NA_{N}}$ is $1$.
\end{cor}
Immediately from the above proposition, we see that $V(\al)^\da$ is irreducible except for countably many $\al$.
Below, think of $\al$ as a formal variable (an element of the polynomial ring $\R[\al]$).

Let $\al_0\in \R$ and recall that $N(\al_0)^\da \subset V(\al_0)^\da$ is the proper maximal submodule and $L(\al_0)^\da$ is the maximal quotient of $V(\al_0)^\da$, that is, $L(\al_0)^\da = V(\al_0)^\da/N(\al_0)^\da$.
Set
\begin{align*}
N(\al_0)_N^\da = V(\al_0)_N^\da\cap N(\al_0)^\da.
\end{align*}
It is important to note that if $v \in V_{N,d}$ satisfies $g_{\al_0}(v) \in N(\al_0)^\da$, then
\begin{align}
(v, w)_{\al_0} =0 \quad\quad\text{ for any }w \in V_{N,d}.\label{eq_N_1_null}
\end{align}
In other words, the polynomial $D_{N,d}(\al)$ has zero of multiplicity at least $\dim N(\al_0)_N^\da$ at $\al=\al_0$.
Therefore, if we can explicitly find real numbers $\al_i \in\R$ such that $\sum_i \dim N(\al_i)_N^\da$ is equal to the degree of $D_{N,d}(\al)$ in Corollary \ref{cor_app_det}, we can determine $D_{N,d}(\al)$.
%言い換えると、多項式$D_{N,d}(\al)$は$\al=\al_0$において少なくとも重複度$\dim N(\al_0)_N^\da$のゼロ点を持つ。
%よって Corollary \ref{cor_app_det}の次数分だけ $N(\al_0)_N^\da$がゼロでないような実数$\al_0\in \R$を明示的に見つけることができれば、我々は$D_{N,d}(\al)$を決定することができる。

%以下、$\al_0 =0,-1,-2,\dots,$の場合と $\al_0 = \dn, \dn-1,\dn-2,\dots $の場合を調べる。
Hereafter, we examine the cases $\al_0 =0,-1,-2,\dots,$ and $\al_0 = \dn,\dn-1,\dn-2,\dots,$.
First, when $\al_0= -K \in \Z_{\leq 0}$, the $\sod$-module homomorphism $\Psi_{-K}^\da: V(-K)^\da \rightarrow \R[x_1,\dots,x_d]$ satisfies
\begin{align*} 
\Psi_{-K}^\da(v_{-K}^\da)=((x_1)^2+\dots+(x_d)^2)^K
\end{align*}
Hence, the image of $V(-K)^\da$ is spanned by the derivatives of the polynomial
\begin{align} 
(\pa_{\mu_1} \dots \pa_{\mu_L}) ((x_1)^2+\dots+(x_d)^2)^K,\label{eq_image_K} 
\end{align}
which is a finite dimensional representation of $\sod$. Since $V(-K)_0^\da$ (the space of lowest weight vectors) is 1-dimensional, the image is an irreducible representation. Thus, $\ker \Psi_{-K}^\da = N(-K)^\da \subset V(-K)^\da$ is given as
follows:
%まず $\al_0= -K \in \Z_{\leq 0}$のとき、$\sod$-module homomorphism $\Psi_{-K}^\da: V(-K)^\da \rightarrow \R[x^1,\dots,x^d]$は
%\begin{align*}
%\Psi_{-K}^\da(v_{-K}^\da)=((x^1)^2+\dots+(x^d)^2)^K
%\end{align*}
%を満たす。$V(-K)^\da$の像は多項式の微分
%\begin{align}
%(\pa_{\mu_1} \dots \pa_{\mu_L}) ((x^1)^2+\dots+(x^d)^2)^K\label{eq_image_K}
%\end{align}
%で張られる。これは$\sod$の有限次元表現であり、$V(-K)_0^\da$ (the space of highest weight vectors)が1次元であることから像は既約加群である。よって$\ker \Psi_{-K}^\da = N(-K)^\da \subset V(-K)^\da$は次のように与えられる:
\begin{lem}\label{lem_null_minus}
For any $n \geq 0$, 
\begin{align*}
\dim N(-K)_n^\da=
\begin{cases}
0 & n \leq K\\
A_n-A_{2K-n} & n > K,
\end{cases} 
\end{align*}
where we set $A_p=0$ for any $p<0$.
\end{lem}
\begin{proof}
Denote by $\R[x_1,\dots,x_d]_l$ the degree $l$ subspace of $\R[x_1,\dots,x_d]$.
Then, by Poincare-Birkhoff-Witt theorem, $V(-K)_n^\da$ can be identified with $\R[x_1,\dots,x_d]_n$ and the image of $\Psi_{-K}^\da |_{V(-K)_n^\da}$ is in $\R[x_1,\dots,x_d]_{2K-n}$ by \eqref{eq_image_K}. Therefore,
if one can prove
\begin{align*} 
\Psi_{-K}^\da |_{V(-K)_n^\da}: \R[x_1,\dots,x_d]_{n}\rightarrow \R[x_1,\dots,x_d]_{2K-n}
\end{align*} 
is injective for $n \leq K$ and surjective for $n >K$, the assertion follows from Lemma \ref{lem_app_dim}.
This is easily seen from the theory of highest weight representations of (semi)simple Lie algebra $\sod$ (see \cite[Section 21.4]{Hum}).
%
%
%
%Then, by Poincare-Birkhoff-Witt theorem, $V(-K)_n$は$\R[x^1,\dots,x^d]_n$と同一視でき、$\Psi_{-K}^\da |_{V(-K)_n^\da}$の像は、\eqref{eq_image_K}より $\R[x^1,\dots,x^d]_{2K-n}$に入る。よって
%\begin{align*}
%\Psi_{-K}^\da |_{V(-K)_n^\da}: \R[x^1,\dots,x^d]_{n}\rightarrow \R[x^1,\dots,x^d]_{2K-n}
%\end{align*}
%が$n \leq K$で単射であり$n >K$で全射であることを言えば、Lemma \ref{lem_app_dim}より命題が従う。
%これは (semi)simple Lie algebra $\sod$ の最高ウェイト表現の理論から簡単に分かる (see \cite[Section 21.4]{Hum})。
%%
%\begin{itemize}
%\item
%$\sod$の有限次元表現の highest weight theory から簡単に分かるはず (?)
%\item
%$d_{-K}$によって$1$はwt $K$, $||x||^K$は wt $-K$の h.w.v (l.w.v).
%\item
%単射性は 有限次元なので singular vector が 真ん中の隣 ($K+1$)にあることが分かるはず。
%\item
%単射であることから像の加群を$M$とおくと$\dim M_{n} = \dim (\R[x^1,\dots,x^d]_{n})$が$n \leq K$で分かる。
%\item
%折り返しの対称性から $\dim M_{2K-n} = \dim M_n$であり、よって全射である。
%\end{itemize}
\end{proof}

Next consider the case $\al_0 = \dn,\dn-1,\dn-2,\dots$.
Set 
\begin{align*}
\D_P = \sum_{\rho=1}^d P_\rho^2.
\end{align*}
In Section \ref{sec_Verma}, we have already shown that
\begin{align*}
\D_P v_{\dn}^\da \in N\left(\dn\right)_2^\da
\end{align*}
and $\dim N(\dn)_2 \geq 1$ (see \eqref{eq_laplace_beta}).
More generally, we have:
\begin{lem}\label{lem_sing_DP}
For any $m \geq 0$, $\D_P^{m+1} v_{\dn-m}^\da \in N\left(\dn-m\right)^\da$.
In particular, $\dim N\left(\dn-m\right)_{l}^\da \geq A_{l-2m-2}$ for any $l \geq 2m+2$.
\end{lem}
\begin{proof}
Recall that there is the unique $\sod$-module homomorphism $\Psi^\da: V(\dn-m)^\da \rightarrow \R[x_1,\dots,x_d,|x|^\R]$ which sends $v_{\dn-m}^\da$ to $||x||^{-d+2+2m}$ (see \eqref{eq_psi_al}).
By \eqref{eq_laplace_beta}, 
\begin{align*}
\Psi(\D_P^{m+1} v_{\dn-m}^\da) = \left(\sum_\rho \pa_\rho^2\right)^{m+1} ||x||^{- d+2+2m}=0.
\end{align*}
Hence, the assertion holds.
\end{proof}
Note that if $d$ is even and we take a sufficiently large $m$, we get $\dn -m \in \Z_{\leq 0}$ and the singular vectors of Lemma \ref{lem_sing_DP} are contained in the first examined singular vectors $N(-K)^\da \subset V(- K)^\da$. We  will see later that they actually contribute to $D_{N,d}(\al)$ with multiplicity two.
First, we will prove Theorem \ref{prop_singular_vector} for the simpler case where $d$ is odd.
%
%さて$d$が偶数の場合、十分に大きな$m$を取ると、$\dn -m \in \Z_{\leq 0}$となって Lemma \ref{lem_sing_DP}の singular vector は 最初に調べた singular vector たち $N(-K)^\da \subset V(-K)^\da$に含まれてしまう。しかし実はこれらは重複度2で寄与していることを後で見る。
%まず最初により簡単な$d$が奇数の場合に命題\ref{prop_singular_vector}を証明する。
\begin{lem}\label{prop_determinant_odd}
Let $d \geq 3$ be an odd integer.
Then, for any $N \geq 0$,
\begin{align}
D_{N,d}(\al) =  \left(\prod_{k=0}^{N-1} (\al+k)^{A_N- A_{2k-N}}\right)
 \left(\prod_{m=0}^{[N/2]-1} (\al - \dn+m)^{A_{N-2m-2}}\right), \label{eq_det_formula_odd}
\end{align}
where $[N/2]$ is the largest integer not exceeding $N/2$.
Moreover, for any $m \in \Z_{\geq 0}$,
\begin{align*}
N\left( \dn - m \right)_{k}^\da =
\begin{cases}
0 & (k < 2m+2),\\
\R[P_1,\dots,P_d]_{k-2m-2} \Delta_P^{m+1} v_{\dn-m}^\da & (k \geq 2m+2).
\end{cases}
\end{align*}
\end{lem}
\begin{proof}
By Lemma \ref{lem_null_minus}, for any $0 \leq K < N$, the degree of zeros in $\al = -K$ of the polynomial $D_{N,d}(\al)$ is at least $A_N- A_{2K-N}$. Similarly, from Lemma \ref{lem_sing_DP}, the degree of zeros in $\al = \dn - m$ is at least $A_{N-2m-2}$.
%By Lemma \ref{lem_null_minus}, for any $0 \leq k < N$, 多項式$D_N(\al)$ の $\al=k$におけるゼロの重複度は少なくとも $A_N- A_{2K-N}$. 同様に Lemma \ref{lem_sing_DP}より、$\al = \dn - m$ におけるゼロの重複度は少なくとも$A_{N-2m-2}$.
Hence, $D_{N,d}(\al)$ is divisible by the right-hand side of \eqref{eq_det_formula_odd}.
Since
\begin{align*}
NA_N - \sum_{k=0}^{N-1} (A_N-A_{2k-N}) -\sum_{m=0}^{[N/2]} A_{N-2m-2}=\sum_{k=0}^{N-1} A_{2k-N} -\sum_{m=0}^{[N/2]} A_{N-2m-2}=0,
\end{align*}
the assertion follows from Corollary \ref{cor_app_det}.
\end{proof}

Hereafter, we will show 
\begin{align}
D_{N,d}(\al) =  \left(\prod_{k=0}^{N-1} (\al+k)^{A_N- A_{2k-N}}\right)
 \left(\prod_{m=0}^{[N/2]-1} (\al - \dn+m)^{A_{N-2m-2}}\right). \label{eq_det_formula_even}
\end{align}
for any even integer $d\geq 2$ and $N\geq 0$. The problem is that if $d$ is even, then $\dn -m \in \Z_{\leq 0}$, and the determinant cannot be calculated by the simple dimension counting of null vectors. In this case, specific null vectors must be counted twice.
%
%
%問題は$d$が偶数の場合は$\dn -m \in \Z_{\leq 0}$となり、単純な null vector の次元の数え上げでは determinant が計算できない点にある。この場合は、特定の null vector を二重に数えないといけない。

More precisely, for any $m \in \Z$ with $\dn>m \geq 0$, by the proof of Lemma \ref{prop_determinant_odd}, $D_{N,d}(\al)$ is divisible by $(\al-\dn+m)^{A_{N-2m-2}}$.
Hence, in order to prove \eqref{eq_det_formula_even}, it suffices to show that
\begin{enumerate}
\item[AD1)]
If $0 \leq k \leq \frac{N-d}{2}$, $D_{N,d}(\al)$ is divisible by $(\al+k)^{A_{N}-A_{2k-N}+A_{N-2k-d}}$.
\item[AD2)]
If $\frac{N-d}{2} < k \leq N-1 $, $D_{N,d}(\al)$ is divisible by $(\al+k)^{A_{N}-A_{2k-N}}$.
\end{enumerate}
(AD2) follows from Lemma \ref{lem_null_minus}. We will show (AD1).

Let $k \in \Z_{\geq 0}$ with $k \leq \frac{N-d}{2}$.
Recall that $g_{-k}: V_{N,d} \rightarrow V(-k)_N^\da$ is a linear isomorphism.
Set
\begin{align*}
N_1^k &= g_{-k}^{-1}(N(-k)_N)\\
N_2^k &=\R[P_1,\dots,P_d]_{N-2k-d} \Delta_P^{k+\frac{d}{2}},
\end{align*}
which are subsets of $V_{N,d}$. Since $\left(\sum_\rho \pa_\rho^2\right)^{k+\frac{d}{2}} ||x||^{2k} =0$,
we have:
\begin{align}
N_2^k \subset N_1^k \subset V_{N,d}.\label{eq_N_inc}
\end{align}

\begin{lem}\label{lem_double_zero}
Let $d\geq 2$ be an even integer and $m \geq 0$ with $m \geq \dn$. Then, for any $w \in V_{N,d}$ and 
$1 \leq \mu_1 \leq \cdots \leq \mu_{N-2m-2}$
\begin{align}
(P_{\mu_1} \cdots P_{\mu_{N-2m-2}} \Delta_P^{m+1}, w)_N \in \R[\al]
\label{eq_double_zero}
\end{align}
is divisible by $(\al+m-\dn)^2$.
\end{lem}
\begin{proof}
%We may assume that
%%\begin{align*}
%%v = P_{\mu_1} \cdots P_{\mu_{N-2k-d}} \Delta_P^{k+\frac{d}{2}}
%%\end{align*}
%%and
%\begin{align*}
%w = P_{\nu_1}\cdots P_{\nu_N}
%\end{align*}
%with  $1 \leq \nu_1 \leq \cdots \leq \nu_N$ and fix it.
By the natural inclusion map $V_{N,d} \hookrightarrow V_{N,D}$ for $D \geq d$, \eqref{eq_double_zero} can be defined for any dimension $D \geq d$ for fixed $w \in V_{N,d}$.
Since the dependence of \eqref{eq_double_zero} on $D$ is only in the definition of the Laplacian $\D_P=\sum_{\rho=1}^D P_\rho^2$, we can regard \eqref{eq_double_zero} as a polynomial of $\al$ and $D$ for each $w$.
Hence, there exists a 2-variable polynomial $P_w(x,y) \in \R[x,y]$ such that $P_w(\al,D)$ coincides with \eqref{eq_double_zero} for any $D \geq d$.
%$P_D(\al)$の$D$に関しての依存性は Laplacian $\D_P=\sum_{\rho=1}^D P_\rho^2$の定義にのみ現れて、とくに$D$について高々$m$次の多項式である。
%%(The determinant \eqref{}は$d$の多項式にはならないことに注意せよ。)
%よってある2変数多項式 $P(x,y) \in \R[x,y]$があって、$P(\al,D)=P_D(\al)$を任意の$D \geq d$に対して満たすものが存在する。
Set
\begin{align*}
S_m = \{(\frac{l}{2}-m-1 ,l) \in \R^2 \mid l\in \Z, l \geq d\}.
\end{align*}
%which is a countable subset on the line 
Then, it is easy to show that the Zariski closure of $S_m$ is the line
\begin{align*}
\{(x,y)\in \R^2\mid x - y/2 + m +1 =0\}.
\end{align*}
By Lemma \ref{lem_sing_DP}, $P_w(\al,D)=0$ for any $(\al,D)\in S_m$.
Hence, $P_w(\al,D) \in \R[\al,D]$ is divisible by $\al - D/2 + m +1$.
Similarly, $P_w(\al,D)$ is divisible by $\al+ l$ for $0 \leq l \leq m$ by 
\begin{align*}
\left(\sum_{\rho=1}^D P_\rho^2 \right)^{m+1}||x||^{2l}=0
\end{align*}
for any $D$. Hence, $P_w(\al,D)$ is divisible by $(\al-D/2+m+1)(\al+m-\dn)$. Hence, the assertion holds.
\begin{comment}
\begin{align*}
[K_\mu^2,P_\nu^2] &=  \delta_{\mu,\nu} (8 P_\mu K_\mu D + 4D(2D+1)) + 
\end{align*}
\begin{align}
\begin{split}
[K_\mu,\D_P] &= \sum_\rho [K_\mu,P_\rho^2] 
=P_\mu 2(2D+1)^2 P_\rho J_{\mu,\rho}
\end{split}
\end{align}
\end{comment}
\end{proof}

Let $v_1,\dots,v_{A_{N}}$ be a basis of $V_{N,d}$ such that:
\begin{enumerate}
\item
$\{v_1,\dots,v_{\dim N_2^k}\}$ is a basis of $N_2^k$;
\item
$\{v_{\dim N_2^k+1},\dots, v_{\dim N_1^k}\}$ together with $\{v_1,\dots,v_{\dim N_2^k}\}$ is a basis of $N_1^k$.
\end{enumerate}
Then, by Lemma \ref{lem_double_zero} and \eqref{eq_N_1_null},
the determinant of the matrix representation of $(-,-)_N:V_{N,d} \otimes V_{N,d} \rightarrow \R[\al]$ with respect to the basis $\{v_1,\dots,v_{A_{N}}\}$ is divisible by $(\al+k)^{\dim N_1^k + \dim N_2^k}$.
Since the difference between this determinant and $D_{N,d}(\al)$ is just a non-zero scalar multiple, (AD1) holds. Hence, we have:
\begin{thm}\label{thm_determinant}
Let $d \geq 2$ be an integer. Then, for any $N \geq 0$,
\begin{align}
D_{N,d}(\al) =  \left(\prod_{k=0}^{N-1} (\al+k)^{A_N- A_{2k-N}}\right)
 \left(\prod_{m=0}^{[N/2]-1} (\al - \dn+m)^{A_{N-2m-2}}\right). \label{eq_det_formula}
\end{align}
Moreover, if $m\in \Z_{\geq 0}$ satisfies $\dn -m \notin \Z_{\leq 0}$, then
\begin{align*}
N\left( \dn - m \right)_{k}^\da =
\begin{cases}
0 & (k < 2m+2),\\
\R[P_1,\dots,P_d]_{k-2m-2} \Delta_P^{m+1} v_{\dn-m}^\da & (k \geq 2m+2).
\end{cases}
\end{align*}
\end{thm}

Now, we will prove Theorem \ref{prop_singular_vector}.
Let $\al_0 \in \R$. 
By Theorem \ref{thm_determinant},  $D_{N,d}(\al_0)=0$ for some $N \geq 0$ if and only if
\begin{align}
\al_0 \in \{0,-1,-2,\dots \} \cup \{\frac{d-2}{2}, \frac{d-2}{2}-1,\frac{d-2}{2}-2,\dots \}.\label{eq_app_spec}
\end{align}
Hence, the $V(\al_0)^\da$ is reducible if and only if $\al_0$ is in \eqref{eq_app_spec} by Lemma \ref{lem_irreducible_non_deg}.
%The generators of $N(\al_0)^\da$ for the reducible cases are given in Theorem \ref{thm_determinant}.
It remains to prove the (semi)-positivity for $\al \geq \dn$.

We first observe that by Lemma \ref{lem_app_leading} $(-,-)_N$ is positive-definite for sufficiently large $\al$.
Since the bilinear form $(-,-)_N$ changes continuously with respect to the parameter $\al \in \R$,
the positivity is preserved while $D_{N,d}(\al)\neq 0$.
Since $\dn$ is the largest value among \eqref{eq_app_spec}, $(-,-)_N$ is positive-definite (resp. positive-semidefinite ) for any $\al > \dn$ (resp. $\al =\dn$).
In the case of $\al=\dn$, since the induced bilinear form on $L\left(\dn\right)$ is non-degenerate,
it is positive-definite.

Conversely, assume that the bilinear form $(-,-)$ on $L(\al)^\da$ is positive-definite. Then, from \eqref{eq_form_deg1}, $\al \geq 0$.
If $\al=0$, then $L(0)^\da = \R$, the bilinear form is positive.
%逆に$L(\al)$上のbilinear form $(-,-)$が正定値とする。このとき \eqref{eq_form_deg1}より、$\al \geq 0$である。
Since
\begin{align*}
D_{2,d}(\al) = \al^{A_2}(\al+1)^{A_2-1} (\al-\dn),
\end{align*}
the determinant is negative for $\dn>\al>0$. Hence, $\al$ satisfies $\al=0$ or $\al \geq \dn$. Hence, the assertion holds.

\section{Spherical harmonics and Gegenbauer polynomials}\label{sec_polynomial}
In this section, we prove Proposition \ref{prop_commutator} in the case
\(\al=\dn\) with \(d\geq 3\), namely
\begin{align*}
[\phi^+(x),\phi^-(y)]= \exp\left(-\sum_\rho y_\rho \frac{d}{dx_\rho}\right) \frac{1}{||x||^{d-2}}.
\end{align*}
All notation used in this section follows Section \ref{sec_Verma} and
Section \ref{sec_construction}.

%この章では$\al =\dn$ with $d\geq 3$ の場合に命題\ref{prop_commutator}
%\begin{align*}
%[\phi^+(x),\phi^-(y)]= \exp\left(\sum_\rho y_\rho \frac{d}{dx_\rho}\right) \frac{1}{||x||^{d-2}}
%\end{align*}
%を示す。
%この章における記号は全てSection \ref{sec_Verma} and Section \ref{sec_construction}に従う。
Set $\mathrm{Harm}_{d} = \bigoplus_{n \geq 0} \mathrm{Harm}_{d,n}$, the space of harmonic polynomials
and $N_{d,n} = \dim \mathrm{Harm}_{d,n}$.
By Proposition \ref{prop_psi_dn}, we can identify $L(\dn)$ as $\mathrm{Harm}_{d} $, which induces the bilinear form
\begin{align*}
(-,-)_H:\mathrm{Harm}_{d} \times \mathrm{Harm}_{d} \rightarrow \R.
\end{align*}
Let $\{p_i^n(x)\}_{i=1,\dots, N_{d,n}}$ be an orthonormal basis of $\mathrm{Harm}_{d,n}$ with respect to $(-,-)_H$ and set 
\begin{align*}
G_n^d(x,y)=\sum_{i=1}^{N_{d,n}}p_i^n(x)p_i^n(y).
\end{align*}
Then, by \eqref{eq_commutator_psi} and Lemma \ref{transform_explicit}, we have
\begin{align*}
[\phi^+(x), \phi^-(y)]=\sum_{n=0}^\infty |x|^{-(d-2)-2n} G_n^d(x,y).
\end{align*}
Hence, to prove Proposition \ref{prop_commutator},
it suffices to show the following by \eqref{eq_app_Gegen2} and \eqref{eq_app_Gegen3}.
\begin{prop}\label{polynomial_sum}
For any $n \geq 0$,
\begin{align*}
G_n^d(x,y)= C_n^{(\frac{d-2}{2})}\left(\frac{(x,y)}{|x||y|}\right)|x|^n|y|^n,
\end{align*}
where
$C_n^{(\frac{d-2}{2})}(t)$ is the Gegenbauer polynomials \eqref{eq_app_Gegen1}.
\end{prop}

Before we prove Proposition \ref{polynomial_sum}, we need a preliminary result.
Let $\omega$ be the standard measure on $S^{d-1}$, which is $\mathrm{SO}(d)$ invariant.
We normalize $\om$ by $\int_{S^{d-1}}\om=1$.
The restriction of a function on $\R^d$ to $S^{d-1}$
defines a $\mathrm{SO}(d)$-module homomorphism $\mathrm{Harm}_{d} \rightarrow C^\infty(S^{d-1},\R)$.
%\begin{rem}
%In fact, this gives an isomorphism $\mathrm{Harm}_{d} \rightarrow \mathfrak{o}_d$ as $\mathrm{SO}(d)$-modules (see Definition \ref{def_frak_o}).
%\end{rem}
For $n \in \Z_{\geq 0}$, 
define an inner product $(-,-)_{S^{d-1}}$ on $\mathrm{Harm}_{d,n}$
by
$$
(f(x),g(x))_{S^{d-1}} = \int_{S^{d-1}}\om(\xi) f(\xi)g(\xi),
$$
which is $\mathrm{SO}(d)$-invariant.
Since $\mathrm{Harm}_{d,n}$ is irreducible as 
a $\mathrm{SO}(d)$-module, $\mathrm{SO}(d)$-invariant bilinear form is unique up to a scalar multiplication.
Since, by Proposition \ref{prop_bilinear_L}, the bilinear form $(-,-)_H: \mathrm{Harm}_{d,n}\times \mathrm{Harm}_{d,n} \rightarrow \R$ is also $\mathrm{SO}(d)$-invariant,
there is a real number $c_{d,n} \in \R$ such that 
$c_{d,n}(-,-)_H = (-,-)_{S^{d-1}}.$
Hereafter, we will prove the following lemma:
\begin{lem}\label{cdl}
For any $d \geq 3$ and $n \in \Z_{\geq 0}$,
$c_{d,n}= \frac{d-2}{2n+d-2}$.
\end{lem}

Then, Proposition \ref{polynomial_sum} follows from Lemma \ref{cdl} with the following well-known theorem, which is called the spherical harmonic addition theorem (see, for example, \cite[(1.2.8)]{DaiXu}):
\begin{thm}\label{classical_normal}
Let $\{q_i(x)\}_{i=1,\dots,N_{d,n}}$ be an orthonormal basis
of $\mathrm{Harm}_{d,n}$ with respect to $(-,-)_{S^{d-1}}$.
Then, 
$\sum_{i=1}^{N_{d,n}}q_i(x)q_i(y) =\frac{2n+d-2}{d-2} C_n^{(\dn)}\left(\frac{(x,y)}{|x||y|}\right)|x|^n |y|^n$.
\end{thm}

Recall that in \eqref{eq_L_sod} we introduce $\{L(i),\Ld(i)\}_{i =-1,0,1}$ which are vectors in $\sod_\C$ satisfying the $\mathrm{sl}_2$-commutator relations.
By setting $z=x_1+ix_2$ and $\z=x_1-ix_2$, we have
\begin{align*}
d_\dn(L(1))= -\dn z -z^2\pa_z - \sum_{i=3}^d (zx_i\pa_i - x_i^2 \pa_\z),\\
d_\dn(\Ld(1))= -\dn \z - \z^2\pa_\z -  \sum_{i=3}^d (\z x_i\pa_i - x_i^2 \pa_z).
\end{align*}
Let $(\Harm_d)_\C$ be the space of harmonic polynomials with complex coefficients
and define a bilinear form on $(\Harm_d)_\C$ by linearly extending 
the bilinear form on $\Harm_d$. We remark that we do not consider Hermitian inner products as usual.
%Since the Laplacian is $\C$-linear, it is easy to show that $d_\dn(L(1)^l)1$ and $d_\dn(L(1)^l)1$ are Harmonic polynomials.

Let $(y)_l$ be the Pochhammer symbol
defined by 
\[
(y)_l= \begin{cases}
   1  & (l=0) \\
   y(y+1)\cdot (y+l-1) & (l \geq 1).
\end{cases}
\]
Then, we have:
\begin{lem}
For any $l \in \Z_{\geq 0}$, $d_\dn(L(1)^l)1=(-1)^l (\dn+l-1)(\dn+l-2)\dots
\dn(x_1+ix_2)^l$ and $d_\dn(\Ld(1)^l)1= (-1)^l(\dn+l-1)(\dn+l-2)\dots
\dn (x_1 - ix_2)^l$ and 
$((x_1+ix_2)^l,(x_1-ix_2)^l)_H=\frac{l!}{(\dn)_l}$.
\end{lem}
\begin{proof}
We will show $d_\dn(L(1)^l)1=(-1)^l (\dn)_l (x_1+ix_2)^l$ by induction on $l$.
For $l=0$, the assertion is clear.
Assume that the induction hypothesis is true for $l$.
Then,
\begin{align*}
d_\dn(L(1))z^l 
&=(-\dn z - z^2\pa_z - \sum_{i=3}^d (zx_i\pa_i -x_i^2 \pa_\z))z^l \\
&=-(l+\dn)z^{l+1}.
\end{align*}
By Proposition \ref{prop_bilinear_Verma}, we have
\begin{align*}
(d_\dn(L(1))^l1, d_\dn(\Ld(1))^l1)_H
&=(d_\dn(\ft(K_1+iK_2))^l1, d_\dn(\Ld(1))^l 1)_H \\
&=(1, d_\dn(\ft(P_1+iP_2))^l d_\dn(\Ld(1))^l 1)_H \\
&=(1,d_\dn(\Ld(-1))^l d_\dn(\Ld(1))^l 1)_H.
\end{align*}
Since $\Ld(-1),\Ld(0),\Ld(1)$ satisfy the commutator relation of $\mathrm{sl}_2$ and $d_\dn(\Ld(-1))1=0$ and 
$d_\dn(\Ld(0))1= - \frac{d-2}{4}$,
we have
\begin{align*}
d_\dn(\Ld(-1))^l d_\dn(\Ld(1))^l 1
&= l! (\dn+l-1)\dots (\dn +1) \dn 1.
\end{align*}
Thus, the assertion holds.
\end{proof}

\begin{lem}
For any $l \in \Z_{\geq 0}$,
$((x_1+ix_2)^l,(x_1-ix_2)^l)_{S^{d-1}}= \frac{l!}{(\frac{d}{2})_l} $.
\end{lem}
\begin{proof}
Let us consider the spherical coordinate:
\begin{align*}
x_1&=\sin \theta_1 \dots \sin \theta_{d-2} \sin \theta_{d-1}\\
x_2&=\sin \theta_1 \dots \sin \theta_{d-2} \cos \theta_{d-1}\\
x_3&=\sin \theta_1 \dots \cos \theta_{d-2}\\
\cdots \\
x_{d-1}&= \sin \theta_1 \cos \theta_2 \\
x_{d}&=\cos \theta_1,
\end{align*}
where $\theta_i \in [0,\pi]$ for $i=1,\dots,d-2$ and $\theta_{d-1} \in [0,2\pi]$.
Then, the volume form $\om$ normalized by $\int_{S^{d-1}}  \om =1$
is $$\om =\frac{1}{\Omega_{d-1}} \sin^{d-2} \theta_1\sin^{d-3} \theta_2 \cdots 
\sin \theta_{d-2} d\theta_1\cdots d\theta_{d-1},$$
where $\Omega_{d-1}= \frac{2\pi^{d/2}}{\Gamma(d/2)}$
and $\Gamma$ is the gamma function.
Then,
since $x_1+ix_2= \sin \theta_1 \dots \sin \theta_{d-2} \exp(i \theta_{d-1})$,
\begin{align*}
\int_{S^{d-1}} (x_1+ix_2)^l (x_1-ix_2)^l\om
&=\frac{1}{\Omega_{d-1}} \int 
\sin^{2l+d-2} \theta_1\sin^{2l+d-3} \theta_2 \cdots 
\sin^{2l+1} \theta_{d-2} d\theta_1 \cdots d\theta_{d-1}\\
&=\frac{2\pi}{\Omega_{d-1}} \Pi_{i=1}^{d-2}
\int_0^\pi \sin^{2l + i} \theta d\theta =\frac{l!}{(\frac{d}{2})_l}.
\end{align*}

\end{proof}
Combining the above lemmas,
we have
$$
c_{d,l}=\frac{l!}{(\frac{d}{2})_l} / \frac{l!}{(\dn)_l}
= \frac{d-2}{d-2+2l}.
$$
Hence, we have
$$
G_{d,l}(x,y)=C_l^{(\dn)}\left(\frac{(x,y)}{|x||y|}\right)|x|^l |y|^l,
$$
and Proposition \ref{polynomial_sum} holds.
%$C_l^{(\dn)}(1)= \binom{-(d-2)}{k}(-1)^k=\binom{d-3+l}{l}$.
%$q_iq_i(1)=N_{d,l}=\frac{2l+d-2}{l}\binom{d-3+l}{l-1}$.

\section{Massless scalar  representation}
\label{sec_massless}
Assume $d \geq 3$ and set
\begin{align*}
\Delta=\dn.
\end{align*}
In this section, we show that, by complexifying the action of $\sod$ on $L(\Delta)^\da$, one obtains a unitary representation of an appropriate covering group of $\mathrm{SO}_e(d,2)$, the conformal group of Minkowski space.
We also show that the restriction of this representation to the Poincar\'e group is the positive-energy massless scalar representation. These facts imply that the Hilbert space of states of the CFT on Minkowski spacetime naturally agrees with the completion of the space of states of the CFT on Euclidean space.
Such an identification is known in $d=2$ from studies of the correspondence between unitary full vertex operator algebras and Wightman fields; see \cite{AGT,AMT}.

We note that the representation-theoretic content of this section is already known from \cite{AngelopoulosLaoues}, and contains no new results. However, since we are not aware of a reference that formulates these facts from the viewpoint of the relation between Euclidean CFT and Minkowski CFT, we include the discussion here.

%
%この章では、$\al=\dn$における$L(\dn)$への$\sod$の作用を複素化することで、ミンコフスキー空間上の共形群である$\mathrm{SO}_e(d,2)$の適切な covering group のユニタリ表現が得られることを見る。
%またその表現の Poincare group への制限が massless scalar field の一粒子状態のなす空間であることを示す。これらの事実は Minkowski時空上の CFT の状態のなすヒルベルト空間と、Euclidian 時空上の CFT の状態のなす空間(の完備化)が自然に一致していることを意味する。
%こうした一致は、$d=2$において (full) vertex operator algebra と Wightman field の間の対応の研究において知られており \cite{AGT,AMT}、この章の内容はその高次元の自由場への拡張である。
%この章の表現論的な内容は\cite{FSS,Angel}により\footnote{
%$\mathrm{SO}_e(d,2)$の minimal representation
%は様々な視点で調べられてきた。
%物理的な視点からは Poincare group の massless scalar 表現の拡張として、\cite{SV,SS,HL}によって AQFT や共形場理論と関連して調べられた。また Klein-Gordon equation の解空間にユニタリ表現を作るという視点からの初期の仕事には \cite{BZ} などがある。また\todo \cite{KO1,KO2} 以下では我々は全ての次元を一様に扱った\cite{HSS}に基づき記述する}、新しい内容を含まないが、Euclidian CFT と Minkowski CFTの関係性において、記述する文献は存在しなかったため、ここに述べる。
%Set
%\begin{align*}
%\Delta=\dn.
%\end{align*}

The complexification $L(\Delta)_\C^\da=L(\Delta)^\da\otimes_\R\C$ is an irreducible module over the complexification $\sod_\C$ of $\sod$.
Moreover, for any $A+iB \in \sod_\C$,
\begin{align*}
((A+iB)u,v)_L = -(u,(\theta(A)-i\theta(B))v)_L\qquad\text{ for }u,v\in L(\Delta)_\C^\da,
\end{align*}
where the positive-definite invariant bilinear form on \(L(\Delta)^\dagger\) is extended sesquilinearly to a Hermitian inner product on \(L(\Delta)_{\mathbb C}^\dagger\), which we denote by \((-, -)_L\).
%where we extend the positive-definite invariant bilinear form on $L(\Delta)^\da$ を Hermitian sesquilinear に拡張し、その form を$(-,-)_L$と書いた。
We also extend the automorphism $\theta$ of $\sod$ anti-linearly to $\sod_\C$, and set
\begin{align*}
\sod_\C^\theta = \{A \in \sod_\C \mid \theta(A)=A\}.
\end{align*}
Then $\sod_\C^\theta$ is a real Lie subalgebra of $\sod_\C$ isomorphic to $\mathfrak{so}(d,2)$, with basis
\begin{align}
J_{\mu,\nu}, iD, K_\mu-P_\mu, i(K_\mu+P_\mu), \qquad (\mu\neq \nu).
\label{app_basis_so_theta}
\end{align}
It follows immediately that, for any $X \in \sod_\C^\theta$,
\begin{align}
(Xu,v)_L+(u,Xv)_L=0
\label{eq_X_cov}
\end{align}
for any $u,v \in L(\Delta)_\C^\da$.

\begin{rem}
\label{rem_explicit_sod_d2}
The real Lie algebra isomorphism
\begin{align}
F^\theta:\sod_\C^\theta \rightarrow \so(d,2)
\label{eq_F_theta}
\end{align}
can be described explicitly as follows.
Let
$\eta=\operatorname{diag}(1,1,-1,\ldots,-1)$ be the quadratic form on $\R^{d+2}$ with standard basis
$\ep_{-1},\ep_0,\ep_1,\ldots,\ep_d$.
Following the convention of \cite[(2.2)]{AngelopoulosLaoues}, for
$A,B\in\{-1,0,1,\ldots,d\}$, let $X_{A,B}\in\so(d,2)$ be defined by
$(X_{A,B})_{CD}
=
\delta_{C,A}\eta_{B,D}
-
\delta_{C,B}\eta_{A,D}$.
Let $e_0,e_1,\ldots,e_{d+1}$ be the basis of $\R^{1,d+1}$
introduced in Section~\ref{sec_conf_inf}.
Define a $\C$-linear map
\[
T:\R^{1,d+1}\otimes_\R\C
\longrightarrow
\R^{2,d}\otimes_\R\C
\]
by
$T(e_0)=\ep_{-1}$, $T(e_\mu)=\ep_\mu$ $(1\leq\mu\leq d)$ and $T(e_{d+1})=i\ep_0$.
Then $T$ preserves the complex bilinear forms determined by
$\{-g_{\mu,\nu}\}$ and $\{\eta_{A,B}\}$, and thus, $T$ defines an isomorphism
$\sod_\C\rightarrow\so(d,2)_\C$ by $A\mapsto TAT^{-1}$.
The elements in \eqref{app_basis_so_theta} are mapped, respectively, to
\[
-X_{\mu,\nu},
-X_{-1,0},
-2X_{-1,\mu},
2X_{0,\mu}.
\]
%実リー代数の同型
%\begin{align}
%F^\theta:\sod_\C^\theta \rightarrow \so(d,2)
%\label{eq_F_theta}
%\end{align}
%は明示的には次のように与えられる.
%Let $\eta=\operatorname{diag}(1,1,-1,\ldots,-1)$
%be the quadratic form on $\R^{d+2}$ with standard basis
%$\ep_{-1},\ep_0,\ep_1,\ldots,\ep_d$, and let $X_{AB}\in\so(d,2)$, $A,B\in\{-1,0,1,\ldots,d\}$, be defined by
%$(X_{A,B})_{C D}
%=
%\delta_{C,A}\eta_{B,D}
%-
%\delta_{C,B}\eta_{A,D}$.
%また$\R^{1,d+1}$の基底を$e_0,e_1,\dots,e_{d+1}$とおく (see Section \ref{sec_conf_inf})。
%このとき$\C$-linear map 
%\begin{align*}
%T:\R^{1,d+1}\otimes_\R \C \rightarrow \R^{2,d}\otimes_\R\C
%\end{align*}
%を$T(e_0)=\ep_{-1}$, $T(e_\mu)=\ep_\mu$, $T(e_{d+1})=i\ep_0$によって定めると、$\{-g_{\mu,\nu}\}$, $\{\eta_{AB}\}$によって定まる bilinear form は保たれ、$A \mapsto TAT^{-1}$は、$\sod_\C$から$\so(d,2)_\C$への同型写像を与える。
%とくに\eqref{app_basis_so_theta}はそれぞれ
%\begin{align*}
%-X_{\mu,\nu}, -X_{-1,0}, -2X_{-1,\mu}, 2X_{0,\mu}
%\end{align*}
%に移る。
\end{rem}

Let $G_0=\mathrm{SO}_e(d,2)$, and let $K_0=\mathrm{SO}(d)\times \mathrm{SO}(2)$ be its maximal compact subgroup. Since $D.v_\Delta^\da=\frac{d-2}{2}v_\Delta^\da$, the action of $\exp(itD)$ factors through $\mathrm{SO}(2)$ when $d$ is even, and through the double cover of $\mathrm{SO}(2)$ when $d$ is odd. Hence, when $d$ is even, we set $G=\mathrm{SO}_e(d,2)$, while when $d$ is odd, we let $G$ be the connected real Lie group which is a double cover of $G_0$ obtained by covering the $\mathrm{SO}(2)$-factor. Its maximal compact subgroup is
\begin{align*}
K= \begin{cases}
\mathrm{SO}(d) \times \mathrm{SO}(2) & d\text{:even}, \\
\mathrm{SO}(d) \times \widetilde{\mathrm{SO}}^{(2)}(2) & d\text{:odd}.
\end{cases}
\end{align*}

The space $L(\Delta)_\C^\da$ is a direct sum of irreducible $\SO(d)$-modules, and $K$ acts unitarily on it. Moreover, this action is clearly compatible with the action of $\sod_\C \cong \mathfrak{so}(d,2)_\C$. Together with \eqref{eq_X_cov}, this gives the following proposition.
\begin{prop}\label{prop_sod2_unitary}
$L(\Delta)_\C^\da$ is a unitary Harish--Chandra $(\mathfrak{so}(d,2)_\C,K)$-module.
\end{prop}

Following \cite{AngelopoulosLaoues}, we show that this Harish--Chandra module arises from a unitary representation of $G$ whose restriction to the Poincar\'e group $P_d=\R^d\rtimes\mathrm{SO}_e(1,d-1)$ is the positive-energy massless scalar representation. We first recall this representation.
%以下では\cite{AngelopoulosLaoues}に従い、我々はこの Harish-Chandra module が、
%Poincar\'e group $P_d=\R^d \rtimes \mathrm{SO}_e(1,d-1)$および$G$ のある unitary表現に由来することを見る。
%We first recall the positive energy massless scalar representation of $P_d$. 
Set
\begin{align*}
N_+= \{(q_0,\dots,q_{d-1}) \in \R^{1,d-1} \mid q_0^2-\sum_{i=1}^{d-1} q_i^2=0,q_0>0 \},
\end{align*}
the forward light cone in the $d$-dimensional Minkowski space, and
\begin{align*}
d\mu(q) = \frac{1}{2q_0}d^{d-1}q
\end{align*}
be the measure on $N_+$.
Then, $L^2(N_+,d\mu)$ is an irreducible unitary representation of the Poincar\'e group $P_d$ by
\begin{align}
(\rho_N(a,\Lambda) f)(q)=e^{i\langle a, q \rangle_M}f(\Lambda^{-1}q)
\label{eq_app_momentum}
\end{align}
for $(a,\Lambda) \in \R^{d}\rtimes \mathrm{SO}_+(1,d-1)$, where $\langle a,q\rangle_M
=a_0q_0-\sum_{i=1}^{d-1}a_i q_i$.
In general, the irreducible unitary representations of the Poincar\'e group $P_d$ are classified, by Wigner's method, in terms of the $\mathrm{SO}_e(1,d-1)$-orbits in $\R^d$ and the irreducible unitary representations of their stabilizer subgroups, called {\it little groups}, and are constructed as induced representations \cite{Wigner,Mackey}.
The representation corresponding to the massless orbit $N_+=\mathrm{SO}_e(1,d-1).(1,1,0,\dots,0)$ and the trivial representation of the little group $\R^{d-2}\rtimes \mathrm{SO}(d-2)$ is called the {\it positive-energy massless scalar representation} and is unitarily equivalent to \eqref{eq_app_momentum}.

We recall that the Poincar\'e group $P_d$ is naturally embedded in $G$ as a standard subgroup; see \cite[Section 1.b]{AngelopoulosLaoues}.
By \cite[Theorem 2.5]{AngelopoulosLaoues}, the unitary representation of $P_d$ on $L^2(N_+,d\mu)$ extends uniquely to a unitary representation of $G$. We denote it by $H^+$.\footnote{For $d=4$, the existence of such extensions was established from the viewpoints of conformal field theory and algebraic quantum field theory in \cite{SV,SS,HL}. The representation $H^+$ has also been realized using solutions to the Klein--Gordon equation \cite{BZ,HSS} and studied from the viewpoint of conformal geometry and conformally invariant differential equations \cite{KO1,KO2}.}

This representation corresponds to the case $s=0$ in \cite{AngelopoulosLaoues}. By \cite[Theorem 2.4 and Proposition 2.6]{AngelopoulosLaoues}, the $K$-finite part $(H^+)_K$ is graded by $D$, with eigenvalues $\{\frac{d-2}{2}+k\}_{k=0,1,\ldots}$. Moreover, the lowest eigenspace, corresponding to the eigenvalue $\Delta=\frac{d-2}{2}$, is the trivial representation of $\mathrm{SO}(d)$. Hence, it satisfies \eqref{eq_dual_Verma}, and thus, by the universal property of the parabolic Verma module, $L(\Delta)_\C^\da$ and $ (H^+)_K$ are isomorphic as $\so(d,2)_\C$-modules. By irreducibility, the inner products on the two sides agree up to normalization.
%We recall that the Poincar\'e group $P_d$ is naturally embedded in $G$
%as a standard subgroup; see \cite[Section 1.b]{AngelopoulosLaoues}.
%さらに$P_d$の$L^2(N_+,d\mu)$へのユニタリ表現は、上記の埋め込みの下で$G$のユニタリ表現に一意的に拡張できることが知られている\cite[Theorem 2.5]{AngelopoulosLaoues}。この表現を$H^+$とおく
%\footnote{こうした拡張の存在は、共形場理論や
%AQFTの視点から$d=4$において\cite{SV,SS,HL}によって確立された. $H_+$には Klein--Gordon equation を用いた実現\cite{BZ,HSS}やその conformal geometry and conformally invariant differential equations の視点からの研究がある \cite{KO1,KO2}.}。
%この表現は AL における$s=0$のパラメータに対応しており、\cite[Theorem 2.4, Proposition 2.6]{}から$(H^+)_{K}$は$D$によって次数付けを持ち、その固有値は$\{\frac{d-2}{2}+k\}_{k=0,1,\dots}$である。また固有値$\frac{d-2}{2}$の最低ウェイト空間は$\mathrm{SO}(d)$の自明表現である。よって \eqref{eq_dual_Verma}を満たすから、普遍性によりスカラー倍をのぞき一意的な $\mathfrak{so}(d,2)_\C$-module 射$V(\Delta)_\C^\da \rightarrow (H_+)_K$が存在する。$(H_+)_K$は既約であるから、as $\so(d,2)_\C$-modules, we have $L(\Delta)_\C^\da \cong (H^+)_K$, and by irreducibility the inner products on the two sides agree up to normalization.
In particular,
\begin{align}
\overline{L(\Delta)_{\C}^\da}^{\,(-,-)_{L}} \cong H^+,
\end{align}
as Hilbert spaces. Hence, we have:
\begin{prop}\label{prop_HC_isom}
As a unitary Harish--Chandra module, $L(\Delta)_\C^\da$ is isomorphic to the $(\mathfrak{so}(d,2)_\C,K)$-module associated with the unitary representation $H^+$ of $G$.
\end{prop}

Finally, we note that the symmetric Fock space
$\Gamma_s(H^+)$
carries the Wightman field of the free massless scalar theory satisfying the Wightman axioms, and hence provides the state space of the corresponding quantum field theory on Minkowski spacetime; see, for example, \cite[Chapter 10]{Arai}.
%\footnote{In \cite{HL}, the one-particle Hilbert space is realized as \(L^2(N_+,d\mu)\). By Proposition~\ref{prop_Hplus_Poincare}, however, we have \(H^+\cong L^2(N_+,d\mu)\).}.
%最後に$H^+ \cong L^2(N_+,d\mu)$の symmetric Fock space $\Gamma_s(H^+)$上には Wightman axioms を満たす Wightman field が定義でき、Minkowski 空間上の場の量子論の状態空間であることを注意されたい (see for example \cite{HL}).
\begin{cor}
\label{cor_app_Fock}
The Hilbert completion of the conformal $d$-vertex algebra
$H_{d,\Delta}\otimes_\R\C$ with respect to its canonical positive-definite
inner product is naturally isomorphic to the symmetric Fock
space $\Gamma_s(H^+)$.
\end{cor}

\noindent
\begin{center}
{\bf Acknowledgements}
\end{center}

The notion of conformal $d$-vertex algebra was obtained during my doctoral studies.
I would like to express my sincere gratitude to my PhD advisor, Masahito
Yamazaki.  I am also grateful to Yoh Tanimoto and Maria Stella Adamo for
encouraging me to write up this work and valuable discussions.
This work is supported by Grant-in-Aid for Early-Career Scientists (24K16911).

\noindent
\begin{center}
{\bf Declaration of AI use}
\end{center}

Work on this manuscript began in June 2021. Most of the manuscript, except for the Introduction and Appendix C, was written by March 2025 without the use of AI.
ChatGPT (OpenAI) was used to find the results of Angelopoulos and Laoues on the extension of massless representations to the conformal group, used in Appendix C.
It was also used to help improve the proof of Lemma \ref{lem_wick_cluster}, to proofread the manuscript, and to translate text written by the author in Japanese into English.
The author takes full responsibility for the mathematical arguments and references in the manuscript.

\end{document}